%% file: main.tex
\documentclass[11pt]{article}

\usepackage[en-US]{datetime2}

\usepackage{graphicx} 
\usepackage{multirow}
\usepackage[normalem]{ulem}
\usepackage{amsmath,amssymb,amsfonts,amsthm}
\usepackage
[margin=1in]
{geometry}
\usepackage{xcolor}
\usepackage{cancel}
\usepackage{hyperref}
\usepackage{graphicx}
\usepackage{lineno}
\usepackage{todonotes}

\usepackage{mathtools}
\usepackage{eqnalign}
\mathtoolsset{showonlyrefs}

\input{macro}

\theoremstyle{definition}

\title{Ricci curvature
for fluid models on the torus 
via Zeitlin's quantization}

\author{Sadashige Ishida and Alex Suri}
\date{\today}

\begin{document}

\maketitle




\begin{abstract}

Ricci curvature measures the average stability of geodesics under transverse perturbations, but how it should be defined in infinite dimensions is often unclear. 
This paper proposes a definition of Ricci curvature on the space of Hamiltonian diffeomorphisms on the two-dimensional flat torus $\mathrm{HDiff}(\mathbb{T}^2)$,  the state space for ideal fluids. 
Our definition is based on Zeitlin's model, which approximates $\mathrm{HDiff}(\mathbb{T}^2)$ by finite-dimensional Lie groups $\mathrm{SU}(N)$. 
We derive a formula for the Ricci curvature tensor on $\mathrm{SU}(N)$ and provide numerical evidence for its convergence in the large-$N$ limit to our conjectured finite value. 
Additionally, we explore potential applications for hydrodynamics through the Lyapunov stability of gravest wave modes and Arnold's tradewind estimates for long-term weather predictability. Our framework extends to a wide range of settings. We demonstrate this by introducing Ricci curvature on the state spaces of fluids on rectangular domains, the Lagrangian averaged Euler equation induced by the $H^1$-Sobolev metric, and the quasi-geostrophic equation incorporating the Coriolis effect.
\medskip

\medskip

\end{abstract}


\tableofcontents

\input{contents}

\newpage
\bibliographystyle{alpha}
\bibliography{references}

\input{appendix}

\end{document}

%% file: macro.tex
\usepackage{amssymb}
\usepackage{amsthm}
\usepackage{mathrsfs}

\usepackage{enumitem}

\usepackage{mathtools}

\usepackage{booktabs}
\usepackage{pifont}

\usepackage{aliascnt}

\hypersetup{
    colorlinks=true,
    linkcolor=black,
    filecolor=black,      
    urlcolor=blue,
    citecolor=black
}

\newtheorem{theorem}{Theorem}[section]

\newcommand{\addtheorem}[2]{%
  \newaliascnt{#1}{theorem}%
  \newtheorem{#1}[#1]{#2}%
  \aliascntresetthe{#1}%
\expandafter\newcommand\csname #1autorefname\endcsname{#2}%

}

\addtheorem{corollary}{Corollary}
\addtheorem{lemma}{Lemma}
\addtheorem{proposition}{Proposition}
\addtheorem{conjecture}{Conjecture}

\theoremstyle{definition}
\addtheorem{definition}{Definition}
\addtheorem{assumption}{Assumption}
\addtheorem{example}{Example}
\addtheorem{remark}{Remark}
\addtheorem{question}{Question}

\numberwithin{equation}{section}

\usepackage{overpic}

\usepackage{mathabx}

\usepackage{nicematrix}

\usepackage{tikz-cd}

\usepackage{xcolor}

\usepackage{wrapfig}
\usepackage{graphicx} 

\usepackage{pb-diagram}

\usepackage[draft]{pdfcomment}

\usepackage[subrefformat=parens]{subcaption}
\definecolor{darkblue}{rgb}{0.0, 0.0, 0.55}

\definecolor{mildgreen}{rgb}{0.2,0.4,0.8}

\definecolor{mildblue}{rgb}{0.3,0.3,1.0}

\newcommand{\on}[1]{{\operatorname{#1}}}

\def\ie{\emph{i.e.}}
\def\eg{\emph{e.g.}}

\def\CC{\mathbb{C}}

\def\RR{\mathbb{R}}
\def\SS{\mathbb{S}}
\def\TT{\mathbb{T}}

\def\ZZ{\mathbb{Z}}

\def\bx{\mathbf{x}}

\usepackage{bbm}
\DeclareSymbolFont{bbold}{U}{bbold}{m}{n}
\DeclareSymbolFontAlphabet{\mathbbold}{bbold}

\DeclareMathOperator{\SU}{SU}

\DeclareMathOperator{\Ric}{Ric}

\DeclareMathOperator{\su}{\mathfrak{su}}
\DeclareMathOperator{\sll}{\mathfrak{sl}}

\DeclareMathOperator{\SDiff}{SDiff}
\DeclareMathOperator{\HDiff}{HDiff}

\DeclareMathOperator{\diff}{diff}

\DeclareMathOperator{\hdiff}{hdiff}%

\DeclareMathOperator{\ad}{ad}

\DeclareMathOperator{\Tr}{Tr}

\definecolor{darkgreen}{rgb}{0.2,0.5,0.2}
\definecolor{darkred}{rgb}{0.5,0.2,0.2}

\newcommand{\cmark}{{\color{darkgreen}\ding{51}}}
\newcommand{\xmark}{{\color{darkred}\ding{55}}}

%% file: contents.tex
\section{Introduction}

Arnold interpreted the motions of  ideal fluids as geodesics on the infinite-dimensional group of volume-preserving diffeomorphisms $\SDiff(M)$ on a fluid domain $M$, equipped with the $L^2$-metric \cite{arnold1966geometrie}. He computed the sectional curvature of $\SDiff(\TT^2)$  on two-dimensional torus $\TT^2$ and showed it is negative in many directions. He hence suggested that the long-time weather forecast is unreliable, as the negativity of the sectional curvature indicates the instability of geodesics. 

There are, however, directions where the sectional curvature on $\SDiff(\TT^2)$ is positive. In fact Le Brigant and Preston \cite{le2024conjugate} showed the existence of conjugate points along the so-called Kolmogorov flow.  On the sphere $\SS^2$, on other hand, Suri \cite{suri2024conjugate} showed that almost every spherical harmonic flow has conjugate points along fluid trajectories, highlighting existence of many sections of positive curvature.

These observations motivate the need for an integrated measure that can capture the average infinitesimal tendency of nearby geodesics to converge or diverge,  quantifying the overall stability of fluid states. In finite dimensions, the Ricci curvature tensor performs this role by summing sectional curvatures over an orthonormal basis.  

In infinite dimensions, however, the notion of Ricci curvature is delicate and lacks a canonical definition. Lukatskii \cite{lukatskii1984curvature} proposed a definition of Ricci curvature on $\SDiff(\TT^2)$  as the formal (normalized) sum of sectional curvatures using sine and cosine waves as a basis. Cruzeiro and Malliavin \cite{cruzeiro2008nonergodicity} later proposed a definition, which is equivalent in structure to Lukatskii's, but using another set of basis elements. However, the values given by these two definitions appear to differ, even up to normalization. This indicates a dependence on the choice of basis, in contrast to the finite-dimensional case where the Ricci curvature is basis-independent.
We speculate that such issues arise because these formal infinite sums rely on specific mode expansions rather than the intrinsic geometric structures (\eg, Lie algebra structure) of the space. The present article seeks a definition of Ricci curvature that reflects the Lie algebra structure of the continuous setting.

Most recently,  Lichtenfelz, Preston and Modin \cite{lichtenfelz2025ricciZeitlin} proposed a definition of Ricci curvature on $\SDiff(\SS^2)$ via Zeitlin's discretization on the finite-dimensional matrix Lie group $\SU(N)$, which allowed them to  utilize the underlying Lie algebra structure of $\su(N)$. The present work shares the same spirit, while the base domain is $\TT^2$. 


\subsection{Main results}

We propose a definition of Ricci curvature on $\HDiff(\TT^2)$, the Hamiltonian diffeomorphism group  on the flat torus $\TT^2$. This space is sufficient for studying the stability of incompressible fluids as it is a totally geodesic submanifold of $\SDiff(\TT^2)$ with respect to $L^2$ (and also $H^s$) metrics.\footnote{This is specific to the flat torus, where the space of harmonic 1-forms is identified with the space of Killing vector fields. On a general multiply connected manifold $M$, $\HDiff(M)$ is not totally geodesic  \cite{haller2002totally}.} Hence it captures the essential dynamics of the Euler equation, as the omitted harmonic components of the velocity field do not evolve in time, while merely translating vorticity without affecting the fluid stability.  

To define a Ricci curvature on $\HDiff(\TT^2)$, we employ Zeitlin's discretization \cite{zeitlin1991torus}, which approximates the Euler equation as a geodesic equation on a finite-dimensional Lie group $\SU(N)$ equipped with a metric called Zeitlin's metric. Specifically, we define the Ricci curvature on $\HDiff(\TT^2)$ as the large-$N$ limit of the Ricci curvature on $\SU(N)$ induced by Zeitlin's metric. 

This is a structure-aware construction: Zeitlin's model is a structure-preserving discretization that replaces the Poisson bracket on the Lie algebra of $\HDiff(\TT^2)$ with the matrix commutator bracket on $\su(N)$, which converges to the Poisson bracket as $N\to \infty$. Thus, this model provides a discretization at the level of Lie algebra. In fact, discrete solutions of Zeitlin's model converge to the continuous Euler equation, as shown in \cite{gallagher2002mathematical}.

Within this setup, we derive a formula for the Ricci curvature tensor on $\SU(N)$, and, in particular, prove that it is diagonal with respect to a certain basis (\autoref{th:Ricci curvature formula}).  We also provide numerical evidence that these Ricci curvature values converge to our conjectured finite values as $N\to\infty$.

\paragraph{Choice of discrete Laplacian}

Zeitlin's metric on \(\SU(N)\) requires a notion of discrete Laplacian. In contrast to the case of \(\SDiff(\SS^2)\) studied in \cite{lichtenfelz2025ricciZeitlin}, where the Hoppe--Yau Laplacian \cite{hoppe-yau1998some} serves as a canonical choice, \(\HDiff(\TT^2)\) admits several natural candidates for a discrete Laplacian. We investigate how different choices of discrete Laplacians, and their induced metrics, alter the resulting Ricci curvature distributions.

\subsubsection{Applications for hydrodynamics}

To illustrate the utility of our framework in hydrodynamics, we explore its potential applications to stability analysis. First, we study the nonlinear stability of some solutions to the Euler equation in vorticity form. The single-mode waves $\omega(x,y)=\sin(nx)+\sin(ny)$ ($n\in\ZZ$), are stationary solutions, but only the gravest modes ($n=\pm1$) are Lyapunov stable \cite{Wir-Shep,Dullin2016instability}, while all higher modes are unstable \cite{Dullin2016instability}. We numerically observe that this distinction is reflected in the Ricci curvature values, suggesting that Ricci curvature may provide a geometric signature of nonlinear stability.

Second, we revisit Arnold's tradewind current \cite[Chapter IV, 4.B]{Arnol-khesin}: he devised an estimate of error growth for long-term weather predictability. In his framework, he used the sectional curvature of the lowest-mode waves, referring to this quantity as the ``mean curvature'', with the hope of capturing the curvature averaged over all directions in $\SDiff(\TT^2)$. As an alternative, we suggest replacing the sectional curvature with the Ricci curvature, and examine how the error growth estimate changes accordingly.

\subsubsection{Extensions to other settings}

A significant feature of our framework is that it extends to a wide range of settings. We demonstrate this by defining the Ricci curvature on the state spaces of several fluid models.

\paragraph{$H^1$-metrics} We first consider $H^1$-Sobolev metrics, which induce the Lagrangian Averaged Euler ($\mathrm{LAE}$-$\alpha$) equation  as the geodesic equation \cite{shkoller2000analysis}. We observe that the resulting Ricci curvature exhibits a smoothing effect on higher-frequency modes.

\paragraph{Rectangular domains}
Another natural extension is to the rectangular torus $\mathbb{T}_\alpha^2 = [0,2\pi/\alpha) \times [0,2\pi)$ ($\alpha>0$) \cite{drivas2022conjugate}. We show that the eigenvalues of the Laplacian undergo anisotropic scaling, which in turn deforms the coadjoint action and the Ricci curvature distribution anisotropically.

\paragraph{Coriolis force}
Finally, we extend our framework for the quasi-geostrophic equation, which incorporates the Coriolis force into fluid motions, causing rotational effects of the domain. In the continuous setting, this is achieved via the central extension of the Lie algebra of $\HDiff(\TT^2)$ by the so-called Roger cocycle \cite{vizman2008cocycles,suri2024curvature}.
We establish a discrete analogue of this formulation by constructing a discrete cocycle $\omega_N$ on $\su(N)$. We show that while the Coriolis force adds spectral corrections to the Ricci curvature on the resulting central extension of $\mathrm{SU}(N)$, this effect vanishes in the large-$N$ limit, recovering the curvature of the non-rotating case.

\subsubsection*{Future work}
There are several promising directions for future research.
First, further extending our framework to other physical models is an appealing avenue. A notable example is ideal magnetohydrodynamics using Zeitlin's model \cite{zeitlinMHD}. It would also be natural to define the Ricci curvature using other discretization schemes. An interesting candidate is the recently proposed model \cite{roy2026vakonomic}, which accommodates a broader class of fluid domains beyond $\mathbb{T}^2$ and $\mathbb{S}^2$ while successfully describing fluid motion as geodesics.

Second, a rigorous analysis of the convergence of the Ricci curvature in various settings remains an important open question.

Finally, it would be insightful to examine whether the contracted Bianchi identity holds (equivalently, whether the Einstein tensor is divergence-free). Since it is a fundamental property of Ricci curvature in finite dimensions, such a verification would serve as a natural benchmark for assessing which infinite-dimensional notions of Ricci curvature are geometrically most natural.

\paragraph{Organization of the article}
In \autoref{sec:preliminaries}, we review the preliminaries on Zeitlin's discretization. In \autoref{sec:Ricci on suN}, we define and compute the Ricci curvature. In \autoref{sec:applications}, we discuss potential applications to hydrodynamics. The remainder of the paper extends the Ricci curvature to $H^1$-metrics (\autoref{sec:Ricci Sobolev}), fluids on rectangular domains (\autoref{sec:Ricc_Rect}), and the Coriolis force (\autoref{sec:Coriolis-force}).

\paragraph{Concurrent work}
We acknowledge concurrent research by Lichtenfelz, Raad, and Valletta \cite{Lichtenfelz2026personal_communication}, of which we became aware very recently via personal communication with the authors. We confirm that the two works were carried out independently. 

While both their work and ours study notions of Ricci curvature on $\HDiff(\TT^2)$ using Zeitlin's discrete framework and contain some overlapping results, a large part of the content differs significantly. Their work places emphasis on conjugate points and a rigorous analysis for the convergence of the discrete Ricci curvature, with a specific choice of the Riemannian metric. Our work, on the other hand, highlights Ricci curvatures induced by different metrics, including extensions to higher-order metrics and the incorporation of the Coriolis effect.

We believe that both works present significant novelties and complement each other's contributions. 

\paragraph{Acknowledgement}
Sadashige Ishida acknowledges support from the European Research Council (ERC) under the Consolidator Grant No. 101045083 (CoDiNA) and from the JSPS Overseas Research Fellowships. Alex Suri acknowledges support from the Deutsche Forschungsgemeinschaft (DFG, German Research Foundation) under Grant No. 517512794.

\section{Preliminaries and settings}\label{sec:preliminaries}

In this section, we review preliminaries and lay out the settings for defining a Ricci curvature on $\SU(N)$. In particular, we spell out a basis and a Riemannian metric on $\SU(N)$, for which Zeitlin's quantization of the incompressible Euler equation arises as the geodesic equation on $\SU(N)$.


\subsection{Zeitlin basis on $\sll(N,\CC)$ and quantized Lie algebra}

We first review matrices called Zeitlin's basis on $\sll(N,\CC)\cong \su(N)\otimes \CC$, which approximate the basis $\{e^{i k\cdot \bf x}\}_{k\in \ZZ^2\setminus(0,0)}$ on $C^\infty_0(\TT^2,\CC)$, the space of complex-valued functions with zero mean. The restriction of $\sll(N,\CC)$ to a real slice $\su(N)$ as an approximation of $C^\infty_0(\TT^2,\RR)$ will be explained in \autoref{sec:basis_of_su(N)}.

Let $N=2M+1$ an odd integer with $M\in\mathbb{N}$. Here and in the rest of the article, we denote the lattice and the half lattice by
\begin{eqnarray}
    &&Z_N\coloneqq \{-M,\ldots, M\}^2 \setminus (0,0),\\
    &&Z_N^+ \coloneqq Z_N/\pm\label{eq:Z_N and Z_N^+}.
\end{eqnarray}
The quotient $\pm$ in $Z_N^+$ means that $k$ and $-k$ in $Z_N$ are identified. Hence 
\begin{eqnarray}
    Z_N^+\cong \{(k_1,k_2)\in Z_N \mid k_2>0 \}\cup  \{(k_1,k_2)\in Z_N \mid k_1>0 \text{ and }k_2=0 \}.
\end{eqnarray}
Note that our definition $Z_N$ differs from a convention $\ZZ_N^2$ given by $\ZZ_N=\ZZ/N\ZZ\cong \{0,\ldots N-1\}$.

For each $k\in Z_N$, consider the matrix
\begin{eqnarray}\label{eq:Zeitlin_Tk}
    T_k=-\frac{iN}{2\pi}\omega^{-k_1 k_2/2 } U^{k_1}V^{k_2},
\end{eqnarray}
where $\omega=e^{2\pi i /N}$, \emph{the clock matrix} $ U=\on{diag}(1,\omega,\ldots, \omega^{N-1})$, and \emph{the shift matrix}
\begin{eqnarray*}
    V=\left(\begin{array}{ccccc}
0 & 1 & 0 & \ldots & 0 \\
0 & 0 & 1 & \ldots & 0 \\
\ldots & \ldots & \ldots & \ldots & \ldots \\
1 & 0 & 0 & \ldots & 0
\end{array}\right).
\end{eqnarray*}
The matrices $\{T_k\}_{k\in Z_N}$ are called \emph{Zeitlin's basis}. Since each $T_k$ is trace-free and the collection $\{T_k\}_k$ is complex linearly independent, they form a basis of $\sll(N,\CC)$. 

The space $C^\infty_0(\TT^2,\CC)$ is discretized into $\sll(N,\CC)$ via the projection map
     $$P_N:C^\infty_0(\TT^2,\CC)\longrightarrow \sll(N,\CC)\quad ;\quad e^{i p\cdot \bf x}\longmapsto T_p.$$
     
\paragraph{Quantized Lie algebra}
A property of Zeitlin's basis plays a key role in discretizing hydrodynamics and in defining notions of curvature in this article: it approximates the Lie algebra on $C^\infty_0(\TT^2, \CC)$ given by the Poisson bracket. 

Recall that the Poisson bracket on $C^\infty_0(\TT^2, \CC)$ is defined by $\{f,g\}=\nabla^\perp f\cdot \nabla g$ where $\nabla^\perp$ denotes the skew gradient. In terms of the basis $\{e^{ip\cdot \bx}\}_p$, it is
\begin{align}\label{eq:poisson bracket commutator}
    \{e^{ip\cdot \bf x}, e^{iq\cdot \bf x}\}=-(p \times q) e^{i(p+q)\cdot \bf x}.
\end{align}
The following lemma is its discrete counterpart.
\begin{lemma}[Sine bracket commutator]\label{lem:sine bracket commutator}
We have
    \begin{align}\label{eq:Bracket 1}
    [T_p, T_q] = - \frac{N}{\pi} \sin\left(\frac{\pi}{N} p\times q\right)T_{p+q} \mod Z_N
\end{align}
for $T_p= P_N e^{i p\cdot \bf x}$ and $T_q, T_{p+q}$ given in the same way.
\end{lemma}
Here and throughout the article, the cross product between a pair of two dimensional integer vectors is $p\times q\coloneqq p_1 q_2-p_2 q_1$.
A proof of \autoref{lem:sine bracket commutator} is given in \autoref{sec:proofs sec preliminaries}.

Note that the discrete commutator \eqref{eq:Bracket 1} converges to the continuous one \eqref{eq:poisson bracket commutator}. To see this, compute
  \begin{eqnarray*}
   P_N(\{ e^{ip\cdot\bf x},e^{iq\cdot\bf x}\}) &=& P_N\Big( -(p\times q)e^{i(p+q)\cdot \bf x}\Big)\\
   &=&-(p\times q)T_{p+q}
  \end{eqnarray*}
and
  \begin{eqnarray*}
   [P_Ne^{ip\cdot \bf x},P_Ne^{iq\cdot \bf x}] &=& [T_p,T_q]\\
   &\stackrel{\eqref{eq:Bracket 1}}{=}&-\frac{N}{\pi} \sin\left(\frac{\pi}{N} p\times q\right)T_{p+q}
   \\
   &=& -\frac{N}{\pi} \left(\frac{\pi}{N}p\times q - \frac{\pi^3}{N^3}(p\times q)^3 +\dots \right)T_{p+q}
   \\
   &=&\left( -p\times q + \frac{\pi^2}{N^2}(p\times q)^3 -\dots \right)T_{p+q}
   \\
   &=&P_N(\{ e^{ip\cdot\bf x},e^{iq\cdot\bf x}\})+O(N^{-2}).
  \end{eqnarray*}
  This shows the following corollary.
  \begin{corollary}\label{cor:Poisson bracket convergence}
We have 
    \begin{eqnarray*}
   P_N(\{ f,g\})= [P_Nf,P_Ng]+O(N^{-2}).
  \end{eqnarray*}
  for $f,g\in C^\infty_0(\TT,\CC)$.
\end{corollary}

  The convergence rate for the discretization error is related to the so-called Berezin-Toeplitz quantization on $C^\infty(\TT^2,\CC)$ \cite{modin2024two}.



\subsection{Basis of $\su(N)$ and Zeitlin's metric}\label{sec:basis_of_su(N)}
We now define a basis on $\su(N)$ using  $\{T_k\}_k$. First note that $\{T_k\}_k$ do not directly form a basis of $\su(N)$, as they are not skew Hermitian but rather unitary with a scaling factor:
\begin{lemma} \label{lem:T_k is unitary}
    We have
    \begin{eqnarray}
        &&T_k^\dag = - T_{-k},\\
        &&T_k T_k^\dag = \frac{N^2}{4\pi^2} I_N.
    \end{eqnarray}
\end{lemma}
A proof is given in \autoref{sec:proofs sec preliminaries}

Using these equalities, we define 
\begin{align}
    &X_k\coloneqq \frac{1}{2}(T_k - T_{k}^\dag)= \frac{1}{2}(T_k + T_{-k}),\\
    &Y_k \coloneqq \frac{1}{2i}(T_k + T_k^\dag) = \frac{1}{2i} (T_k - T_{-k}).
\end{align}

They are $\su(N)$ counterparts of the basis $\{\cos(k\cdot \bx), \sin(k\cdot \bx)\}_k$ on $C_0^\infty(\TT,\RR)$.
By design, $X_k$ and $Y_k$ are trace-free and skew-Hermitian. Note also that $X_k = X_{-k}$ and $Y_k=-Y_{-k}$ just like $\cos$ and $\sin$. Hence the collection $\{X_k, Y_k\}_{k\in Z_N^+}$ is linearly independent in real coefficients. Since $\#Z_N^+=(N^2-1)/2$,  the vectors $\{X_k, Y_k\}_{k\in Z_N}$ span a $N^2-1$ dimensional space on $\su(N)$, and hence form a basis.

\begin{proposition}\label{prop:basis on su (N) with standard metric}
 The collection $\{ X_k, Y_k\}_{k\in Z_N}$ is an orthogonal basis of $\su(N)$ with $(X_k,X_k)=(Y_k, Y_k)=\frac{N^{3}  }{8\pi^2}$, with respect to the standard metric $(u,v)=\Tr( u^\dag v)=-\Tr(uv)$.
\end{proposition}
A proof is given in \autoref{sec:proofs sec preliminaries}

\subsection{Quantized Laplacians and their eigenvalues}\label{sec:quantized Laplacians}
Zeitlin's discretization of fluid dynamics requires a notion of Laplacian operator on the Lie algebra $\su(N)(\subset \sll(N,\CC))$.  We  review several Laplacians proposed in the previous work, which are ingredients of the Riemannian metric and the resulting Ricci curvature we introduce in this article. 

\paragraph{Spectral Laplacian}
The simplest one may be \emph{the spectral Laplacian}:
\begin{eqnarray*}
    \Delta ^\on{spec}:T_k\mapsto -|k|^2 T_k
\end{eqnarray*}
 proposed in Zeitlin's original discretization \cite{zeitlin1991torus}. The operator $ \Delta ^\on{spec}$ has the same eigen values as the Laplacian on $C^\infty_0(\TT^2, \CC)$.  At the discrete level, however, $ \Delta ^\on{spec}$ corresponds to a non-local operation on  the grid $(\ZZ/N\ZZ)^2$ discretizing $\TT^2$, given by $D^\on{spec}\coloneqq -\on{IFT}\circ |k|^2\circ \on{FT}$, defined using the discrete Fourier transform $\on{FT}$ and the inverse discrete Fourier transform $\on{IFT}$. 

\paragraph{Adjoint Laplacians}
Another discretization of Laplacian uses the structure of Poisson bracket. On the torus, the Laplacian can be expressed using nested Poisson brackets
  \begin{equation}\label{Laplacian}
   \Delta f\coloneqq\{e^{-ix},\{e^{ix},f\}\}+\{e^{-iy},\{e^{iy},f\}\}=-(\frac{\partial^2}{\partial x^2} + \frac{\partial^2}{\partial y^2})f.
  \end{equation} 

\emph{The adjoint Laplacian} $\Delta^{\rm ad}$  \cite{dowker1992finite} directly replaces the Poisson brackets with the matrix commutators 
  \begin{eqnarray*}\label{eq:adjoint_Laplacian1}
    \Delta^{\ad} w &\coloneqq& [ P_N e^{-ix},[P_N e^{-ix},w]] + [P_N e^{-iy},[P_N e^{iy},w]]\\
     &=& [ T_{-1,0},[T_{1,0},w]] + [ T_{0,-1},[T_{0,1},w]].
  \end{eqnarray*}
  
A direct computation using the commutator relation \eqref{eq:Bracket 1} shows that the $T_k$ and $T_{-k}$ are eigenvectors of $\Delta_N^{\ad}$ sharing the same eigenvalue:
  \begin{equation}\label{eq:adjoint L eigenvalues}
    \lambda_k = -\left(\frac{N}{\pi}\right)^2 \left( \sin^2\left(\frac{\pi k_1}{N}\right) + \sin^2\left(\frac{\pi k_2}{N}\right) \right).
  \end{equation}
The operator $\Delta_N^{\ad}$ is consistent the Lie algebra structure itself. Since $\sin(\theta)/\theta\to 1$ as $\theta \to 0$,  the eigenvalue $\lambda_k$ converges to $-|k|^2$, that of the continuous Laplacian. 


Note also that these eigenvalues \eqref{eq:adjoint L eigenvalues} agree with those of the finite difference Laplacian $D_N^\on{diff}$ on the discrete grid $(\ZZ/N\ZZ)^2$ of $\TT^2$ (See \cite{leveque2007finite} for example),  given by
  \begin{eqnarray}\label{eq:finite difference Laplacian}
      D_N^\on{diff} f|_{k} = \frac{f_{k_1+1, k_2}+ f_{k_1-1, k_2}+ f_{k_1, k_2+1} + f_{k_1, k_2-1}-4f_{k_1, k_2}}{(2\pi/N)^2}.
  \end{eqnarray}
  The operator $D_N^\on{diff}$  is a purely local operator, in contrast to the non-local operator $D_N^\on{spec}\coloneqq -\on{IFT}\circ |k|^2\circ \on{FT}$ associated to the spectral Laplacian $\Delta^{\rm spec}$. Hence it may be more natural to use $\Delta_N ^\on{ad}$, as a discrete counterpart of $\Delta$ on $C^\infty(\TT,\CC)$.
  
  
\paragraph{Diagonal Laplacian}
The adjoint Laplacian has a variant \cite{zeitlinMHD} given by
  \begin{equation}\label{eq:adjoint_Laplacian2}
    \Delta^{\rm diag}_N w := \frac{1}{2}\left([T_{-1,-1}, [T_{1,1},w]]] + [T_{-1,1},[T_{1,-1},w]]\right).
  \end{equation} 
  We call $\Delta^{\rm diag}$ the \emph{diagonal Laplacian}. 
  Its eivenvectors are also $T_k, T_{-k}$ with thier common  eivenvalues
  \begin{eqnarray*}
    \lambda_k = -\frac{1}{2}\left(\frac{N}{\pi}\right)^2 \left( \sin^2\left(\frac{\pi }{N}(k_1-k_2)\right) + \sin^2\left(\frac{\pi}{N}(k_1+k_2)\right) \right),
  \end{eqnarray*}
 which also converges to $-|k|^2$ as $N\to\infty$. 
 The operator $\Delta^{\rm diag}_N$ corresponds to another notion of finite difference Laplacian on $(\ZZ/N\ZZ)^2$,
   \begin{eqnarray*}
      D_N^\on{diag} f|_{k} = \frac{f_{k_1-1, k_2-1}+ f_{k_1-1, k_2+1}+ f_{k_1+1, k_2-1} + f_{k_1+1, k_2+1}-4f_{k_1, k_2}}{(2\sqrt{2}\pi/N)^2},
  \end{eqnarray*}
 which uses values of the grid points aligned diagonally around $k$, in contrast to that  $D^{\rm diff}_N$ uses values vertically and horizontally \eqref{eq:finite difference Laplacian}.

The diagonal Laplacian appears to be as natural a choice as the adjoint Laplacian, since it also comes from a local operator on grids. However, it exhibits behavior not present in the continuous Laplacian: the eigenvalue $\lambda_k$ is not decreasing in $|k_1|$ and $|k_2|$, as shown in \autoref{fig:eigenvalues of Laplacians}.  
This appears to affect the distribution of the resulting Ricci curvature values, as we observe later.

 \begin{figure}[htbp]
\centering
\begin{minipage}{0.32\textwidth}
    \caption*{$\Delta^{\rm spec}$}
    \vspace{-15pt}
\includegraphics[width=1.0\textwidth]{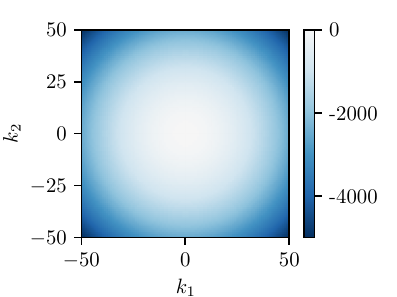}
\end{minipage}
\begin{minipage}{0.32\textwidth}
     \caption*{$\Delta^{\ad}$}
     \vspace{-15pt}\includegraphics[width=1.0\textwidth]{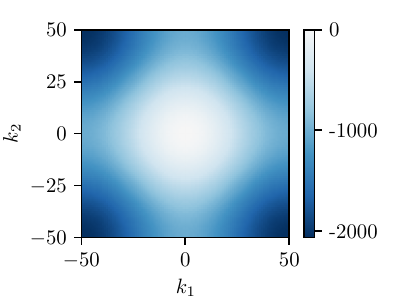}
\end{minipage}
\begin{minipage}{0.32\textwidth}  
  \caption*{$\Delta^{\rm diag}$}
\vspace{-15pt}\includegraphics[width=1.0\textwidth]{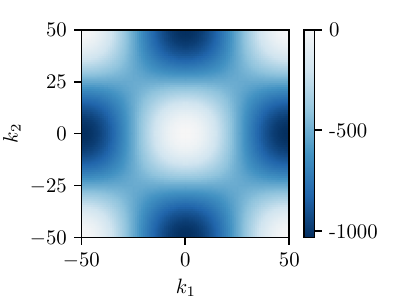}
\end{minipage}
\caption{
Eigenvalues of discrete Laplacians on $\su(101)$: the spectral (left), the adjoint (middle), and the diagonal (right).
 }
\label{fig:eigenvalues of Laplacians}
\end{figure}

\autoref{tab:properties of Laplacians} summarizes the properties of the above discrete Laplacians. Since the adjoint Laplacian $\Delta^{\rm ad}$ inherits all the listed properties of the continuous Laplacian, it could be the most canonical candidate.
In \autoref{sec:Ricci on suN}, we will compute and compare the Ricci curvature for these three Laplacians.

\begin{table}[htbp]
    \centering
    \begin{tabular}{l||c|ccc}
    \hline
    & $\Delta$ on $C^\infty_0(\TT^2)$
      & $\Delta^{\rm spec}$
      & $\Delta^{\rm ad}$
      & $\Delta^{\rm diag}$ \\
    \hline
    Local operator?
    & \cmark & \xmark & \cmark & \cmark \\

    $\lambda_k$ decreasing in $|k_1|$ and $|k_2|$?
    & \cmark & \cmark & \cmark & \xmark \\
     Lie algebra structure?
    & \cmark & \xmark & \cmark & \xmark \\
    \hline
\end{tabular}
    \caption{Properties of the continuous and discrete Laplacians.}
    \label{tab:properties of Laplacians}
\end{table}



\subsection{Zeitlin's metric on $\su(N)$}
In this subsection, we define Zeitlin's metric on $\su(N)$ using discrete Laplacians we just reviewed. We first observe that these Laplacians defined on $\sll(N,\CC)$ behave well when restricted onto $\su(N)$:

\begin{proposition}\label{prop:L is nice on su(N)}
    Let $L\colon \sll(N,\CC)\to \sll(N,\CC)$ be either of $\Delta^{\rm spec}, \Delta^{\ad} $ or $ \Delta^{\rm diag}$. Then $L$ is a bijective and Hermitian operator on $\su(N)$ with respect to the standard metric $(u,v)=\Tr(u^\dag v)$.
\end{proposition}

To prove the proposition, we use the following lemma:
\begin{lemma}\label{lem:eigenvectors_of_L}
    Let $L$ is a linear operator on $\on{M}(N,\CC)$. Suppose that the eigenvectors of $L$ are $\{T_k, T_{-k}\}_{k\in Z_N}$  and that each pair $(T_k,T_{-k})$ share the same real eigenvalues $\lambda_k$. Then  $\{X_k, Y_k\}_k$ are eigenvectors of $L$ with eigenvectors $\{\lambda_k\}_k$.
\end{lemma}
\begin{proof}
    Since $L$ is linear, we can pass it through the scalar multiplication and addition defining 
    $X_k$ as follows
\begin{align}
LX_k &= L\left( \frac{1}{2}(T_k + T_{-k}) \right) 
= \frac{1}{2} (\lambda_k T_k) + \frac{1}{2} (\lambda_k T_{-k}) 
= \lambda_k X_k.
\end{align}
The same computation applies to $LY_k=\lambda_k Y_k$.
\end{proof}

\begin{proof}[Proof of \autoref{prop:L is nice on su(N)}]
The first claim that $L$ is bijective on $\su(N)$ follows directly from \autoref{lem:eigenvectors_of_L}. 

We now prove the second claim that $L$ is Hermitian. For the spectral Laplacian $\Delta^{\rm spec}$, we have for $u=\sum_j u_j E_j, v=\sum_j v_j E_j$ that,
\begin{eqnarray*}
    (\Delta u,v)=\sum_{E_j} (u_j \Delta E_j, v_j E_j)=\sum_j -|j|^2 \bar u_j v_j (E_j, E_j) 
    = \sum_j (u_j E_j, v_j \Delta E_j) = (u,\Delta v)
\end{eqnarray*}
where $E_j$ runs through the basis $\{X_k,Y_k\}_k$.

For the adjoint Laplacians $\Delta^{\ad}, \Delta^{\rm diag}$, it suffices to show the Hermitianity of the operator $L$ given by $Lw\coloneqq [T_{-k},[T_k,w]]$ for each fixed $k\in Z_N$. We compute using $T_k^\dag=-T_{-k}$ (\autoref{lem:T_k is unitary}), 
    \begin{align}
        \Tr((Lu)^\dag v) 
        & = -\Tr([T_{-k},[T_k,u]]v)\\
        & = \Tr([[T_k,u^\dag],T_{-k}]v)\\
        &=-\Tr(u^\dag[[T_k,v],T_{-k}])\\
        &=\Tr(u^\dag[T_{-k},[T_k,v]])\\
        &=\Tr(u^\dag Lv),
    \end{align}
    which verifies the statement.
\end{proof}

We are now ready to define Zeitlin's metric on $\SU(N)$.
\begin{definition}[Zeitlin's metric]\label{def:Zeitlin metric}
    Let $L$ be an invertible and Hermitian operator on $\su(N)$ with respect to the standard inner product $(u,v)=\Tr(u^\dag v)$ with negative eigenvalues.  Then we define a Riemannian metric on $\SU(N)$ by
    \begin{align}\label{eq:Zeitlin metric}
        \langle u,v \rangle^L= -\hbar_N(u, Lv) =- \hbar_N\Tr(u^\dag L v).
    \end{align}    
    Here $\hbar_N=\frac{2^4\pi^4}{N^3}$ and $u,v$ are interpreted as elements of $T_\mathbb{I}\SU(N)=\su(N)$ through the left translation by $\SU(N)$. 
\end{definition}

In  \autoref{def:Zeitlin metric}, the constant $\hbar_N$ is chosen so that Zeitlin's metric approximates the metric on $\SDiff(\TT^2)$ for the incompressible Euler equation, as follows. 

Using \autoref{prop:basis on su (N) with standard metric} ($(X_k, X_k)=\frac{N^3}{8\pi^2}$), compute
\begin{eqnarray*}
   \langle X_k, X_k\rangle^L= -\hbar_N(X_K, LX_k)=-\hbar_N \lambda_k (X_k,X_k) =-\lambda_k \frac{2^4\pi^4}{N^3}\frac{N^3}{8\pi^2} = -2\pi^2\lambda_k.
\end{eqnarray*}
For $L$ being either of the spectral or the adjoint Laplacians, $\lambda_k$ converges to $-|k|^2$. Hence $ \langle X_k, X_k\rangle^L$ asymptotically agrees with the continuous counterpart on $\SDiff(\TT^2)$ as 
\begin{align*}
    g(\nabla^\perp \cos k\cdot \bx, \nabla^\perp  \cos k\cdot \bx) 
    &=\int_{\TT^2} \nabla^\perp \cos (k\cdot{\bf x})\cdot \nabla^\perp \cos (k\cdot{\bf x}) \, d{\bf x}
    \\
    &= -\int_{\TT^2}\cos (k\cdot{\bf x}) \Delta \cos (k\cdot{\bf x}) \, d{\bf x}
    \\
    &=|k|^2 \int_{\TT^2} \cos^2(k\cdot {\bf x}) \, d{\bf x}
    =2\pi^2 |k|^2.
\end{align*}
The same applies to $Y_k \sim \sin(k\cdot \bx)$.
Summarizing the argument, we have the following:
\begin{proposition}
    Let $L$ be an operator on $\su(N)$ as assumed in \autoref{def:Zeitlin metric}. Suppose $L$ has eigenvectors $\{X_k, Y_k\}_k$ with eivenvalues $\{\lambda_k\}_k$. Then the matrices $\{X_k, Y_k\}_k$ define an orthogonal basis on $\su(N)$ with $\langle X_k, X_k\rangle^L=\langle Y_k, Y_k\rangle^L=-2\pi^2 \lambda_k$.
\end{proposition}


\section{Ricci curvature on $\SU(N)$}\label{sec:Ricci on suN}
In this section, we present our main result. That is, we define a Ricci curvature tensor on $\SU(N)$ and derive an explicit formula.

\begin{definition}[Ricci curvature on $\SU(N)$]\label{def:Ricci}
Let $N$ be a positive odd integer. 
On $\SU(N)$ equipped with Zeitlin's metric $\langle \cdot, \cdot \rangle^L$, we define the Ricci curvature tensor $\Ric_N \in \Gamma (T^*\SU(N)\otimes T^*\SU(N))$ by
\begin{eqnarray*}
\Ric_N(E_k,E_l)=
%
\frac{1}{N^2-1} \sum_{i\in Z_N^+} {\langle  R(X_i,E_k)E_l  ,X_i  \rangle \over \langle X_i, X_i \rangle}+ {\langle  R(Y_i,E_k)E_l  ,Y_i  \rangle \over \langle Y_i, Y_i \rangle}.
\end{eqnarray*}
where $E_k\in\{X_k, Y_k\}, E_l \in\{X_l, Y_l\}$. 
\end{definition}
When it is clear from context, we just write $\Ric$ by dropping the subscript $N$.

\autoref{def:Ricci} is based on the formula for the Ricci curvature on a finite-dimensional Riemannian manifold. Additionally, we used the normlization factor $\frac{1}{N^2-1}=\frac{1}{{\rm dim}(\SU(N))}$, following the approach in \cite{lukatskii1979curvature, lukatskii1984curvature} for $\SDiff(\TT^2)$, which prevents the Ricci curvature value from diverging as $N\to \infty$. Alternatively, one may encode the normalization factor within Zeitlin's metric, as in \cite{lichtenfelz2025ricciZeitlin}. The choice of where to place  normalization is largely a matter of preference. 

A main result of the article is the following formula:
\begin{theorem}\label{th:Ricci curvature formula}
For $E_k, E_l\in \{X_i, Y_i\}_i$, we have
    \begin{align}
        \Ric_N(E_k, E_l)=
        \begin{cases}
			 \sum_{i\in Z_N^+} \frac{\left(\frac{N}{\pi}\sin(\frac{\pi}{N}k\times i)\right)^2}{N^2-1}  \frac{ -3(\lambda_{k+i} + \lambda_{k-i})
    + (\lambda_i-\lambda_k)^2( \lambda_{k+i}^{-1} +\lambda_{k-i}^{-1})  +4(\lambda_k+\lambda_i)}{8\lambda_{i}} & \text{if } E_k = E_l,\\
            0 & \text{otherwise.}
		 \end{cases} 
    \end{align}
\end{theorem}
\vspace{5pt}
In particular, the Ricci curvature tensor is diagonal with respect to the basis $\{X_i,Y_i\}_i$.

\paragraph{Numerical examples.} \autoref{fig:Ricci different Laplacians} numerically evaluates the Ricci curvatures for different choices of quantized Laplacians.
Using the formula given here, we compute the Ricci curvature for each orthonormal base $\tilde E_k= \frac{1}{\sqrt{- 2\pi^2 \lambda_k}} E_k$ for $E_k\in \{X_i,  Y_i\}_i$. That is, $\Ric(\tilde E_k, \tilde E_k)= \Ric(E_k, E_k)/(-2\pi^2\lambda_k)$.
%
 We observe that the adjoint $\Delta^{\rm ad}$ (top row) and the diagonal Laplacian $\Delta^{\rm diag}$ (middle row) exhibit both negative and positive Ricci curvature values. In contrast, the spectral Laplacian $\Delta^{\rm spec}$ (bottom row) takes only negative values, similarly to Lukastkii's definition of Ricci curvature \cite{lukatskii1984curvature} (\autoref{prop:negative Ricci for spec}). 

 Note also that the distribution of Ricci curvature values for $\Delta^{\rm diag}$ (middle left) exhibits negative values at large $k$ and a checkerboard-like artifact near the center. This anomalous behavior appears to originate from the non-monotonic decrease of $\lambda_k$ and its diagonal nature (see \autoref{sec:quantized Laplacians}).

\begin{figure}[htbp]
\centering
\begin{minipage}{0.45\textwidth}
    \centering
    On $\SU(101)$
\end{minipage}
\begin{minipage}{0.5\textwidth}
    \centering
    Transition in $N$ 
\end{minipage}
\caption*{$L=\Delta^{\ad}$} 
 \vspace{-15pt}
\begin{minipage}{0.45\textwidth}  
 \includegraphics[width=1.0\textwidth]{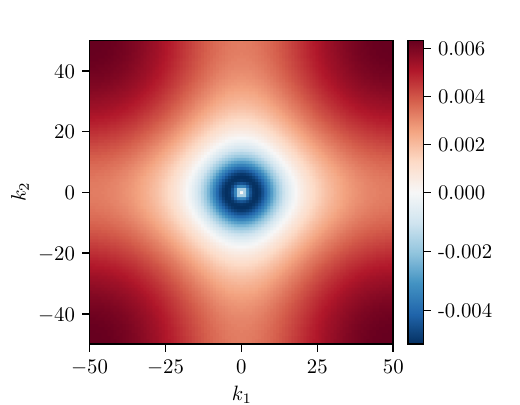}\end{minipage}
 \begin{minipage}{0.5\textwidth}  
 \includegraphics[width=1.0\textwidth]{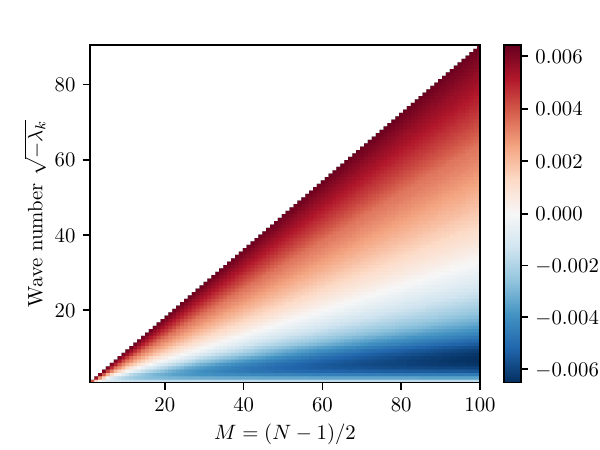}\end{minipage}

 \caption*{$L=\Delta^{\rm diag}$} 
  \vspace{-15pt}
 \begin{minipage}{0.45\textwidth}  
 \includegraphics[width=1.0\textwidth]{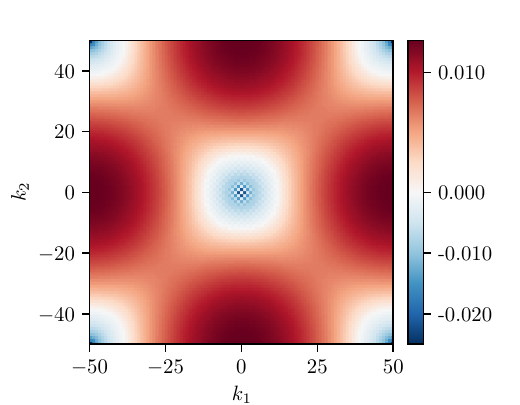}\end{minipage}
\begin{minipage}{0.5\textwidth} 
\includegraphics[width=1.0\textwidth]{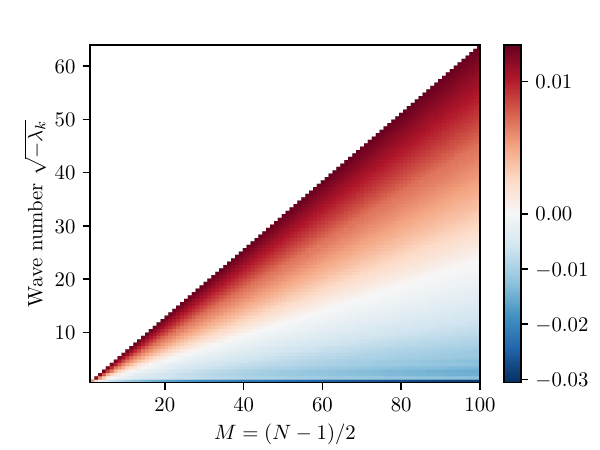} 
\end{minipage}

 \caption*{$L=\Delta^{\rm spec}$}
  \vspace{-15pt}
 \begin{minipage}{0.45\textwidth}  
 \includegraphics[width=1.0\textwidth]{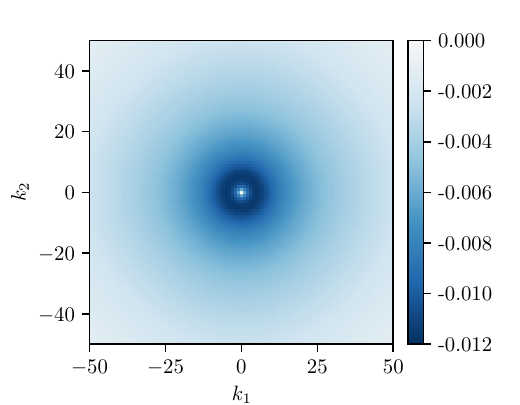}\end{minipage}
\begin{minipage}{0.5\textwidth} 
\includegraphics[width=1.0\textwidth]{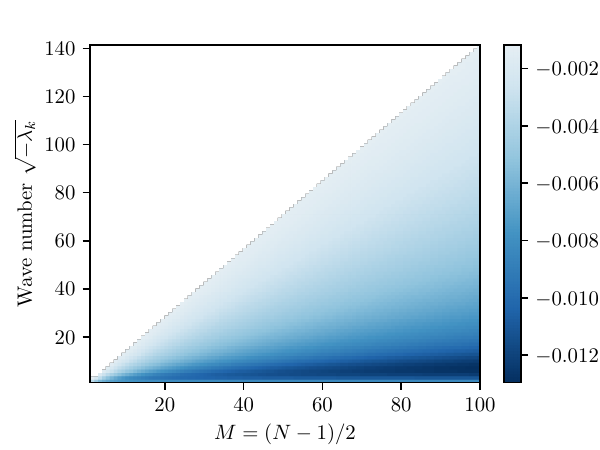} 
\end{minipage}
\caption{
Ricci curvature computed with the adjoint $\Delta^{\ad}$ (top), diagonal adjoint $\Delta^{\rm diag}$ (middle), and spectral $\Delta^{\rm spec}$ (bottom) Laplacians.
In each row, the left image shows $\Ric(\tilde E_k, \tilde E_k)$ on $\SU(101)$, and the right image shows the transition of the curvature, where the horizontal axis is the matrix size $N$, the vertical axis $\sqrt{-\lambda_k}\sim |k|$ approximates the magnitude of the wave number, and the color at each location represents the value of $\Ric(\tilde E_k, \tilde E_k)$. 
 }
\label{fig:Ricci different Laplacians}
\end{figure}

\paragraph{Outline of derivation.}

In the rest of the section,  we derive the Ricci curvature formula in \autoref{th:Ricci curvature formula} on $\SU(N)$.
Recall that, for the left (right) invariant vector fields $X,Y,Z,W$ and the Levi-Civita connection $\nabla$, we have 
\footnote{Here we used the identity $-\langle\nabla_{[X,Y]}Z, W\rangle=\frac{1}{2}(\langle -\ad_{[X,Y]}Z+\ad^\star_{[X,Y]}Z + \ad^\star_Z[X,Y], W\rangle)$. The first term  reads $\langle -\ad_{[X,Y]}Z, W\rangle =\langle \ad_{Z}[X,Y], W\rangle=\langle[X,Y], \ad^\star_Z W\rangle$. The second term is $\langle \ad^\star_{[X,Y]}Z, W\rangle = \langle Z, \ad_{[X,Y]}W\rangle=-\langle Z, \ad_W [X,Y]\rangle=\langle[X,Y], - \ad_W^\star Z\rangle$. The third term is $\langle \ad^\star_Z [X,Y], W\rangle=\langle[X,Y], [Z,W]\rangle$.}
    \begin{eqnarray*}
     &&\langle  R(X,Y)Z  ,   W  \rangle  =   \langle  \nabla_X\nabla_YZ  -  \nabla_Y\nabla_XZ  -  \nabla_{\ad_X Y}Z   ~,~   W  \rangle
     \\
     &=&  - \langle  \nabla_YZ    ,   \nabla_XW  \rangle    +   \langle\nabla_XZ    ,   \nabla_YW \rangle        -    \langle  \nabla_{\ad_X Y}Z  , W\rangle  
     \\
     &=&     \langle\nabla_XZ    ,   \nabla_YW \rangle   - \langle  \nabla_YZ    ,   \nabla_XW  \rangle + \frac{1}{2}\langle \ad_
     X Y   ,  \ad_Z W   + \ad^\star_ZW    -  \ad^\star_WZ    \rangle.
     \end{eqnarray*}
Here $R$ denotes the Riemann curvature tensor, $\nabla_X Y$ is the covariant derivative, and $\ad_X Y=[X,Y]$ is the adjoint representation on $\su(N)$. The corresponding coadjoint representation $\ad_X^\star Y$ is defined by $$\langle \ad_X^\star Y, Z\rangle= \langle Y,\ad_X Z  \rangle.$$

Using this,  we compute for $E_i,E_k,E_l\in \{X_k,Y_k\}_k$ that,
\begin{eqnarray}\label{eq:numerator_of_Ricci}
     \langle  R(E_i,E_k)E_l  ,   E_i  \rangle  &=&      
     \langle\nabla_{E_i}E_k,\nabla_{E_l} E_i   \rangle   
     -
     \langle  \nabla_{E_k}E_l    ,   {\xcancel {\nabla_{E_i}E_i}}  \rangle \nonumber
     \\
     &&+     \frac{1}{2}\langle [E_i,E_k]   ,  [E_l,E_i]   + \ad^\star_{E_l}E_i   -   \ad^\star_{E_i}E_l    \rangle.
    \end{eqnarray}
    
    The second term turns out to vanish as $\nabla_{E_i}E_i=0$, following from \autoref{lem:covariant_derivatives}.
    Note also that each term of \eqref{eq:numerator_of_Ricci} is invariant under the sign flip of index $i\to -i$, namely $X_i\to X_{-i}=X_i$ and $Y_i\to Y_{-i}=-Y_i$ (and for $k,l$ respectively). 
    
    Therefore, we have
    \begin{align}
      (N^2-1) \Ric(E_k,E_l)=\sum_{E_i}&\Bigg(\frac{\langle\nabla_{E_i}E_k,\nabla_{E_l} E_i   \rangle  
    +\frac{1}{2}\langle \ad_{E_i}E_k,  \ad_{E_l}E_i \rangle}{\langle E_i, E_i\rangle}\nonumber \\
    &\quad + \frac{
    +\frac{1}{2}\langle \ad_{E_i}E_k,   \ad^\star_{E_l}E_i   \rangle -\frac{1}{2}\langle \ad_{E_i}E_k,  \ad^\star_{E_i}E_l    \rangle}{\langle E_i, E_i\rangle}
    \Bigg)\label{eq:Ricci four terms}
    \end{align}
where $E_i$ runs through the basis $\{X_i, Y_i\}_i$.

We will evaluate each term in \eqref{eq:Ricci four terms}.
To this end, we start from collecting auxiliary results in the next subsection.
\subsection{Auxiliary computations}\label{sec:auxiliary computations}

\subsubsection{$\ad$ computations}
 We first compute adjoint operator $\ad_u v=[u,v]$ for $u,v$ chosen from $\{X_k, Y_k\}_k$, namely structure constants with respect to this basis. 

\begin{lemma}[Structure constants]\label{lem:[XkXl]}
    We have
    \begin{align*}
        [X_k, X_l]&=-\frac{1}{2}\frac{N}{\pi}\sin(\frac{\pi}{N}k\times l)(X_{k+l}-X_{k-l})\\
        &\to \{\cos k\cdot\bx, \cos l\cdot\bx\}=-\frac{1}{2}k\times l (\cos (k+l)\cdot\bx - \cos (k-l)\cdot\bx),\\
        [Y_k, Y_l]&=\frac{1}{2}\frac{N}{\pi}\sin(\frac{\pi}{N}k\times l)(X_{k+l}+X_{k-l}),\\
         &\to \{\sin k\cdot\bx, \sin l\cdot\bx\}=\frac{1}{2}k\times l (\cos (k+l)\cdot\bx + \cos (k-l)\cdot\bx),\\
        [X_k, Y_l]&=-\frac{1}{2}\frac{N}{\pi}\sin(\frac{\pi}{N}k\times l)(Y_{k+l}+Y_{k-l})\\
        &\to \{\cos k\cdot\bx, \sin l\cdot\bx\}=-\frac{1}{2}k\times l (\sin (k+l)\cdot\bx + \sin (k-l)\cdot\bx)\\
        [Y_k, X_l]&=-\frac{1}{2}\frac{N}{\pi}\sin(\frac{\pi}{N}k\times l)(Y_{k+l}-Y_{k-l})\\
        &\to \{\sin k\cdot\bx, \cos l\cdot\ x\}=-\frac{1}{2}k\times l (\sin (k+l)\cdot\bx - \sin (k-l)\cdot\bx).
    \end{align*}
    In particular, they all vanish if $k=l$.
\end{lemma}
A proof is given in \autoref{sec:proofs auxiliary lemmas}.

\subsubsection{$\ad^\star$ computations}
We then compute the coadjoint operator with respect to the Zeitlin metric. Recall that the coadjoint representation $\ad^\star$ is defined by
$\langle\ad^\star_XY,Z\rangle=\langle Y,\ad_XZ\rangle$. 

We start from a generic formula for the coadjoint representation and then apply this to the basis $\{X_k,Y_k\}_k$.


\begin{lemma}\label{lem:coad_formula_for_Laplacian_metric}
    Let $\ad^\star$ be the coadjoint representation with respect to Zeitlin's metric (\autoref{def:Zeitlin metric}).  We have 
    \begin{align}
        \ad^{\star}_u v =-L^{-1}\ad_u L v.
    \end{align}
\end{lemma}
\begin{proof}
We compute 
\begin{align}
    \Tr((\ad_u^\star v)^\dag L w)
    & =-\frac{1}{\hbar_N}\langle \ad_u^\star v,w\rangle \\
    & = -\frac{1}{\hbar_N}\langle v,[u,w]\rangle\\
    & = \Tr(v^\dag L[u,w])\\
    & = \Tr(w [v^\dag L, u ])\quad \text{ by identity } \Tr(A[B,C])=\Tr(C[A,B])\\
    & = \Tr([v^\dag L, u ] w).
\end{align}
Hence $(\ad_u^\dag v)^\dag L= [v^\dag L, u ]$ and since $[A,B]^\dag =[B^\dag ,A^\dag ]$, we have
\begin{align}
    \ad_u^\star v 
    = ([v^\dag L, u ]L^{-1})^\dag  
    = L^{-1}[u^\dag , Lv] 
    = - L^{-1}[u, Lv] = -L^{-1}\ad_u Lv.
\end{align}

\end{proof}

\begin{lemma}[Coadjoint operators]\label{lem:coadjoint}
    \begin{eqnarray*}
        \ad_{X_k}^\star X_l &=& \frac{1}{2}\frac{N}{\pi} \sin(\frac{\pi}{N}k\times l) \lambda_l (\frac{X_{k+l}}{ \lambda_{k+l}}-\frac{X_{k-l}}{ \lambda_{k-l}}),\\
         \ad_{Y_k}^\star Y_l &=& -\frac{1}{2}\frac{N}{\pi} \sin(\frac{\pi}{N}k\times l) \lambda_l (\frac{X_{k+l}}{ \lambda_{k+l}}+\frac{X_{k-l}}{ \lambda_{k-l}},)\\
          \ad_{X_k}^\star Y_l &=& \frac{1}{2}\frac{N}{\pi} \sin(\frac{\pi}{N}k\times l) \lambda_l (\frac{Y_{k+l}}{ \lambda_{k+l}}+\frac{Y_{k-l}}{ \lambda_{k-l}}),\\
          \ad_{Y_k}^\star X_l &=&\frac{1}{2}\frac{N}{\pi} \sin(\frac{\pi}{N}k\times l) \lambda_l (\frac{Y_{k+l}}{ \lambda_{k+l}}-\frac{Y_{k-l}}{ \lambda_{k-l}}).
    \end{eqnarray*}
    When $k=l$, all of the above vanish.

     
\end{lemma}
A proof is given in \autoref{sec:proofs auxiliary lemmas}.

\subsubsection{Covariant derivative}
We next compute the covariant derivative on $\su(N)$ by evaluating the formula,
\begin{eqnarray}\label{eq:covariant-derivative}
&&  2\nabla_{X}Y  =   \ad_{X}Y - \ad^\star_{X}Y - \ad^\star_{Y}X.
\end{eqnarray}
for the left invariant metric.

\begin{lemma}[Covariant derivative]\label{lem:covariant_derivatives}
        \begin{eqnarray*}
      2\nabla_{X_k}X_l  &=&  \frac{N}{2\pi}\sin(\frac{\pi}{N}k\times l)
      \Big(-(1+\frac{\lambda_l-\lambda_k}{\lambda_{k+l}}) X_{k+l} + (1+\frac{\lambda_l-\lambda_k}{\lambda_{k-l}})  X_{k-l}
       \Big)\\
       2\nabla_{Y_k}Y_l  &=&\frac{N}{2\pi}\sin(\frac{\pi}{N}k\times l)
      \Big((1+\frac{\lambda_l-\lambda_k}{\lambda_{k+l}}) X_{k+l} + (1+\frac{\lambda_l-\lambda_k}{\lambda_{k-l}})  X_{k-l}
       \Big) \\
        2\nabla_{X_k}Y_l  &=& \frac{N}{2\pi}\sin(\frac{\pi}{N}k\times l)
      \Big(-(1+\frac{\lambda_l-\lambda_k}{\lambda_{k+l}}) Y_{k+l} - (1+\frac{\lambda_l-\lambda_k}{\lambda_{k-l}})  Y_{k-l}
       \Big) \\
        2\nabla_{Y_k}X_l  &=& \frac{N}{2\pi}\sin(\frac{\pi}{N}k\times l)
      \Big(-(1+\frac{\lambda_l-\lambda_k}{\lambda_{k+l}}) Y_{k+l} + (1+\frac{\lambda_l-\lambda_k}{\lambda_{k-l}})  Y_{k-l}
       \Big)
    \end{eqnarray*}
    
    When $k=l$, all of them vanish. 
\end{lemma}
A proof is given in \autoref{sec:proofs auxiliary lemmas}.


\

\subsection{Off-diagonal entries of $\Ric$}\label{sec:off diagonal}
Using the auxiliary computations we collected, we compute the Ricci curvature tensor with respect to the basis $\{X_i, Y_i\}_i$. 

We begin with showing that the off-diagonal entries are zero. It turns out that each term of \eqref{eq:Ricci four terms} is zero, which we show in this subsection.

\begin{lemma}\label{lem:vanishing_off_diagonals}
Let $E_k, E_l \in \{X_i, Y_i\}_i$ such that $E_k\neq E_l$. Then 
\begin{eqnarray}
    &&\sum_{E_i}{\langle\nabla_{E_i}E_k,\nabla_{E_l} E_i   \rangle \over \langle E_i, E_i \rangle}=0, \label{eq:vanishing_off_diagonal1}\\
    &&\sum_{E_i} {\langle [E_i, E_k], [E_l, E_i]\rangle\over \langle E_i, E_i\rangle}=0,\label{eq:vanishing_off_diagonal2}\\
     &&\sum_{E_i} {\langle [E_i, E_k], \ad_{E_l}^\star E_i\rangle\over \langle E_i, E_i\rangle}=0\label{eq:vanishing_off_diagonal3},\\
     &&\sum_{E_i} {\langle [E_i, E_k], \ad^\star_{E_i}E_l \rangle\over \langle E_i, E_i\rangle}=0\label{eq:vanishing_off_diagonal4}
\end{eqnarray}
    where $E_i$ runs through the basis $\{X_i, Y_i\}_i$.
\end{lemma}
A proof of \autoref{lem:vanishing_off_diagonals} is given in \autoref{sec:proofs off diagonal}.

Substituting these results into the formula \eqref{eq:Ricci four terms} shows that the Ricci curvature is diagonal:
\begin{proposition}\label{prop:vanishing_off_diagonal}
    Let $E_k, E_l \in \{X_i, Y_i\}_i$ such that $E_k\neq E_l$. Then
    \begin{align}
        \Ric(E_k,E_l)=0.
    \end{align}
\end{proposition}

\subsection{Diagonal entries of $\Ric$}\label{sec:diagonal_entries}
We now compute the diagonal entries by evaluating each term of the formula \eqref{eq:Ricci four terms}.

\begin{lemma}\label{lem:diagonal_entries}
    Let $E_k\in \{X_i,Y_i\}_i$. Then we have,
    \begin{eqnarray}
        &&\sum_{E_i} {\langle\nabla_{E_i}E_k,\nabla_{E_k} E_i \rangle \over \langle E_i, E_i \rangle}=\sum_{i\in Z_N^+} c_{ik}^2\frac{ -(\lambda_{k+i} + \lambda_{k-i})
+ (\lambda_i-\lambda_k)^2( \lambda_{k+i}^{-1} +\lambda_{k-i}^{-1}) }{8\lambda_{i}},\label{eq:diagonal_entry1}\\
&&\sum_{E_i} {\langle [E_i, E_k], [E_k, E_i]\rangle\over \langle E_i, E_i\rangle}= -\frac{1}{2} \sum_{i\in Z_N^+} c_{ik}^2 \left( \frac{\lambda_{i+k} +\lambda_{i-k}}{\lambda_{i}}\right),\label{eq:diagonal_entry2}\\
&&\sum_{E_i} {\langle [E_i, E_k], \ad^\star_{E_k}E_i \rangle\over \langle E_i, E_i\rangle}
=
 \sum_{i\in Z_N^+} c_{ik}^2,\label{eq:diagonal_entry3}\\
&&\sum_{E_i} {\langle [E_i, E_k], \ad^\star_{E_i}E_k \rangle\over \langle E_i, E_i\rangle}=- \sum_{i\in Z_N^+} c_{ik}^2   \frac{\lambda_k}{\lambda_{i}}, \label{eq:diagonal_entry4}
    \end{eqnarray}
    where $c_{ik}=\frac{N}{\pi}\sin(\frac{\pi}{N}i\times k)$, the base $E_i$ runs through $\{X_i,Y_i\}_i$, and $Z_N^+=Z_N/\pm$ is the upper half of the lattice $Z_N$.
\end{lemma}
A proof of \autoref{lem:diagonal_entries} is given in \autoref{sec:proofs  diagonal entries}. 

We are now ready to complete the derivation for our Ricci curvature formula (\autoref{th:Ricci curvature formula}).  
We substitute the identities in \autoref{lem:diagonal_entries} into the formula \eqref{eq:Ricci four terms}. For $E_k\in \{X_i,Y_i\}_i$, compute
  \begin{align}
    (N^2-1)\Ric(E_k,E_k)
    &=\sum_{E_i} \Bigg( {\langle\nabla_{E_i}E_k,\nabla_{E_k} E_i \rangle \over \langle E_i, E_i \rangle}
    +\frac{1}{2} {\langle [E_i, E_k], [E_k, E_i]\rangle\over \langle E_i, E_i\rangle}
    \\
    &\quad+\frac{1}{2} {\langle [E_i, E_k], \ad^\star_{E_k}E_i \rangle\over \langle E_i, E_i\rangle}
    -\frac{1}{2}{\langle [E_i, E_k], \ad^\star_{E_i}E_k \rangle\over \langle E_i, E_i\rangle}\Bigg)
    \\
    &=
    \sum_{i\in Z_N^+} c_{ik}^2 \Big( \frac{ -(\lambda_{k+i} + \lambda_{k-i})
    + (\lambda_i-\lambda_k)^2( \lambda_{k+i}^{-1} +\lambda_{k-i}^{-1}) }{8\lambda_{i}}
    \\
    & \qquad -\frac{1}{4}  \frac{\lambda_{i+k} +\lambda_{i-k}}{\lambda_i} +\frac{1}{2} + \frac{1}{2}\frac{\lambda_k}{\lambda_i} \Big)
    \\
    &=
    \sum_{i\in Z_N^+} c_{ik}^2 \Big( \frac{ -3(\lambda_{k+i} + \lambda_{k-i})
    + (\lambda_i-\lambda_k)^2( \lambda_{k+i}^{-1} +\lambda_{k-i}^{-1})  +4(\lambda_k+\lambda_i)}{8\lambda_{i}}\Big),\qquad
    \label{eq:Ricci computation}
\end{align}
which confirms the expression stated in \autoref{th:Ricci curvature formula}.

\subsection{Asymptotics toward Ricci curvature on $\HDiff(\TT^2)$}\label{sec:Ric asymptotics}
Based on the Ricci curvature $\Ric_N$ on $\SU(N)$, we propose a definition of the Ricci curvature on $\HDiff(\TT^2)$.
\begin{definition}\label{def: continuous Ricci}
    We define the Ricci curvature $\Ric_\infty \in \Gamma (T^*\HDiff(\TT^2)\otimes T^*\HDiff(\TT^2))$ by 
\begin{align}\label{eq:continuous Ricci}
        \Ric_\infty (e_k, e_l)\coloneqq \lim_{N\to\infty }\Ric_N(P_N e_k, P_N e_l),\qquad e_k,e_l\in\{\cos(i\cdot \bx),\sin(i\cdot \bx)\}_i.
\end{align} 
\end{definition}
The value of $\Ric_\infty$  depends on the choice of the discrete Laplacian $\Delta_N$ on $\su(N)$ as \autoref{fig:Ricci different Laplacians} illustrates.

The concurrent work \cite{Lichtenfelz2026personal_communication}  showed that, for fixed $k$, the value $\Ric_N(E_k, E_k)$ of the Ricci curvature induced by the spectral Laplacian $\Delta^{\rm spec}$ converges to a finite value. We expect a similar result for the Ricci curvature with respect to the adjoint Laplacian $\Delta^{\ad}$:
\begin{conjecture}\label{conj:Ricci convergence}
    Let $L=\Delta^{\ad}$, we have 
    \begin{align}\label{eq:Ricci convergence}
         \Ric_\infty (e_k, e_k)
& =
\frac{1}{2}
\int_{[-\frac{1}{2},\frac{1}{2}]^2} \sin ^2(\pi k \times \bx) \left\{\frac{\left(k_1^2 \cos \left(2 \pi x \right)+k_2^2 \cos \left(2 \pi y\right)\right)}{\left( \sin^ 2(\pi x)+\sin^2(\pi y) \right)}\right. \\
&\qquad \left.\quad-\frac{\left(k_1 \sin \left(\pi x\right) \cos \left(\pi x\right)+k_2 \sin \left(\pi y\right) \cos \left(\pi y\right)\right)^2}{\left( \sin^ 2(\pi x)+\sin^2(\pi y) \right)^2}\right\} dx dy,
\end{align}
which is finite for any fixed $k\in \ZZ^2$.
For normalized base $\tilde E_k= \frac{1}{\sqrt{- 2\pi^2 \lambda_k}} E_k$, we have
\begin{align}
    \lim_{N\to \infty} \Ric(\tilde E_k, \tilde E_k)=\frac{1}{2\pi^2|k|^2}\lim_{N\to\infty} \Ric_N(E_k,E_k).
\end{align}

\end{conjecture}
The integral expression \eqref{eq:Ricci convergence} can be  attained by considering the formal  Riemann sum 
\begin{align}
    \Ric_N(E_k,E_k)
    =\frac{1}{N^2-1}\sum_{i\in Z_N^+} f_N(x_i,y_i)
    \approx \frac{1}{N^2}\sum_{i\in Z_N} \frac{1}{2}f_N(x_i,y_i).
\end{align}
Here $(x_i,y_i)\coloneqq (\frac{i_1}{N},\frac{i_2}{N})$ are coordinates in $[-\frac{1}{2},\frac{1}{2}]^2$ and the function $f_N$ is the summand in the formula \eqref{eq:Ricci computation}. By computing the limit $f\coloneqq \lim_{N\to \infty} f_N$, we can formally derive \eqref{eq:Ricci convergence}.
 A rigorous proof would involve  tools like the dominated convergence theorem as in \cite{Lichtenfelz2026personal_communication}.

\autoref{fig:Ric asymptotics} provides some numerical evidence for the asymptotic study. 

\begin{figure}[htbp]
\centering
\begin{minipage}{0.6\textwidth}  
 \includegraphics[width=1.0\textwidth]{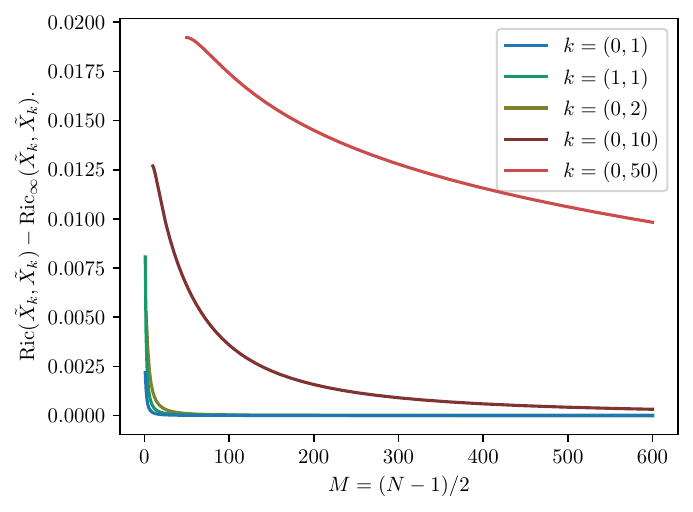}
 \end{minipage}

\caption{Transition of the discrepancies between $\Ric_N(\tilde X_k, \tilde X_k)$ and $\Ric_\infty(\tilde X_k, \tilde X_k) \coloneqq \lim_{N\to\infty}\Ric_N(X_k,X_k)$ for different $k$. For each $k$, we evaluated $\Ric_\infty(\tilde X_k, \tilde X_k)$ via a numerical integration of \eqref{eq:Ricci convergence}. The convergence appears to be slower for large $k$. }
\label{fig:Ric asymptotics}
\end{figure}

\subsection{Sectional curvature}\label{sec:sectional curvature}
A formula for the sectional curvature on $\SU(N)$ can be obtained as a consequence of the Ricci curvature computation we have performed. 
Recall that the sectional curvature of the plane spanned by tangent vectors $u,v$ is defined as 
\begin{align}
    K(u,v)=\frac{\langle R(u,v)v, u\rangle}{|u|^2|v|^2-\langle u,v\rangle^2}.
\end{align}
with the Riemann curvature tensor $R$.

\begin{theorem}\label{th:K(X_k,X_l)}
    For $E_k, E_l\in \{X_i,Y_i\}_i$ with $E_k\neq E_l$, we have
    \begin{align}
        K(E_k, E_l)
        &=\left(\frac{N}{\pi}\sin(\frac{\pi}{N}k\times l)\right)^2\frac{-(\lambda_l-\lambda_k)^2(\lambda_{k-l}^ {-1}+\lambda_{k+l}^{-1})-4(\lambda_k+\lambda_l)+3(\lambda_{k-l}+\lambda_{k+l})}{32\pi^2\lambda_k \lambda_l}.\quad \label{eq:K(X_k,X_l)}
    \end{align}
\end{theorem}
\begin{proof}
    From the computation \eqref{eq:Ricci computation}, we have 
    \begin{align}
        \frac{\langle R(E_k, E_l)E_l, E_k\rangle}{\langle E_l, E_l \rangle}
        &=\frac{1}{2} \Bigg( {\langle\nabla_{E_k}E_l,\nabla_{E_l} E_k \rangle \over \langle E_l, E_l \rangle}
    +\frac{1}{2} {\langle [E_l, E_k], [E_k, E_l]\rangle\over \langle E_l, E_l\rangle}
    \\
    &\qquad+\frac{1}{2} {\langle [E_l, E_k], \ad^\star_{E_k}E_l \rangle\over \langle E_l, E_l\rangle}
    -\frac{1}{2}{\langle [E_l, E_k], \ad^\star_{E_l}E_k \rangle\over \langle E_l, E_l\rangle}\Bigg)\\
    &=c_{kl}^2 \Big( \frac{ -3(\lambda_{k+l} + \lambda_{k-l})
    + (\lambda_i-\lambda_k)^2( \lambda_{k+l}^{-1} +\lambda_{k-l}^{-1})  +4(\lambda_k+\lambda_l)}{16\lambda_{l}}\Big)
    \end{align}
    where $c_{kl}:=\frac{N}{\pi}\sin(\frac{\pi}{N}k\times l)$.
In the first line, the coefficient $\frac{1}{2}$ is needed because in the expression \eqref{eq:Ricci computation},  
 the contributions of $ \frac{\langle R(E_k, X_i)X_i, E_k\rangle}{\langle X_i, X_i \rangle}$ and $ \frac{\langle R(E_k, Y_i)Y_i, E_k\rangle}{\langle Y_i, Y_i \rangle}$ are doubly counted within the upper lattice $Z_N^+$. 
 
 Hence we have, 
    \begin{align}
        K(E_k, E_l)
        &=\frac{1}{\langle E_k, E_k\rangle} \frac{\langle R(E_k, E_l)E_l, E_k\rangle}{\langle E_l,E_l \rangle}\\
        &=\frac{1}{-2\pi^2\lambda_k}c_{kl}^2  \frac{ -3(\lambda_{k+l} + \lambda_{k-l})
    + (\lambda_i-\lambda_k)^2( \lambda_{k+l}^{-1} +\lambda_{k-l}^{-1})  +4(\lambda_k+\lambda_l)}{16\lambda_{l}}\\
        &= c_{kl}^2 \frac{ 3(\lambda_{k+l} + \lambda_{k-l})
    - (\lambda_i-\lambda_k)^2( \lambda_{k+l}^{-1} +\lambda_{k-l}^{-1})  -4(\lambda_k+\lambda_l)}{32\pi^2\lambda_k\lambda_{l}},\label{eq:K(E_k,E_l)}
    \end{align}
    which gives the stated expression.
\end{proof}

The sectional curvature in the continuous case is recovered by replacing the factor $\left(\frac{N}{\pi}\sin(\frac{\pi}{N}k\times l)\right)^2$ with $k\times l$ in \autoref{th:K(X_k,X_l)}.
\begin{corollary}\label{cor:continuous sectional curvature}
    Let $L$ be a linear and invertible operators on $C^\infty_0(\TT^2)$ with eigenvalues $\{\lambda_k\}_k$ for eigenfunctions $\{\cos (i\cdot \bx),\sin(i\cdot \bx)\}_i$.  For the metric $g(\nabla^\perp f, \nabla^\perp h)=\int f(1-L)g$ on the space of Hamiltonian diffeomorphisms ${\rm HDiff} (\TT^2)$, we have 
       \begin{align}\label{eq:continuous sectional curvature}
        \hspace{-25pt}K_{\HDiff(\TT^2)}(e_k, e_l)
        &=|k\times l|^2\frac{-(\lambda_l-\lambda_k)^2(\lambda_{k-l}^ {-1}+\lambda_{k+l}^{-1})-4(\lambda_k+\lambda_l)+3(\lambda_{k-l}+\lambda_{k+l})}{32\pi^2\lambda_k \lambda_l}
    \end{align}
    for $e_k,e_l\in\{\cos (i\cdot \bx),\sin(i\cdot \bx)\}_i$.
\end{corollary}
In case $L$ is the Laplacian $\Delta$, one can verify that the expression \eqref{eq:continuous sectional curvature} coincides with the known formula \cite[Chapter VI Corollary 3.6]{Arnol-khesin} by direct computation.

From the estimate $| \left(\frac{N}{\pi}\sin(\frac{\pi}{N}k\times l)\right)^2- |k\times l|^2|=  O(N^{-2})$ and \autoref{cor:continuous sectional curvature}, we also obtain the convergence of the sectional curvature: 
\begin{corollary}\label{cor:convergence_of_K}
Let $e_k,e_l\in\{\cos (i\cdot \bx),\sin(i\cdot \bx)\}_i$. Suppose the eigenvalues of the discrete operator $L_N$ on $\su(N)$ converges to those of the operator $L$ on $C^\infty_0(\TT^2)$. Then we have
\begin{eqnarray*}
      |K_{{\rm HDiff}(\TT^2)}(e_k, e_l)- K_{\SU(N)}( P_N e_k, P_N e_l)|\to 0
\end{eqnarray*}
    as $N\to\infty$. For $L_N$ being either of $\Delta^{\rm spec},\Delta^{\ad},\Delta^{\rm diag}$ specifically (See \autoref{sec:quantized Laplacians} for definition), there is a constant $C_{k,l}>0$ such that, 
    \begin{eqnarray*}
      |K_{{\rm HDiff}(\TT^2)}(e_k, e_l)- K_{\su(N)}( P_N e_k, P_N e_l)|<C_{k,l} N^{-2}.
    \end{eqnarray*}
\end{corollary}

\begin{remark}[Non-convergence to Lukatskii's Ricci curvature]
    Note that the convergence in \autoref{cor:convergence_of_K} holds only for fixed $k,l\in\ZZ^2\setminus (0,0)$. That is, the convergence is not in general guaranteed when $k$ or $l$ varies dependently on $N$. Consequently, our Ricci curvature give different values from the definition proposed by the earlier work of Lukatskii \cite{lukatskii1984curvature}:
    \begin{align}\label{eq:Lukatskii Ricci}
        \Ric_{\rm Luk}(e_k, e_k)\coloneqq \lim_{N\to\infty}\frac{\langle e_k,e_k\rangle}{N^2-1}\sum_{e_i}K_{{\rm HDiff}(\TT^2)}(e_k, e_i).
    \end{align}
    
\end{remark}

We conclude this subsection by observing that some choice of the discrete Laplacian makes the Ricci curvature tensor purely negative for the basis $\{X_i, Y_i\}_i$.

\begin{lemma}\label{lem:sectional_curvature_qualitative}
    Let $L=\Delta^{\rm spec}$ be the spectral Laplacian. Then $K_{\SU(N)}(E_k,E_l)\leq 0$ for any $E_k,E_l\in \{X_i,Y_i\}_i$. The equality holds if and only if $k$ and $l$ are parallel to each other. Similarly, $K_{\HDiff(\TT^2)}(e_k,e_l)\leq 0$ for any $e_k,e_l\in \{\cos (i\cdot \bx),\sin(i\cdot \bx)\}_i$.
\end{lemma}
A proof of \autoref{lem:sectional_curvature_qualitative} is given in \autoref{sec:proofs sectional curvature}.

\begin{proposition}\label{prop:negative Ricci for spec}
Let $L=\Delta^{\rm spec}$. Then 
we have
    \begin{align*}
       \Ric(E_k,E_k)<0
    \end{align*}
    for $E_k\in\{X_i,Y_i\}_i$ on $\SU(N)$.
    
   Similarly, we have for Lukatskii's Ricci curvature (See \eqref{eq:Lukatskii Ricci} for definition), 
    \begin{align*}
        \Ric_{\rm Luk}(e_k)<0,\qquad 
    \end{align*}
    for $e_k\in\{\cos (i\cdot \bx),\sin(i\cdot \bx)\}_i$ on $\HDiff(\TT^2)$.
\end{proposition}
\begin{proof}
    Since $\{X_i, Y_i\}_i$ are orthogonal, we have 
    \begin{align}
        (N^2-1)\Ric(E_k,E_k)=\langle E_k, E_k\rangle \sum_{E_i} K_{\SU(N)}(E_k, E_i).
    \end{align}
    Then the statement directly follows from the first statement of \autoref{lem:sectional_curvature_qualitative}. The same result for $\Ric_{\rm Luk}(e_k)$ follows from the definition \eqref{eq:Lukatskii Ricci} and the second statement of \autoref{lem:sectional_curvature_qualitative}. 
\end{proof}

\section{Experimental applications for hydrodynamics}\label{sec:applications}

So far, we have introduced a notion of Ricci curvature on $\HDiff(\TT^2)$ using Zeitlin's model. In this section, we take a somewhat different and exploratory approach. We present an experimental attempt to seek the possible implications of our Ricci curvature for the stability of hydrodynamic flows, based on heuristic arguments. The discussions in this section should therefore be viewed as exploratory, without rigorous justifications.

Readers interested in extensions of Ricci curvature to various fluid settings may jump directly to the next sections.


\subsection{Nonlinear stability of the Euler flows}\label{sec:nonlinear-stability}
This subsection discusses the nonlinear stability of a class of stationary solutions to the incompressible Euler equation.
Following the geometric framework established by Arnold \cite{Arnol-khesin}, the continuous system preserves the total kinetic energy
\[
E = -\frac{1}{2}\int_{\mathbb{T}^2} \omega \psi  d^2\mathbf{x},
\]
as well as an infinite family of Casimir invariants of the form $C_n = \int_{\mathbb{T}^2} \omega^n \, d^2\mathbf{x}$ with  integers $n$.

Following \cite{Wir-Shep}, we consider a partition of vorticity data on $\TT^2$ into three mutually orthogonal components:
\begin{equation}\label{eq:vorticity_decomposition}
    \omega(\mathbf{x}, t) = \bar{\omega}(\mathbf{x}, t) + \tilde{\omega}(\mathbf{x}, t) + \omega'(\mathbf{x}, t).
\end{equation}
Here  $\bar{\omega}$ and $\tilde{\omega}$ are the $L^2$ projections of $\omega$ onto the lowest $x$-mode (\ie, ${\rm Span}\{\sin(x),\cos(x)\}$) and $y$-mode respectively.
The reminder $\omega'$ represents the higher frequency modes. 

Up to translations in $x$ and $y$, the total vorticity field is written exactly as:
\begin{equation}\label{eq:vorticity_profile}
    \omega(\mathbf{x}, t) = \sqrt{\frac{\bar{Z}(t)}{\pi^2}} \sin x + \sqrt{\frac{\tilde{Z}(t)}{\pi^2}} \sin y + \omega'(\mathbf{x}, t),
\end{equation}
where $\bar{Z}(t) = \frac{1}{2}\int \bar{\omega}^2 \, d^2\mathbf{x}$ and $\tilde{Z}(t) = \frac{1}{2}\int \tilde{\omega}^2 \, d^2\mathbf{x}$ are the enstrophies associated with the two orthogonal gravest modes. The total enstrophy is conserved and partitioned as $Z = \bar{Z} + \tilde{Z} + Z' = 1$, where $Z'$ is the enstrophy of the smaller scales.

For the vorticity \eqref{eq:vorticity_profile}, the corresponding energy and enstrophy components satisfy:
\begin{equation}\label{eq:energy_relations}
    \bar{E} = \bar{Z}, \quad \tilde{E} = \tilde{Z}, \quad \text{and} \quad E' \le \frac{1}{2}Z'
\end{equation}
where $\bar E= -\frac{1}{2} \int\bar \omega \Delta^{-1}\bar \omega d^2\mathbf{x} $, and $\tilde E, E'$ are defined in the same way.

By invoking Fj{\o}rtoft's  energy-enstrophy conservation argument on transfer across different scales \cite{fjortoft1953changes}, the authors of  \cite{Wir-Shep} established a strict two-sided bound on the small-scale enstrophy $Z'$ in terms of its initial value $Z'_0$:
\begin{equation}\label{eq:enstrophy_bound}
    \frac{1}{2}Z'_0 \le Z'(t) \le 2Z'_0.
\end{equation}
As a result, if  for the initial vorticity, $Z'_0$ is small, it will remain small for all the time. 

Using the relations between the energy and the enstrophies \eqref{eq:energy_relations} and \eqref{eq:enstrophy_bound}, one can show that the stationary solution
\begin{equation}\label{eq:vorticity_profile1}
    \omega(\mathbf{x}, t) =  \sin x +  \sin y 
\end{equation}
and more generally
\begin{equation}\label{eq:stationaryk=1}
    \omega(\mathbf{x}, t) =  A \sin (x+a) + B \sin (y+b) 
\end{equation}
with some $A,B,a,b\in \RR$ are nonlinear stable (Lyapunov stable) in $L^2$ norm \cite{arnold1966geometrie, elgindi2023remark}.

On the other hand, \cite[Theorem 12]{Dullin2016instability} showed using Zeitlin's discretization that stationary solutions 
\begin{equation}\label{eq:stationaryk>1}
    \omega(\mathbf{x}, t) =  A \sin (k_2 x+a) + B \sin (k_2 y+b) 
\end{equation}
with $k_2\neq \pm 1$ are all unstable. 

We argue that the nonlinear stability for $k_2=\pm 1$ and the instability for $k_2\neq \pm 1$ are numerically hinted in the Ricci curvature values.  \autoref{fig:Ric with different k} shows the transition of Ricci curvature values computed the adjoint Laplacian $\Delta^{\ad}$ for different choices of $k$. We notice that the Ricci curvature value for the lowest mode $\Ric(\tilde E_{0,1}, \tilde E_{0,1})$ stays nearly zero as $N$ varies. On the other hand, $\Ric(\tilde E_{0,k_2}, \tilde E_{0,k_2})$  for higher frequency modes ($k_2>1$) are decreasing in $N$. 
This suggests that $\Ric_N(\tilde E_{0,k_2},\tilde E_{0,k_2})< \Ric_N(\tilde E_{0,1},\tilde E_{0,1})$ as $N\to \infty$. 
We speculate the following asymptotics:
\begin{conjecture}
    Let $L=\Delta^{\ad}$. For $k_2>0$, we have 
    \begin{align}
        \lim_{N\to\infty}\Ric(\tilde E_{0,k_2+1},\tilde E_{0,k_2+1})< \lim_{N\to\infty} \Ric(\tilde E_{0,k_2},\tilde E_{0,k_2}).
    \end{align}
    We also have
    \begin{align}
        \lim_{N\to\infty}\Ric(\tilde E_{0,1},\tilde E_{0,1}) = \lim_{N\to\infty}\Ric(\tilde E_{1,1},\tilde E_{1,1}).
    \end{align}
\end{conjecture}
Note that the lowest modes are special as $\lim_{N\to\infty}\Ric(\tilde E_{0,n},\tilde E_{0,n}) \neq \lim_{N\to\infty}\Ric(\tilde E_{n,n},\tilde E_{n,n})$ for higher modes $n\neq \pm 1$. 

\begin{figure}[htbp]
\centering
\begin{minipage}{0.6\textwidth}  
 \includegraphics[width=1.0\textwidth]{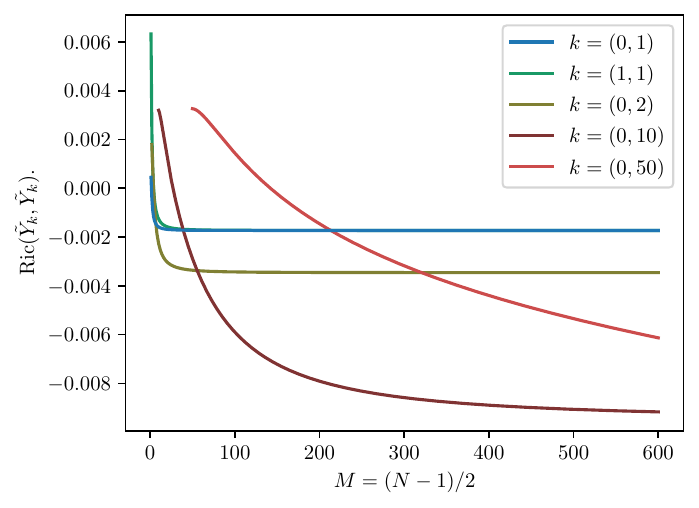}
 \end{minipage}

\caption{Transition of $\Ric(\tilde Y_k, \tilde Y_k)$ in $N$ for different $k$, highlighting the asymptotic behavior $\Ric(\tilde Y_{0,k_2},\tilde Y_{0,k_2})< \Ric(\tilde Y_{0,1},\tilde Y_{0,1})$ as $N\to \infty$ for $k_2> 1$. On the other hand, $\Ric(\tilde Y_{0,1},\tilde Y_{0,1})$ and $\Ric(\tilde Y_{1,1},\tilde Y_{1,1})$ are almost overlapping for large $N$.}

\label{fig:Ric with different k}
\end{figure}

\subsection{Tradewind current}
This subsection discusses the \emph{tradewind current}, proposed by Arnold \cite[Chapter IV, 4.A]{Arnol-khesin}. On the flat torus, it is modeled by the sinusoidal velocity field $\xi(x,y)=(\sin y,0)$, which serves as a simplified, idealized model of global atmospheric motion in geometric hydrodynamics arising from the Coriolis force and the temperature difference between the equator and the poles. This velocity field was originally proposed to investigate the stability of the Earth's weather system. Arnold then showed that the sectional curvature of planes containing the tradewind current is negative in many directions.

When the sectional curvature is negative, geodesics are expected to diverge exponentially. Based on this exponential growth of geodesic deviations, he concluded that reliable weather prediction on Earth is limited to a timescale of at most a few weeks.

Later, related results on the negativity of sectional curvature were obtained on the sphere by Lukatskii \cite{lukatskii1979curvature} and Yoshida \cite{yoshida1997riemannian}, and in the presence of the Coriolis force by Suri \cite{suri2024curvature}. These works considered sectional curvature averaged over different directions.



With our definition of Ricci curvature, we  revisit the tradewind current on $\TT^2$.
First let us review Arnold's heuristic argument regarding the error growth of weather forecasts \cite[Chapter IV, 4.B]{Arnol-khesin}. 
That is, if the initial state of the weather is known with a small error $\epsilon$, then the prediction error after $n$ months grows like $10^{2.5c n}\epsilon$. Here the parameter $c$ is given by
\begin{align}\label{eq:c and C_0}
    c=\sqrt{\frac{-1}{8\pi^2C_0}}, \quad C_0 = K_{\HDiff(\TT^2)}(\nabla^\perp \cos(x), \nabla^\perp \cos(y))  =-\frac{1}{8\pi^2}\approx -0.0126,
\end{align}
using the sectional curvature for lowest modes $\xi=\nabla^\perp\cos(y)$ and  $\nabla^\perp \cos(x)$. He referred to $C_0$ as ``mean curvature'', with the hope that it captures the  curvature averaged in all the directions in $\hdiff(\TT^2)$.

We here suggest an alternative choice of $C_0$ by replacing the sectional curvature with the Ricci curvature 
\begin{align}
    \Ric_\infty (\cos(x),  \cos(x))=\lim_{N\to\infty} \Ric_N(X_{1,0}, X_{1,0})\approx - 0.03415,
\end{align}
and its normalized version
\begin{align}
\Ric_\infty^{\rm norm} (\cos(x),  \cos(x))=\lim_{N\to\infty} \Ric_N(\tilde X_{1,0}, \tilde X_{1,0})\approx -0.00173, \qquad \tilde X_{1,0}=\sqrt{\frac{1}{-2\pi^2 \lambda_k}}X_{1,0}.
\end{align}
The Ricci curvature values here are computed via the expression in \autoref{conj:Ricci convergence} with the Laplacian $\Delta^{\ad}$. 

This replacement changes the parameter $c$ in \eqref{eq:c and C_0} to
\[
c_{\mathrm{unnorm}}=\sqrt{\frac{-1}{8\pi^2\Ric_\infty(\cos(x),  \cos(x)) \big)}}\approx 0.37,\quad
c_{\mathrm{norm}}=\sqrt{\frac{-1}{8\pi^2\Ric_\infty^{\rm norm}(\cos(x),  \cos(x)) \big)}}\approx 7.32.
\]

Consequently, the corresponding growth estimates are
$10^{2.5\times \sqrt{0.37}n}\varepsilon=10^{1.5n}\varepsilon$
and
$10^{2.5\times \sqrt{7.32} n}\varepsilon=10^{6.70 n}\varepsilon$, 
As a result, after two months ($n=2$), our choice of $C_0$ as $\Ric_\infty (\cos(x),  \cos(x))$ and $\Ric_\infty^{\rm norm} (\cos(x),  \cos(x))$ amplify the error $\varepsilon$ by a factor of $10^3$ and $10^{14.5}$ respectively, whereas Arnold's result yields an amplification factor of $10^5$. At this point, we are unaware which of the unnormalized or normalized Ricci curvature value is a more plausible choice. 

As a future direction, it would be interesting to extend this approach to other surfaces, such as the two-dimensional sphere, where Zeitlin's model with a canonical choice of discrete Laplacian, the Hoppe--Yau Laplacian \cite{hoppe-yau1998some}, is available.


\section{Ricci curvature with respect to  Sobolev metrics}\label{sec:Ricci Sobolev}
Our notion of Ricci curvature extends to a wide variety of settings. To showcase this, we define Ricci curvature on the state spaces of several fluids in the present and following sections. 
We begin with an extension to higher-order metrics and the resulting geodesic flow.

The $H^1$-Sobolev metrics on the Hamiltonian diffeomorphisms $\HDiff(\TT^2)$ is given as 
\begin{align}
    g(\nabla^\perp u , \nabla^\perp v)
    = \int_{\TT^2} \nabla u\cdot \nabla v + \alpha \nabla^2 u: \nabla^2  v \, d{\bf{x}}
    = \int - u(\Delta - \alpha \Delta^2)v\, d{\bf{x}} 
\end{align}
with some parameter $\alpha>0$. The Lagrangian Averaged Euler LAE-$\alpha$ equations is the geodesic flow with respect to the $H^1$-metric \cite{shkoller2000analysis}. 

As in the continuous case, one can obtain Zeitlin-type discretizations of the LAE-$\alpha$ equations on $\SU(N)$ by replacing the discrete Laplacian operator  $L_N$ with the higher-order one $L^\alpha_N\coloneqq \Delta_N-\alpha\Delta^2_N$. This approximates LAE-$\alpha$ equations as geodesic flows on $\SU(N)$ with respect to the Zeitlin metric (\autoref{def:Zeitlin metric}) with $L_N^\alpha$. 

Consequently, our definition of Ricci curvature (\autoref{def:Ricci}) extends to the higher-order setting. Our formula (\autoref{th:Ricci curvature formula}) remains unchanged by design. It suffices to use the eigenvalue $\lambda_k+\alpha \lambda_k$ of $L^\alpha$ instead of  $\lambda_k$ of the discrete Laplacian $\Delta_N$.

\paragraph{Numerical examples}
\autoref{fig:Ricci higher order} numerically compares the results with respect to the standard metric using the adjoint Laplacian $L=\Delta^{\ad}$ and the $H^1$ Laplacian $L^\alpha=\Delta^{\ad}-\alpha (\Delta^{\ad})^2$ with a small $\alpha$ value.  The $H^1$ Laplacian shows a strong smoothing effect in high frequency modes, while strengthening the negative values of low frequency modes.

\begin{figure}[htbp]
\begin{minipage}{0.45\textwidth}
    \centering
    On $\SU(101)$
\end{minipage}
\begin{minipage}{0.5\textwidth}
    \centering
    Transition in $N$ 
\end{minipage}

\centering
 \caption*{$L=\Delta^{\ad}$} 
 \vspace{-15pt}
\begin{minipage}{0.45\textwidth}  
 \includegraphics[width=1.0\textwidth]{images/ad_M50_2d.pdf}\end{minipage}
\begin{minipage}{0.5\textwidth}  
\includegraphics[width=1.0\textwidth]{images/adM100_labels.pdf}
 \end{minipage}

  \caption*{$L^\alpha=\Delta^{\ad}-\alpha (\Delta^{\ad})^2$} 
   \vspace{-15pt}
  \begin{minipage}{0.45\textwidth}  
\includegraphics[width=1.0\textwidth]{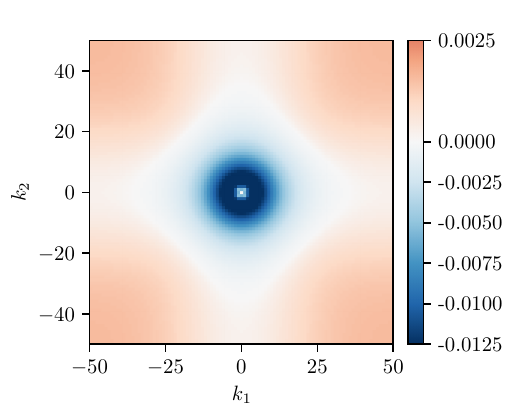}\end{minipage}
\begin{minipage}{0.5\textwidth} 
\includegraphics[width=1.0\textwidth]{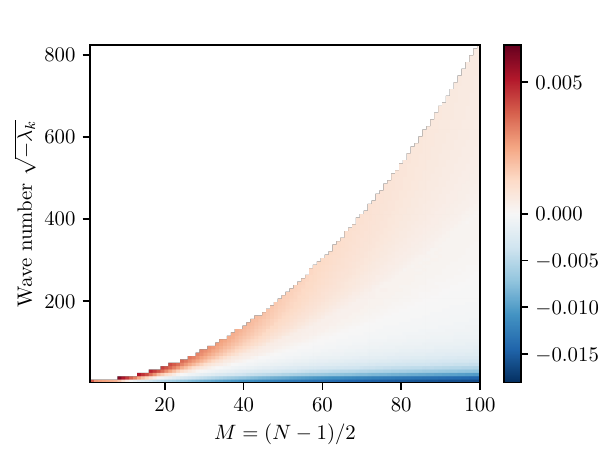} 
\end{minipage}
\caption{Ricci curvature computed with the adjoint Laplacian $\Delta^{\ad}$ (top), and $H^1$ Laplacian $L=\Delta^{\ad}-\alpha (\Delta^{\ad})^2$ with $\alpha=0.01$ (bottom).  }
\label{fig:Ricci higher order}
\end{figure}

\section{Ricci curvature for  fluids on rectangular domains}\label{sec:Ricc_Rect}
Another interesting extension is to rectangular tori $\mathbb{T}_\alpha^2 = [0, 2\pi/\alpha) \times [0, 2\pi)$ for $\alpha > 0$, where fluids tend to be more unstable \cite{Wir-Shep, drivas2022conjugate}, \eg, the initial vorticity data $\sin(\alpha nx)+\sin(ny)$ $(n\in \ZZ)$ is not even a stationary solution, unlike the case of $\alpha=1$ discussed in \autoref{sec:nonlinear-stability}.

On $\mathbb{T}_\alpha^2$, the Lie algebra structure constants scale linearly with $\alpha$, while the Laplacian eigenvalues undergo an anisotropic scaling $\lambda_k = -(\alpha^2 k_1^2 + k_2^2)$. This deformation directly alters the coadjoint operator, the covariant derivative, and the curvature tensor.

On $C^\infty_0(\mathbb{T}_\alpha^2,\CC)$, the Fourier basis elements are 
    \[
    e_k(x_1, x_2) = e^{i (\alpha k_1 x_1 + k_2 x_2)}
    \]
and the continuous Poisson bracket is given by
   \[
    \{e_p, e_q\} = -\alpha (p \times q) \, e_{p+q}, \quad \text{where } p \times q \coloneqq p_1 q_2 - p_2 q_1.
  \]  
As a result, the continuous structure constants are
   \[ C_{p,q}^{p+q} = -\alpha (p \times q)\]
and for the Laplacian we have
    \[ \Delta_\alpha e_k = -(\alpha^2 k_1^2 + k_2^2) e_k. \]
Consequently, the rescaled discrete commutator bracket is
    \[
    [T_p, T_q]_\alpha \coloneqq \alpha [T_p, T_q] = -\alpha \frac{N}{\pi} \sin\left(\frac{\pi}{N} (p \times q)\right) T_{p+q} \pmod{Z_N}.
    \]
and the discrete \emph{adjoint $\alpha$-Laplacian} $\Delta_\alpha^{\ad}$ is
  \begin{equation}\label{eq:adjoint_Laplacian_alpha}
    \Delta_\alpha^{\ad} w \coloneqq \alpha^2 [ T_{0,-1}, [T_{0,1}, w]] + [ T_{-1,0}, [T_{1,0}, w]].
  \end{equation}
The basis matrices $T_k$ ($k \in Z_N$) are eigenvectors of $\Delta_\alpha^{\ad}$ with eigenvalues
  \begin{equation}\label{eq:adjoint_L_alpha_eigenvalues}
    \lambda_k^\alpha := -\left(\frac{N}{\pi}\right)^2 \left( \alpha^2 \sin^2\left(\frac{\pi k_1}{N}\right) + \sin^2\left(\frac{\pi k_2}{N}\right) \right).
  \end{equation}
For the  real basis elements $\{X_k, Y_k\}_{k \in Z_N^+}$ we have
 \begin{equation*}
        \Delta_\alpha^{\ad} X_k = \Delta_\alpha^{\ad} \left(\frac{1}{2}(T_k + T_{-k})\right) = \frac{1}{2} \left( \lambda_k^\alpha T_k + \lambda_k^\alpha T_{-k} \right) = \lambda_k^\alpha X_k
    \end{equation*}
and the same applies to $Y_k$.

 Let $L = \Delta_\alpha^{\ad}$ be the adjoint $\alpha$-Laplacian. The norm of the real basis elements $X_k, Y_k \in \su(N)$ under Zeitlin's metric $\langle \cdot, \cdot \rangle^L$ is given by:
    \begin{equation}\label{eq:Zeitlin_metric_Xk}
        \langle X_k, X_k \rangle^L = \langle Y_k, Y_k \rangle^L = 
         -2\pi^2 \lambda_k^\alpha.
    \end{equation}
The coadjoint operators become
    \begin{align*}
        \ad_{X_k}^\star X_l &= \frac{1}{2}\alpha\frac{N}{\pi} \sin\left(\frac{\pi}{N} k \times l\right) \lambda_l^\alpha \left( \frac{X_{k+l}}{\lambda_{k+l}^\alpha} - \frac{X_{k-l}}{\lambda_{k-l}^\alpha} \right), \\[6pt]
        \ad_{Y_k}^\star Y_l &= -\frac{1}{2}\alpha\frac{N}{\pi} \sin\left(\frac{\pi}{N} k \times l\right) \lambda_l^\alpha \left( \frac{X_{k+l}}{\lambda_{k+l}^\alpha} + \frac{X_{k-l}}{\lambda_{k-l}^\alpha} \right), \\[6pt]
        \ad_{X_k}^\star Y_l &= \frac{1}{2}\alpha\frac{N}{\pi} \sin\left(\frac{\pi}{N} k \times l\right) \lambda_l^\alpha \left( \frac{Y_{k+l}}{\lambda_{k+l}^\alpha} + \frac{Y_{k-l}}{\lambda_{k-l}^\alpha} \right), \\[6pt]
        \ad_{Y_k}^\star X_l &= \frac{1}{2}\alpha\frac{N}{\pi} \sin\left(\frac{\pi}{N} k \times l\right) \lambda_l^\alpha \left( \frac{Y_{k+l}}{\lambda_{k+l}^\alpha} - \frac{Y_{k-l}}{\lambda_{k-l}^\alpha} \right).
    \end{align*}
which imply following formulas for the covariant derivatives:
    \begin{align*}
        2\nabla_{X_k} X_l &= \alpha \frac{N}{2\pi}\sin\left(\frac{\pi}{N}k\times l\right) \left[ -\left(1+\frac{\lambda_l^\alpha-\lambda_k^\alpha}{\lambda_{k+l}^\alpha}\right) X_{k+l} + \left(1+\frac{\lambda_l^\alpha-\lambda_k^\alpha}{\lambda_{k-l}^\alpha}\right) X_{k-l} \right], \\[6pt]
        2\nabla_{Y_k} Y_l &= \alpha \frac{N}{2\pi}\sin\left(\frac{\pi}{N}k\times l\right) \left[ \left(1+\frac{\lambda_l^\alpha-\lambda_k^\alpha}{\lambda_{k+l}^\alpha}\right) X_{k+l} + \left(1+\frac{\lambda_l^\alpha-\lambda_k^\alpha}{\lambda_{k-l}^\alpha}\right) X_{k-l} \right], \\[6pt]
        2\nabla_{X_k} Y_l &= \alpha \frac{N}{2\pi}\sin\left(\frac{\pi}{N}k\times l\right) \left[ -\left(1+\frac{\lambda_l^\alpha-\lambda_k^\alpha}{\lambda_{k+l}^\alpha}\right) Y_{k+l} - \left(1+\frac{\lambda_l^\alpha-\lambda_k^\alpha}{\lambda_{k-l}^\alpha}\right) Y_{k-l} \right], \\[6pt]
        2\nabla_{Y_k} X_l &= \alpha \frac{N}{2\pi}\sin\left(\frac{\pi}{N}k\times l\right) \left[ -\left(1+\frac{\lambda_l^\alpha-\lambda_k^\alpha}{\lambda_{k+l}^\alpha}\right) Y_{k+l} + \left(1+\frac{\lambda_l^\alpha-\lambda_k^\alpha}{\lambda_{k-l}^\alpha}\right) Y_{k-l} \right].
    \end{align*}
As a result we have the following Ricci formula for the rectangular domain.
\begin{proposition}[Ricci curvature for general $\alpha$]\label{prop:Ricci_curvature_alpha}
For $E_k \in \{X_k, Y_k\}$, the Ricci curvature under Zeitlin's metric $\langle \cdot, \cdot \rangle^L$ with $L = \Delta_\alpha^{\ad}$ on the rectangular torus $\mathbb{T}_\alpha^2$ is given by:
\begin{align}
    (N^2-1)\Ric(E_k,E_k)
    &=\sum_{E_i} \Bigg( {\langle\nabla_{E_i}E_k,\nabla_{E_k} E_i \rangle^L \over \langle E_i, E_i \rangle^L}
    +\frac{1}{2} {\langle [E_i, E_k]_\alpha, [E_k, E_i]_\alpha\rangle^L \over \langle E_i, E_i \rangle^L}
    \nonumber \\
    &\quad +\frac{1}{2} {\langle [E_i, E_k]_\alpha, \ad^\star_{E_k}E_i \rangle^L \over \langle E_i, E_i \rangle^L}
    -\frac{1}{2}{\langle [E_i, E_k]_\alpha, \ad^\star_{E_i}E_k \rangle^L \over \langle E_i, E_i \rangle^L}\Bigg)
    \nonumber \\
    \nonumber \\
    &= \alpha^2 \sum_{i\in Z_N^+} c_{ik}^2 \left( \frac{ -3(\lambda_{k+i}^\alpha + \lambda_{k-i}^\alpha) + (\lambda_i^\alpha - \lambda_k^\alpha)^2 \left( (\lambda_{k+i}^\alpha)^{-1} + (\lambda_{k-i}^\alpha)^{-1} \right) + 4(\lambda_k^\alpha + \lambda_i^\alpha) }{8\lambda_{i}^\alpha} \right), \label{eq:Ricci_computation_alpha}
\end{align}
where $c_{ik} = \frac{N}{\pi} \sin\left(\frac{\pi}{N} i \times k\right)$, $E_i$ runs through the basis $\{X_i, Y_i\}_i$, and $\lambda_p^\alpha$ are the eigenvalues of the adjoint $\alpha$-Laplacian $\Delta_\alpha^{\ad}$ \eqref{eq:adjoint_Laplacian_alpha}.
\end{proposition}

\paragraph{Numerical examples}
\autoref{fig:Ricci rectanguler} shows numerical results for the Ricci curvature on the standard torus and rectangular tori. These results illustrate that the Ricci curvature of low-frequency modes becomes anisotropic (see the left column).

\begin{figure}[htbp]
\centering
\begin{minipage}{0.45\textwidth}
    \centering
    On $\SU(101)$
\end{minipage}
\begin{minipage}{0.5\textwidth}
    \centering
    Transition in $N$ 
\end{minipage}
 \caption*{$L=\Delta^{\ad}$ on $\TT^2$} 
 \vspace{-15pt}
\begin{minipage}{0.45\textwidth}  
 \includegraphics[width=1.0\textwidth]{images/ad_M50_2d.pdf}\end{minipage}
 \begin{minipage}{0.5\textwidth}  
 \includegraphics[width=1.0\textwidth]{images/adM100_labels.pdf}\end{minipage}

 \caption*{$L=\Delta^{\ad}_2$ on $\TT^2_2$} 
  \vspace{-15pt}
 \begin{minipage}{0.45\textwidth}  
 \includegraphics[width=1.0\textwidth]{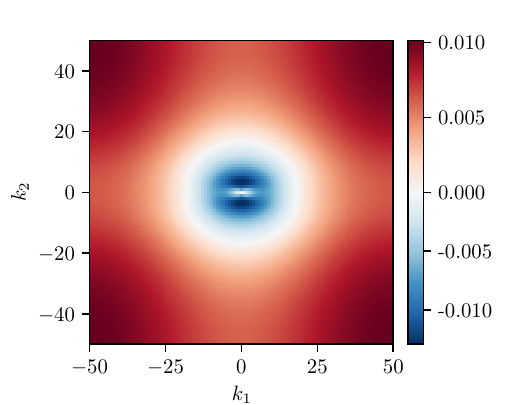}\end{minipage}
\begin{minipage}{0.5\textwidth} 
\includegraphics[width=1.0\textwidth]{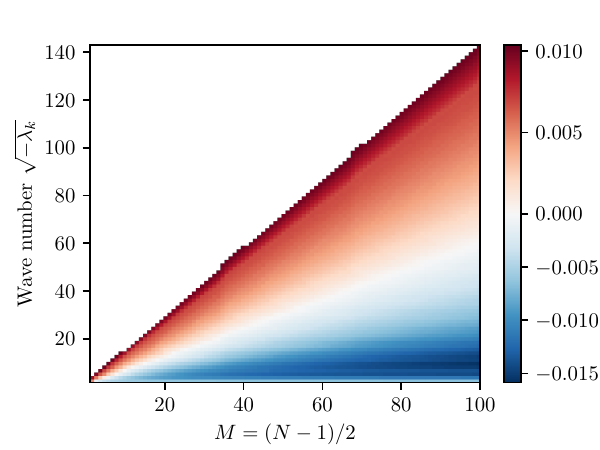} 
\end{minipage}

 \caption*{$L=\Delta^{\ad}_{1/2}$ on $\TT^2_{1/2}$} 
  \vspace{-15pt}
 \begin{minipage}{0.45\textwidth}  
 \includegraphics[width=1.0\textwidth]{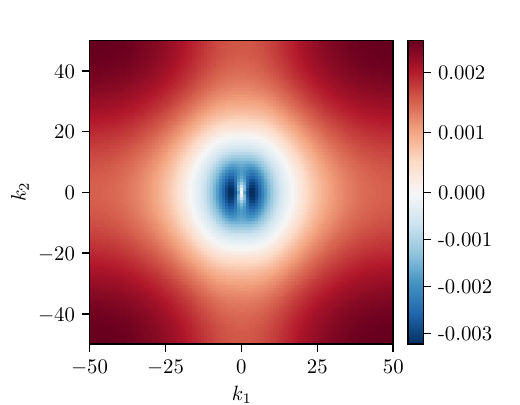}\end{minipage}
\begin{minipage}{0.5\textwidth} 
\includegraphics[width=1.0\textwidth]{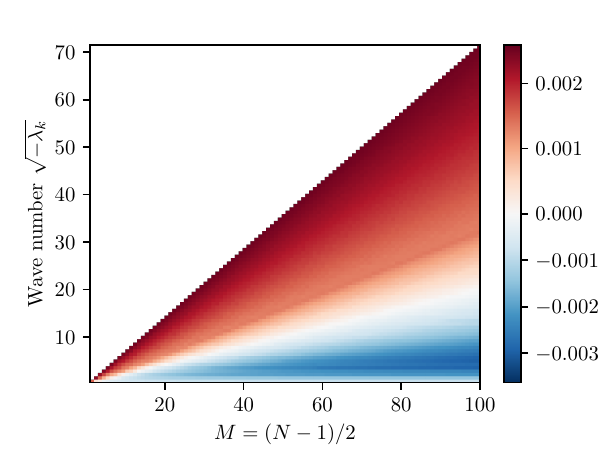} 
\end{minipage}
\caption{
Ricci curvature computed on the standard torus $\TT^2$ (top), and rectangular tori $\TT^2_2$ (middle) and  $\TT^2_{1/2}$ (bottom) with their corresponding Laplacians.
 }
\label{fig:Ricci rectanguler}
\end{figure}

\begin{remark}
As illustrated in \autoref{fig:Ricci extreme alpha},  larger or smaller values of $\alpha$ render the Ricci curvature distribution more anisotropic. 
Analyzing the Ricci curvature across varying aspect ratios $\alpha$ particularly in the asymptotic limits ($\alpha \to 0$ and $\alpha \to \infty$), combined with stability analysis (\autoref{sec:nonlinear-stability}) and rotational effects via the Coriolis force (\autoref{sec:Coriolis-force}), would promise a framework for validating the geometric role of Ricci curvature in hydrodynamic stability. 
\begin{figure}[htbp]
\centering

 \begin{minipage}{0.45\textwidth}  
  \caption*{$L=\Delta^{\ad}_{100}$ on $\TT^2_{100}$} 
    \vspace{-10pt}
\includegraphics[width=1.0\textwidth]{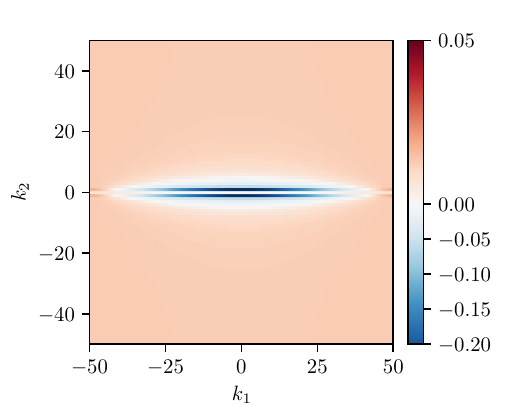}\end{minipage}
\begin{minipage}{0.45\textwidth} 
 \caption*{$L=\Delta^{\ad}_{1/100}$ on $\TT^2_{1/100}$} 
   \vspace{-10pt}
\includegraphics[width=1.0\textwidth]{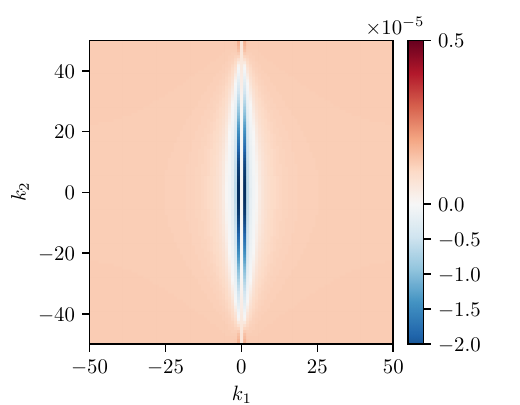} 
\end{minipage}
\caption{
 Ricci curvature computed on thin tori $\TT^2_{100}$ and $\TT^2_{1/100}$, exhibiting very high anisotropies.
 }
\label{fig:Ricci extreme alpha}
\end{figure}

\end{remark}

\section{Ricci curvature with the Coriolis force}\label{sec:Coriolis-force}

Our notion of Ricci curvature can be further extended to incorporate the Coriolis effect. More precisely, we present a definition of Ricci curvature on a certain extension of $\HDiff(\TT^2)$, which serves as the configuration space for  \emph{the quasi-geostrophic equation},
\begin{equation}\label{eq:qg on torus}
\partial_t\Delta\psi+  \{\psi, \Delta\psi \}   +  \beta\partial_x\psi   =0.
\end{equation}
This equation incorporates the forcing term induced by the Coriolis force under the $\beta$-plane approximation into the  vorticity form of the Euler equation $\partial_t \omega +\{\Delta^{-1}\omega,\omega\}=0$ where $\omega=\Delta \psi$. 

The quasi-geostrophic equation \eqref{eq:qg on torus} is the geodesic equation on a central extension of the group of Hamiltonian diffeomorphisms on the torus \cite{vizman2008cocycles}. We develop a Zeitlin-type discrete analogue of this setup, which plays a central role in defining a Ricci curvature.

To this end, we begin by briefly reviewing the continuous theory. Let us consider an extension of the Lie algebra $\frak g\coloneqq \hdiff(\TT^2)$ consisting of the Hamiltonian vector fields on $\TT^2$. Define a central extension $\widehat{\frak g}=\frak g\rtimes_{\omega}\mathbb{R}\subseteq \frak g\times\mathbb{R}$ with  \emph{the Roger cocycle} $\omega$ given by
 \[
  \omega:{\mathfrak{g}} \times {\mathfrak{g}} \longrightarrow \mathbb{R}\quad;\quad (\nabla^\perp f,\nabla^\perp g)\longmapsto \int_{\TT^2}f\alpha(\nabla^\perp g)dx dy,
  \]
  where $\alpha=\beta dy$ is the differential 1-form on $\TT^2$ with a constant $\beta$.
The Lie bracket on $\widehat{\frak g}$ is 
\begin{eqnarray}
    [(\nabla^\perp f,a),(\nabla^\perp g,b)]_{\widehat {\frak g}}=([\nabla^\perp f,\nabla^\perp g]_{\diff(\TT^2)}~,~\omega(\nabla^\perp f,\nabla^\perp g) ).
  \end{eqnarray}
Details are explained in \cite{vizman2008cocycles}.  In \cite{suri2025stochastic}, the Roger cocycle $\omega$ was proven to be integrable. Namely, there exists a central extension group $\widehat{\HDiff(\TT^2)}$ whose Lie algebra is $\widehat{\frak g}$. 

Let us define a Riemannian metric on $\widehat{\HDiff(\TT^2)}$ by
    \[ 
    \ll (u,a),(v,b)\gg:=\langle u,v\rangle+ab
    \]
with the $L^2$ inner product  $\langle\cdot,\cdot\rangle$ on $\SDiff(\TT^2)$. Then the quasi-geostrophic \eqref{eq:qg on torus} arises as the geodesic equation on $(\widehat{\HDiff(\TT^2)},\ll,\gg)$.

\subsection{Zeitlin's metric with the Coriolis force}
We now build a Zeitlin-like discrete counterpart of $(\widehat{\HDiff(\TT^2)},\ll,\gg)$.
%
First notice that the Roger cocycle  $\omega$ can be expressed in terms of the $L^2$ metric on $\SDiff(\TT^2)$. 
Let us consider the operator 
\[
T:\mathfrak{g}\longrightarrow\mathfrak{g}
\]
defined by the relation 
$\langle T\nabla^\perp f,\nabla^\perp g\rangle=\omega(\nabla^\perp f,\nabla^\perp g)$. The operator $T$ for the basis $\{\cos(k\cdot \bx), \sin(k\cdot \bx)\}_k$ is computed explicitly in \cite[Lemma 5.3]{suri2025stochastic}: we have
\begin{eqnarray*}
T(\nabla^{\perp}\sin (k\cdot\bf x))&=&\frac{-\beta k_1}{|k|^2}  \nabla^{\perp}\cos (k\cdot\bf x),
\\ 
T(\nabla^{\perp}\cos (k\cdot\bf x))&=&\frac{\beta k_1}{|k|^2} \nabla^{\perp}\sin (k\cdot\bf x).
\end{eqnarray*}
Note that $T$ acts on vector fields, while Zeitlin's quantization map $P_N$ acts on functions $C^\infty_0(\TT^2)$. To handle this, we consider the induced operator $T_0:C^\infty_0(\TT^2)\longrightarrow C^\infty_0(\TT^2)$ 
defined so that the following diagram commutes.
\[
\begin{tikzcd}
\mathfrak{g} \arrow{r}{T} &
\mathfrak{g} \\
C^\infty_0(\TT^2) \arrow{r}{{T_0}}
\arrow{u}{\nabla^{\perp}} &
C^\infty_0(\TT^2)
\arrow{u}[swap]{\nabla^{\perp}}
\end{tikzcd}
\]
As a result
\begin{eqnarray}
{T_0}(\sin (k\cdot{\bf x}))= \frac{-\beta k_1}{|k|^2}  \cos (k\cdot{\bf x})
~~\text{and}~~
{T_0}(\cos (k\cdot{\bf x}))=\frac{\beta k_1}{|k|^2} \sin (k{\cdot\bf x})
\end{eqnarray}
The quantized version of $T_0$ is defined by
\begin{eqnarray*}
T_NX_k:=   \frac{-\beta}{-\lambda_k}\frac{N}{\pi}\sin(\frac{\pi k_1}{N})S_{N,k}  Y_k
~~\text{and}~~
T_NY_k := \frac{\beta}{-\lambda_k} \frac{N}{\pi}\sin(\frac{\pi k_1}{N}) S_{N,k}X_k
\end{eqnarray*}
where the normalizing constant $S_{N,k}$ is computed as follows. 
The coefficient $S_{N,k}$ is chosen such that 
\begin{align}
    \lim_{N\rightarrow \infty}\langle T_N X_k,Y_k\rangle^L=\langle \nabla^{\perp}T_0\cos(k\cdot{\bf x}),\nabla^{\perp}\sin(k\cdot{\bf x})\rangle.
\end{align}
For this, compute
\begin{align}\label{eq:SNK1}
\langle \nabla^{\perp}T_0\cos(k\cdot{\bf x}),\nabla^{\perp}\sin(k\cdot{\bf x})\rangle\nonumber
&=\langle T\nabla^{\perp}\cos(k\cdot{\bf x}),\nabla^{\perp}\sin(k\cdot{\bf x})\rangle\nonumber
\\
&=\omega(\nabla^{\perp}\cos(k\cdot{\bf x}),\nabla^{\perp}\sin(k\cdot{\bf x}))\nonumber
\\
&= \int_{\TT^2}\cos(k\cdot{\bf x})\alpha(\nabla^{\perp}\sin(k\cdot{\bf x})) d\bf x\nonumber
\\
&= \int_{\TT^2}\cos(k\cdot{\bf x})\beta k_1\cos(k\cdot{\bf x}) d\bf x\nonumber
\\
&=\beta k_1 \int_{\TT^2}\cos^2(k\cdot{\bf x}) d\bf x\nonumber
\\
&=\frac{\beta k_1}{2} \int_{\TT^2}1+\cos(2k\cdot{\bf x}) d{\bf x}
=2\pi^2\beta k_1.
\end{align}
On the other hand  we have
\begin{align}\label{eq:SNK2}
    \langle T_N X_k,Y_k\rangle^L
    &=\langle \frac{\beta}{\lambda_k}\frac{N}{\pi}\sin(\frac{\pi k_1}{N})S_{N,k}  Y_k , Y_k\rangle^L\nonumber
    \\
    &= \frac{\beta}{\lambda_k}\frac{N}{\pi}\sin(\frac{\pi k_1}{N})S_{N,k}  \langle Y_k , Y_k\rangle^L\nonumber
    \\
    &= \frac{\beta}{\lambda_k}\frac{N}{\pi}\sin(\frac{\pi k_1}{N})S_{N,k}  (-2\pi^2\lambda_k)\nonumber
    \\
    &= (-2\pi^2\beta)\frac{N}{\pi}\sin(\frac{\pi k_1}{N})S_{N,k}.  
\end{align}
Comparing \eqref{eq:SNK1} and \eqref{eq:SNK2} we get $S_{N,k}=-1$. 
Hence the linear operator $T_N:\su(N)\longrightarrow\su(N)$  is written as
    \begin{align}\label{eq:TN}
     T_NX_k:=   \frac{\beta}{-\lambda_k}\frac{N}{\pi}\sin(\frac{\pi k_1}{N})  Y_k
    ~~\text{and}~~
     T_NY_k := \frac{\beta}{\lambda_k} \frac{N}{\pi}\sin(\frac{\pi k_1}{N}) X_k.
    \end{align}

With this we define the discrete counterpart of the Roger cocycle $\omega$:
\begin{definition}[Roger cocycle on $\su(N)$]\label{def:discrte cocycle}
    The discrete Roger cocycle $\omega_N:\su(N)\times\su(N)\longrightarrow\mathbb{R}$ is defined by
    \[
    \omega_N(X,Y):=\langle T_NX,Y\rangle^L
    \]
    where $\langle \cdot,\cdot\rangle^L$ is the Zeitlin metric in \autoref{def:Zeitlin metric} with a chosen Laplacian $L:\su(N)\to\su(N)$. 
\end{definition}
When it is clear from context, we may drop the subscript $N$ and just write $\omega$ and $T$.

Using the discrete Roger cocycle,
we define a central extension of $\su(N)$ as a discrete counterpart of $\widehat{\frak g}=\frak g\rtimes_\omega \RR$:
\begin{definition}[$\widehat{\su(N)}$ and $\widehat{\SU(N)}$]
Define the central extension \[
\widehat{\su(N)}:=\su(N)\rtimes_{\omega_N} \mathbb{R}.
\] 
The Lie bracket on $\widehat{\su(N)}$ is
\[
[(u,a),(v,b)]_{\widehat{\su(N)}}:=([u,v], \omega_N(u,v)).
\] 
The Lie algebra $\widehat{\su(N)}$ integrates to a $N^2$-dimensional Lie group $\widehat{\SU(N)}$.
\end{definition}
 In fact,  $\widehat{\SU(N)}\cong \SU(N)\times \RR$ algebraically  since $H^2(\frak{su}(N),\RR)=0$.

\begin{definition}[Zeitlin's metric on $\widehat{\SU}(N)$]
    For $(u,a),(v,b)\in\widehat{\su(N)}$, we define a Riemannian metric
\begin{equation}
    \ll (u,a),(v,b)\gg^L:=\langle u,v\rangle^L+ab =-\hbar_N\Tr(u^\dag\Delta_N v)+ab
\end{equation}
using $\hbar_N=\frac{2^4\pi^4}{N^3}$ and Zeitlin's metric $\langle \cdot, \cdot \rangle^L$ with a suitable Laplacian operator $L$ on $\su(N)$ (See \autoref{def:Zeitlin metric}).
\end{definition}

We have now all the ingredients to define a Ricci curvature on $\widehat{\SU(N)}$.
\begin{definition}[Ricci curvature on $\widehat{\SU(N)}$]\label{def:Ricci on extension}
The Ricci curvature tensor $\widehat{\Ric_N} \in \Gamma (T^*\widehat{\SU(N)}\otimes T^*\widehat{\SU(N)})$ is defined by
\begin{eqnarray}\label{eq:Ricci def extensionn}
\widehat{\Ric_N}(\widehat{E_k},\widehat{E_l})=
\frac{1}{N^2}\sum_{\widehat{E_i}}{\ll  \widehat{R}(\widehat{E_i},\widehat{E_k})\widehat{E_l}, \widehat{E_i}  \gg \over \ll \widehat{E_i} , \widehat{E_i}  \gg}
\end{eqnarray}
where $\widehat{E_k},\widehat{E_l}$ are basis elements in $\{(X_i,0),(Y_i,0), (0,1)\}_i$ and $\widehat{E_i}$ runs through all of them.
\end{definition}

\subsection{Computation of $\widehat{\Ric}$}
Similarly to \autoref{th:Ricci curvature formula} on $\SU(N)$, we derive a formula for $\widehat{\Ric}$ on $\widehat{\SU(N)}$. We do so specifically for the diagonal entries $\widehat{\Ric}((E_k,1),(E_k,1))$ for $E_k\in\{X_i,Y_i\}_i$. This is because these diagonal entries are particularly relevant to the quasi-geostrophic equation \eqref{eq:qg on torus}, as the evolution of the state with the initial data $\psi_0$ is constrained to the coadjoint orbit  containing $(\psi_0, 1)\in \widehat{\frak g}$.

To evaluate the summand in \eqref{eq:Ricci def extensionn}, we compute each term of
\begin{eqnarray}\label{eq:curv-centra-ext}
     \ll  \widehat{R}(\widehat{E_i},\widehat{E_k})\widehat{E_l}  ,   \widehat{E_i}  \gg  &=&      
     \ll\widehat{\nabla}_{\widehat{E_i}}\widehat{E_k},\widehat{\nabla}_{\widehat{E_l}} \widehat{E_i}   \gg   
     -
     \ll  \widehat{\nabla}_{\widehat{E_k}}\widehat{E_l}    ,    {\widehat{\nabla}_{\widehat{E_i}}\widehat{E_i}}  \gg \nonumber
     \\
     &&+     \frac{1}{2}\ll [\widehat{E_i},\widehat{E_k}]   ,  [\widehat{E_l},\widehat{E_i}]   + \widehat{\ad}^\star_{\widehat{E_l}}\widehat{E_i}   -   \widehat{\ad}^\star_{\widehat{E_i}}\widehat{E_l}    \gg
    \end{eqnarray}
    just as we did in \autoref{sec:Ricci on suN}. Here $\widehat{R}, \widehat{\nabla}, \widehat{\ad}^\star$ are the the Riemann curvature tensor, the covariant derivative, and the coadjoint operator on $(\widehat{SU(N)},\ll,\gg)$, which we give explicitly.


The adjoint action on $\widehat{\su(N)}$ is 
\begin{equation}
    \widehat{\ad}_{(u,a)}(v,b)=\big( {\ad}_u v, \omega_N(u,v)  \big).
\end{equation}
Consequently, for every $(w,d)\in\widehat{\su(N)}$ we have
\begin{eqnarray*}
    \ll\widehat{\ad}_{(u,a)}^\star(v,b),(w,d)\gg 
    &=& \ll(v,b),\tilde{\ad}_{(u,a)}(w,d)\gg\nonumber
    \\
    &=&\ll(v,b),({\ad}_{u}w ,\omega_N(u,w))\gg\nonumber
    \\
    &=&\langle v,{\ad}_{u}w \rangle + b\omega_N(u,w))\nonumber
    \\
    &=&\langle {\ad}_{u}^\star v,w \rangle + b\langle{T}_N u,w\rangle\nonumber
    \\
    &=&\langle {\ad}_{u}^\star v+b{T}_N u ,  w\rangle\nonumber
    \\
    &=& \ll \big( {\ad}_u^\star v +b{T}_Nu, 0  \big) ,(w,d)\gg
\end{eqnarray*}
which implies that
     \begin{align}\label{eq:ad star on central extensioin}
    \widehat{\ad}_{(u,a)}^\star(v,b)= \big( {\ad}_u^\star v +bT_Nu, 0  \big).
    \end{align}

\noindent
The covariant derivative is computed by the same argument as that used in \autoref{sec:Ricci on suN}. Namely,
    \begin{align}\label{eq:cov-der-centr-ext}
     2\widehat{\nabla}_{(u,a)}(v,b)
     &=\widehat{\ad}_{(u,a)}(v,b) 
     - \widehat{\ad}_{(u,a)}^\star(v,b) - \widehat{\ad}_{(v,b)}^\star(u,a)\nonumber
     \\
     &=\Big({\ad}_{u}v 
     - {\ad}_{u}^\star v - {\ad}_{v}^\star u -(b{T}_Nu +a{T}_Nv)~,~\omega_N(u,v)\Big)\nonumber \\
     &=\Big(2\nabla_uv -(b{T}_Nu +a{T}_Nv)~,~\omega_N(u,v)\Big)
     \end{align}
%



We also introduce useful identities.
\begin{lemma}\label{lem: identities for omega and T}
    We have
    \begin{eqnarray}
      && \omega_N(Y_i,X_k)=-\omega_N(X_i,Y_k) = -2\pi\beta N\sin(\frac{\pi k_1}{N})\delta^i_k,\\
      &&\omega_N(X_i, X_k)=\omega_N(Y_i, Y_k)=0,\\
      && \| T_NX_k\|^2 =    \| T_NY_k\|^2 = \frac{-2 N^2\beta^2}{\lambda_k} \sin^2(\frac{\pi k_1}{N}),
    \end{eqnarray}
    where $\|u\|^2\coloneqq\langle u,u\rangle$ on $\su(N)$. 
    
\end{lemma}
\begin{proof}
    The results follow from direct computation using $\omega_N(Y_i, X_k)=\langle T_N Y_i, X_k\rangle$ and the orthogonality of the basis $\{X_k, Y_k\}_k$.
\end{proof}

We are now able to compute the curvatures for several choices of basis elements needed to compute $\widehat{Ric}((E_k,0)),(E_k,0))$, given in the next proposition.
\begin{proposition}[Riemann curvature tensor on the central extension]\label{prop:Riemannian tensors for hat su(N)}
The curvature of $\widehat{\su(N)}$ satisfies the following:
\begin{eqnarray}
    &&\ll  \widehat{R}\big((X_i,0),(X_k,a)\big)(X_k,a),(X_i,0)\gg=
\langle  R(X_i,X_k)X_k,X_i\rangle
+\frac14 a^2\|T_NX_i\|^2,\\
&& 
\ll  \widehat{R}\big((Y_i,0),(X_k,a)\big)(X_k,a),(Y_i,0)\gg=
\langle  R(Y_i,X_k)X_k,Y_i\rangle
+\frac14 a^2\|T_NY_i\|^2  
-\frac34\omega_N(Y_i,X_k)^2,\\
&&\ll  \widehat{R}\big((0,b),(X_k,a)\big)(X_k,a),(0,b)\gg=\frac14 b^2\|T_NX_k\|^2,\\
&& \ll  \widehat{R}\big((Y_i,0),(Y_k,a)\big)(Y_k,a),(Y_i,0)\gg=\langle  R(Y_i,Y_k)Y_k,Y_i\rangle
+\frac14 a^2\|T_NY_i\|^2,\\
&&\ll  \widehat{R}\big((X_i,0),(Y_k,a)\big)(Y_k,a),(X_i,0)\gg=\langle  R(X_i,Y_k)Y_k,X_i\rangle
+\frac14 a^2\|T_NX_i\|^2-\frac34\omega_N(Y_k,X_i)^2,\\
&& \ll  \widehat{R}\big((0,b),(Y_k,a)\big)(Y_k,a),(0,b)\gg=\frac14 b^2\|T_NY_k\|^2.
\end{eqnarray}






\end{proposition}
A proof is given in \autoref{sec:proofs Coriolis-force}.

\begin{remark}
 These curvatures in the continuous case $\widehat{\frak g}=\frak g\ltimes_\omega \RR$ as well as its effect on the stability of the quasi-geostrophic equation have been studied in the literature \cite{vizman2001geodesics},  \cite{lee2021nonpositive} and \cite{suri2024curvature}.
\end{remark}

Using the identities in \autoref{prop:Riemannian tensors for hat su(N)}, we can now obtain the Ricci curvature formula on $\widehat{\SU(N)}$. 
\begin{theorem}\label{th:Ricci with Coriolis effect}
    We have
    \begin{eqnarray}\label{eq:Ric with Coriolis}
        \widehat{\Ric}((E_k,1), (E_k,1))
         &=&\Ric (E_k,E_k)+\beta^2 \left( \frac{\sin^2(\frac{k_1\pi}{N})}{\lambda_k}+\sum_{i\in Z_N^+} \frac{\sin^2(\frac{i_1\pi}{N})}{\lambda_i^2} \right)
    \end{eqnarray}
    where $E_k\in \{X_i, Y_i\}_i$. 



\end{theorem}
\autoref{th:Ricci with Coriolis effect} shows that the curvature contribution of the central extension is given by an additional term depending on the extension parameter $\beta$ and the eigenvalues $\lambda_k$. Hence, the Coriolis effect modifies the Ricci curvature by adding a spectral correction. We expect this may influence the stability of the system, similarly to the continuous setting on the sphere, shown in \cite{modinSuri2026geodesic}.

\begin{proof}
First recall that $\omega(X_k, X_i)=0$, $\omega_N(X_k,Y_i)=\delta_{ki}\langle T_N X_k, Y_i\rangle=\delta_{ki}2\pi \beta N \sin(\frac{\pi k_1}{N})$ and $T_N X_k= 
-\frac{\beta}{\lambda_k}\frac{N}{\pi} \sin (\frac{\pi k_1}{N})Y_k$. 
    Using the identities in \autoref{prop:Riemannian tensors for hat su(N)}, we compute
   {\small \begin{eqnarray*}
         &&N^2 \widehat{\Ric}((X_k,1), (X_k,1))
         =  \sum_{i\in Z_N^+} {\ll  \widehat R((X_i,0),(X_k,1))(X_k,1)  ,(X_i,0)  \gg \over \ll (X_i,0), (X_i,0) \gg}\\
         &&+ \sum_{i\in Z_N^+} {\ll  \widehat R((Y_i,0),(X_k,1))(X_k,1)  ,(Y_i,0)  \gg \over \ll (Y_i,0), (Y_i,0) \gg}
         +  {\ll  \widehat R((0,1),(X_k,1))(X_k,1)  ,(0,1)  \gg \over \ll (0,1), (0,1) \gg}
         \\
         &=&\sum_{i\in Z_N^+} { \langle R(X_i, X_k)X_k, X_i\rangle 
         {+\frac{1}{4}\|T_NX_i\|^2}
         \over \langle X_i, X_i\rangle  } 
         + \sum_{i\in Z_N^+} { \langle R(Y_i, X_k)X_k, Y_i\rangle 
         {+\frac{1}{4}\|T_NY_i\|^2} -\frac{3}{4}\omega_N(Y_{i},X_k)^2
         \over \langle Y_i, Y_i\rangle } 
          +\frac{\frac{1}{4}\|T_N X_k\|^2}{1}
\end{eqnarray*}}
{\small\begin{eqnarray*}
         &=& \Ric (X_k,X_k)  -\frac{3}{4} \frac{\omega_N(Y_{k},X_k)^2}{ \langle Y_k, Y_k\rangle }+\frac{1}{4}\|T_N X_k\|^2 +\frac{1}{2} \sum_{i\in Z_N^+} \frac{ \|T_NX_i\|^2}{ \langle X_i, X_i\rangle}\\
         &=& \Ric (X_k,X_k) + \frac{\frac{-3}{4} 2^2 \pi^2\beta^2 N^2\sin^2(\frac{ k_1 \pi }{N})}{ -2\pi^2\lambda_k} + \frac{1}{4} \frac{2 N^2\beta^2 \sin^2(\frac{ k_1 \pi }{N})}{-\lambda_k}\\
         &&\qquad+\frac{1}{2}\sum_{i\in Z_N^+} \frac{  2 \pi^2 N^2\beta^2 \sin^2(\frac{i_1 \pi }{N})}{( -\lambda_i)(-2\pi^2 \lambda_i) }\\
         &=&\Ric (X_k,X_k)+ \frac{\beta^2 N^2 \sin^2(\frac{k_1\pi}{N})}{\lambda_k}\left( \frac{3}{2}-\frac{1}{2}\right)+
         \frac{1}{2}\sum_{i\in Z_N^+} \frac{\beta^2N^2\sin^2(\frac{i_1\pi}{N})}{\lambda_i^2}\\
         &=&\Ric (X_k,X_k)+\beta^2N^2 \left( \frac{\sin^2(\frac{k_1\pi}{N})}{\lambda_k}+ \frac{1}{2}\sum_{i\in Z_N^+} \frac{\sin^2(\frac{i_1\pi}{N})}{\lambda_i^2} \right)
         \end{eqnarray*}}
and as $N\rightarrow \infty$ the previous expression approaches to
\[
 \Ric (X_k,X_k) +\beta^2 \pi^2\left(\frac{ k_1^2}{\lambda_k} + \frac{1}{2}\sum_{i\in Z_N^+} \frac{ i_1^2}{\lambda_i^2} \right).
\]

    
\noindent
    The computation for $ \widehat{\Ric}((Y_k,1), (Y_k,1))$ can be done in the same way. 

    
\end{proof}


One can derive the sectional curvature $\widehat{K}$ on $\widehat{\SU(N)}$ using the same computational machinery as for the sectional curvature $K$ on $\SU(N)$ (see \autoref{sec:sectional curvature}). This shows that $\widehat{K}((E_k,1), (E_l, 0))>K(E_k, E_l)$ for many choices of basis elements $E_k$ and $E_l$, suggesting the existence of conjugate points. Interestingly, at the level of the Ricci curvature, on the other hand, the contribution of the Coriolis effect vanishes in the large-$N$ limit:
\begin{corollary}\label{cor:Coriolis converges to normal Ricci}
    As $N\to \infty$, we have 
    \begin{eqnarray*}
        \widehat{\Ric}((E_k,1), (E_k,1))\to \Ric(E_k, E_k)
    \end{eqnarray*}
    where $E_k\in\{X_k,Y_k\}$.
\end{corollary}

This corollary follows immediately from the following lemma.
\begin{lemma}\label{lem:infinite_lambda_sum}
    Let $\lambda_i:=-\frac{N^2}{\pi^2}\left(\sin^2(\frac{i_1 \pi}{N})+\sin^2(\frac{i_2 \pi}{N})\right)$ be the eigenvalues of $\Delta^{\ad}$. For positive integer $M$ and $N=2M+1$, we have 
    \begin{eqnarray*}
       \lim_{N\to\infty} \sum_{i\in Z_N^+}\frac{i_1^2}{\lambda_i^2}=\infty,
       \\ \lim_{N\to\infty}\frac{1}{N^2-1} \sum_{i\in Z_N^+} \frac{i_1^2}{\lambda_i^2}=0.
    \end{eqnarray*}
The same results hold also for $\lambda_i=-|i|^2$, the eigenvalues of $\Delta^{\rm spec}$.
\end{lemma}
A proof is given in \autoref{sec:proofs Coriolis-force}.

\paragraph{Numerical examples}
\autoref{fig:Ricci Coriolis}  numerically evaluates the Ricci curvature $\widehat{\Ric}$ with the Coriolis effect and $\Ric$ without the Coriolis effect using the formulas obtained in \autoref{th:Ricci with Coriolis effect} and \autoref{th:Ricci curvature formula} respectively.

While the region with positive Ricci curvature is larger in the presence of the Coriolis effect for relatively small $N$, the Ricci curvature values approach to the ones without the Coriolis force as $N$ increases. This confirms the results of \autoref{th:Ricci with Coriolis effect} and \autoref{cor:Coriolis converges to normal Ricci}. 

\begin{figure}[htbp]
\centering
\begin{minipage}{0.45\textwidth}
    \centering
    On $\widehat{\SU(51)}$
\end{minipage}
\begin{minipage}{0.5\textwidth}
    \centering
    Transition in $N$ 
\end{minipage}
 \caption*{$L=\Delta^{\ad},\,\beta=0$} 
 \vspace{-20pt}
\begin{minipage}{0.45\textwidth}  
 \includegraphics[width=1.0\textwidth]{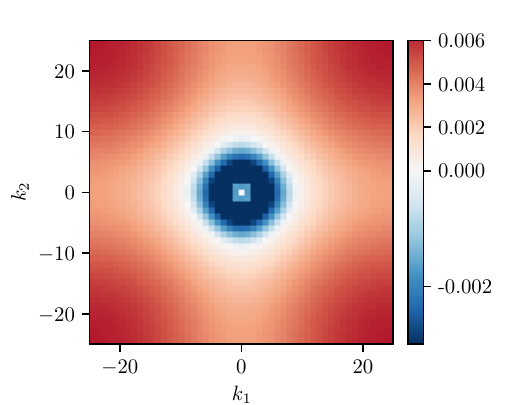}\end{minipage}
\begin{minipage}{0.5\textwidth}  
\includegraphics[width=1.0\textwidth]{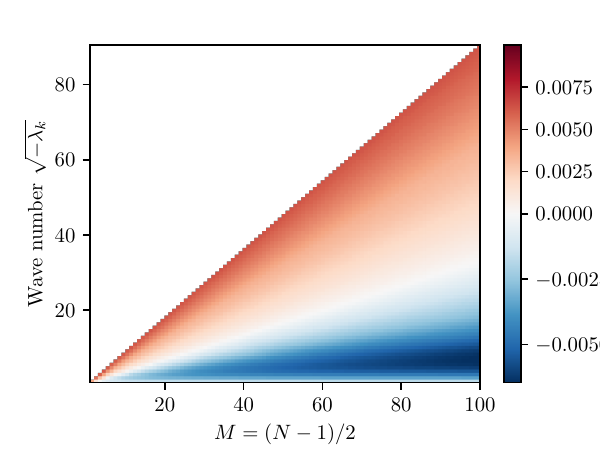}
 \end{minipage}

  \caption*{$L=\Delta^{\ad},\, \beta=0.1$} 
   \vspace{-20pt}
  \begin{minipage}{0.45\textwidth}  
\includegraphics[width=1.0\textwidth]{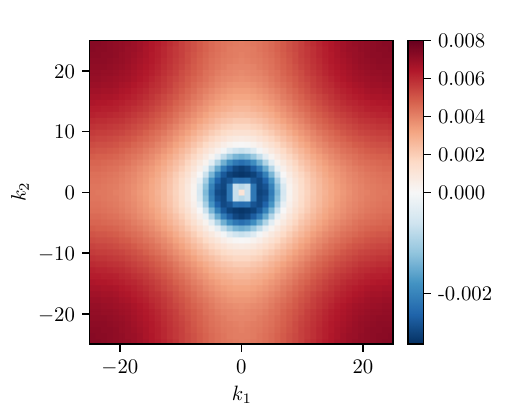}\end{minipage}
\begin{minipage}{0.5\textwidth} 
\includegraphics[width=1.0\textwidth]{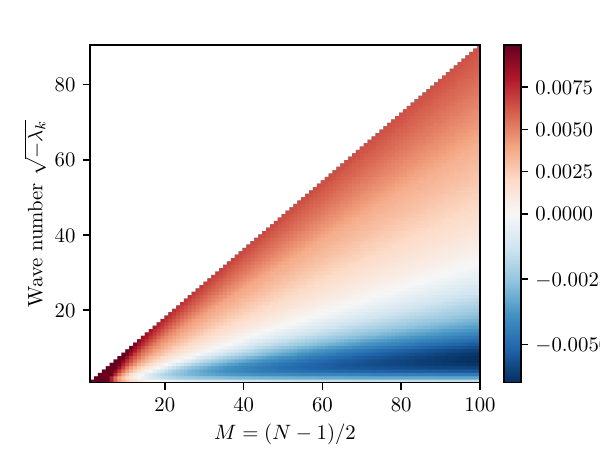} 
\end{minipage}

\caption{Comparison between the Ricci curvatures in the absence of the Coriolis force (top) and in its presence (bottom). }
\label{fig:Ricci Coriolis}
\end{figure}



%% file: appendix.tex
\appendix

\section{Proofs of auxiliary results}

\addtocontents{toc}{\protect\setcounter{tocdepth}{1}}

Here we provide proofs of auxiliary statements that consist primarily of straightforward yet lengthy calculations.

\subsection{Proofs of propositions in \autoref{sec:preliminaries}}\label{sec:proofs sec preliminaries}
\begin{proof}[Proof of \autoref{lem:sine bracket commutator}]
     We compute
  \begin{eqnarray}
    T_pT_q &=& -\frac{N^2}{4\pi^2}\omega^{-p_1p_2/2}U^{p_1}V^{p_2}\omega^{-q_1q_2/2}U^{q_1}V^{q_2}\\
    &=& -\frac{N^2}{4\pi^2} \omega^{\frac{-p_1p_2 -q_1q_2}{2}}U^{p_1}V^{p_2}U^{q_1}V^{q_2}\\
    &\stackrel{I}=& -\frac{N^2}{4\pi^2}\omega^{\frac{-p_1p_2 -q_1q_2}{2}}\omega^{-p_2q_1} U^{p_1+q_1}V^{p_2+q_2}\\
    &=& -\frac{N^2}{4\pi^2}\omega^{\frac{-p_1p_2 -q_1q_2 -2p_2q_1}{2}} U^{p_1+q_1}V^{p_2+q_2}
  \end{eqnarray}
where in $I$ we used the identity $U^aV^b=\omega^{ab}V^bU^a$. On the other hand
    \begin{eqnarray*}
    T_{p+q} &=& - \frac{iN}{2\pi}\omega^{-\frac{(p_1+q_1)(p_2+q_2)}{2}}U^{p_1+q_1}V^{p_2+q_2}\\
    &=&- \frac{iN}{2\pi}\omega^{\frac{(-p_1p_2-q_1q_2-2p_2q_1 ) + (q_2p_1 -p_1q_2) }{2}}U^{p_1+q_1}V^{p_2+q_2}
  \end{eqnarray*}
Comparing the last two equations we get 
  \begin{eqnarray}
    T_pT_q &=&  - \frac{iN}{2\pi} \omega^{-\frac{1}{2}p\times q}T_{p+q}.\label{eq:Tp Tq}
  \end{eqnarray}
As a result
  \begin{eqnarray*}
    [T_p,T_q]&=& T_pT_q -T_qT_p\\
    &=&- \frac{iN}{2\pi}\Big(  \omega^{-\frac{1}{2}p\times q} - \omega^{-\frac{1}{2}q\times p}\Big) T_{p+q}\\
     &=& - \frac{iN}{2\pi} \Big( \exp(-\frac{2\pi i}{N}\frac{1}{2}p\times q) - \exp(\frac{2\pi i}{N}\frac{1}{2}p\times q) \Big) T_{p+q}\\
     &=& - \frac{iN}{2\pi}\cdot - 2i\sin\left(\frac{\pi p\times q}{N}\right)T_{p+q}\\
     &=&-\frac{N}{\pi}\sin\left(\frac{\pi}{N}p\times q\right)T_{p+q}
  \end{eqnarray*}
which verifies the stated equality.
\end{proof}

\begin{proof}[Proof of \autoref{lem:T_k is unitary}]
We compute,
    \begin{eqnarray}
    T_k^\dag 
    &=& \frac{iN}{2\pi} \omega^{k_1k_2/2} (V^{k_2})^\dag (U^{k_1})^\dag\nonumber\\
    &=& \frac{iN}{2\pi} \omega^{k_1k_2/2} V^{-k_2} U^{-k_1}, \qquad (U,V\in SU(N))\nonumber\\
    &=& \frac{iN}{2\pi} \omega^{k_1k_2/2}  \omega^{-k_1k_2} U^{-k_1} V^{-k_2},\qquad (U^aV^b=\omega^{ab}V^bU^a)\nonumber\\
    &=& \frac{iN}{2\pi} \omega^{-k_1k_2/2} U^{-k_1} V^{-k_2}\nonumber\\
    &=&-T_{-k}.
\end{eqnarray}
The second claimed identity $T_k T_k^\dag = \frac{N^2}{4\pi^2} I_N$ follows from $- T_k= T_{-k}$ and \eqref{eq:Tp Tq}.
\end{proof}

\begin{proof}[Proof of \autoref{prop:basis on su (N) with standard metric}]
We prove that $(X_k, X_l)=(Y_k,Y_l) = =\frac{N^3}{8\pi^2} \delta_{kl}$ and $(X_k,Y_l)=0$. 
Compute, 
\begin{align}
    (X_k, X_l)
    &=\frac{1}{4}\Tr\left((T_k-T_k^\dag)^\dag (T_l-T_l^\dag) \right)\\
    &= \frac{1}{4}\Tr \left((T_k^\dag-T_k) (T_l-T_l^\dag) \right)\\
    & =  \frac{1}{4} \left(\Tr(T_k^\dag T_l) + \Tr(T_k T_l^\dag) - \Tr(T_k T_l) - \Tr(T_k^\dag T_l^\dag) \right).
\end{align}
Note that
\begin{eqnarray*}
\Tr(T_k^\dag T_l )
&=&
\Tr\left( \left(\frac{iN}{2\pi} \omega^{\frac{k_1 k_2}{2}}  V^{-k_2} U^{-k_1}\right) \left( -\frac{iN}{2\pi} \omega^{\frac{-l_1 l_2}{2}} U^{l_1} V^{l_2} \right) \right)
\\
&=& \Tr\Big( \frac{N^2}{4\pi^2}  \omega^{\frac{k_1 k_2- l_1l_2}{2}} V^{-k_2}U^{-k_1}     U^{l_1} V^{l_2} \Big) 
\\
&=& \Tr\Big( \frac{N^2}{4\pi^2}  \omega^{\frac{k_1 k_2- l_1l_2}{2}} V^{l_2-k_2}U^{l_1-k_1}      \Big) 
\\
&=& \frac{N^3}{4\pi^2} \delta_{kl}.
\end{eqnarray*}
As a result we have $\Tr(T_k^\dag T_l) + \Tr(T_k T_l^\dag)=\frac{N^{3}}{2\pi^2}\delta_{kl}$. 
Note $T_k T_l$ is a multiple of $T_{k+l}$, which is trace-free unless $k+l=0 \mod N$. But since $k_1,k_2, l_1, l_2 \leq (N-1)/2$, this is the case only when $k+l=0$. If $k+l\neq 0$, namely $l\neq -k$, we have $ \Tr(T_k T_l)=0$ and similarly $ \Tr(T_k^\dag T_l^\dag)=0$. Therefore $(X_k,X_l)= \frac{N^3}{8\pi^2}  \delta_{kl}$.
If $l=-k$, recall that $X_k=X_{-k}$, hence $(X_k,X_{-k})=(X_k,X_k)=\frac{N^3}{8\pi^2}$.

The results $(Y_k,Y_l)=\frac{N^3}{8\pi^2}\delta_{kl}$ and $(X_k,Y_l)=0$ follow from similar computations.
\end{proof}

\subsection{Proofs of propositions in \autoref{sec:auxiliary computations}}\label{sec:proofs auxiliary lemmas}

\begin{proof}[Proof of \autoref{lem:[XkXl]}]
    Using \eqref{eq:Bracket 1}, we have
    \begin{align*}
        [X_k,X_l]&=\frac{1}{4}[T_k + T_{-k}, T_l+T_{-l}]\\
        &=\frac{1}{4}\frac{N}{\pi}\sin(\frac{\pi}{N}k\times l)(-T_{k+l}+T_{l-k}+T_{k-l}-T_{-k,-l})\\
        &= -\frac{1}{2}\frac{N}{\pi}\sin(\frac{\pi}{N}k\times l)(X_{k+l}-X_{k-l}).
    \end{align*}
    
    \begin{align*}
        [Y_k,Y_l]&=-\frac{1}{4}[T_k - T_{-k}, T_l-T_{-l}]\\
        &=\frac{1}{4}\frac{N}{\pi}\sin(\frac{\pi}{N}k\times l)(T_{k+l}+T_{l-k}+T_{k-l}+T_{-k,-l})\\
        &= \frac{1}{2}\frac{N}{\pi}\sin(\frac{\pi}{N}k\times l)(X_{k+l}+X_{k-l}).
    \end{align*}

    \begin{align*}
        [X_k,Y_l]&=\frac{1}{4i}[T_k + T_{-k}, T_l-T_{-l}]\\
        &=\frac{1}{4i}\frac{N}{\pi}\sin(\frac{\pi}{N}k\times l)(-T_{k+l}+T_{l-k}-T_{k-l}+T_{-k,-l})\\
        &= -\frac{1}{2}\frac{N}{\pi}\sin(\frac{\pi}{N}k\times l)(Y_{k+l}+Y_{k-l}).
    \end{align*}

    \begin{align*}
        [Y_k,X_l]&=\frac{1}{4i}[T_k - T_{-k}, T_l+T_{-l}]\\
        &=\frac{1}{4i}\frac{N}{\pi}\sin(\frac{\pi}{N}k\times l)(-T_{k+l}-T_{l-k}+T_{k-l}+T_{-k,-l})\\
        &= -\frac{1}{2}\frac{N}{\pi}\sin(\frac{\pi}{N}k\times l)(Y_{k+l}-Y_{k-l}).
    \end{align*}
\end{proof}

\begin{proof}[Proof of \autoref{lem:coadjoint}]
    Recall the formula $\ad^{\star}_u v =-L^{-1}\ad_u L v$ (\autoref{lem:coad_formula_for_Laplacian_metric}) for the Zeitlin metric. Together with the structure constant formulas (\autoref{lem:[XkXl]}), we compute
    \begin{align}
        \ad_{X_k}^\star X_l
        &=-L^{-1}[X_k, L X_l] 
        =- \lambda_l L^{-1}[X_k, X_l]\\
        &= \lambda_l  L^{-1} \frac{1}{2}\frac{N}{\pi} \sin(\frac{\pi}{N}k\times l) (X_{k+l}-X_{k-l})\\
        &=\frac{1}{2}\frac{N}{\pi} \sin(\frac{\pi}{N}k\times l) \lambda_l (\frac{X_{k+l}}{ \lambda_{k+l}}-\frac{X_{k-l}}{ \lambda_{k-l}}).
    \end{align}
    In the  same way, we have
    \begin{align}
        \ad_{Y_k}^\star Y_l
        &=-L^{-1}[Y_k, L Y_l] 
        =- \lambda_l L^{-1}[Y_k, Y_l]\\
        &= -\lambda_l  L^{-1} \frac{1}{2}\frac{N}{\pi} \sin(\frac{\pi}{N}k\times l) (X_{k+l}+X_{k-l})\\
        &=-\frac{1}{2}\frac{N}{\pi} \sin(\frac{\pi}{N}k\times l) \lambda_l (\frac{X_{k+l}}{ \lambda_{k+l}}+\frac{X_{k-l}}{ \lambda_{k-l}}),
    \end{align}
    
     \begin{align}
        \ad_{X_k}^\star Y_l
        &=-L^{-1}[X_k, L Y_l] 
        =- \lambda_l L^{-1}[X_k, Y_l]\\
        &= \lambda_l  L^{-1} \frac{1}{2}\frac{N}{\pi} \sin(\frac{\pi}{N}k\times l) (Y_{k+l}+Y_{k-l})\\
        &=\frac{1}{2}\frac{N}{\pi} \sin(\frac{\pi}{N}k\times l) \lambda_l (\frac{Y_{k+l}}{ \lambda_{k+l}}+\frac{Y_{k-l}}{ \lambda_{k-l}}),
    \end{align}
    and 
    \begin{align}
        \ad_{Y_k}^\star X_l
        &=-L^{-1}[Y_k, L X_l] 
        =- \lambda_l L^{-1}[Y_k, X_l]\\
        &=  \lambda_l  L^{-1} \frac{1}{2}\frac{N}{\pi} \sin(\frac{\pi}{N}k\times l) (Y_{k+l}-Y_{k-l})\\
        &=  \frac{1}{2}\frac{N}{\pi} \sin(\frac{\pi}{N}k\times l) \lambda_l (\frac{Y_{k+l}}{ \lambda_{k+l}}-\frac{Y_{k-l}}{ \lambda_{k-l}}).
    \end{align}
\end{proof}

\begin{proof}[Proof of \autoref{lem:covariant_derivatives}]
Let $c_{kl}=\frac{N}{\pi}\sin(\frac{\pi}{N}k\times l)$ for simplicity. 
    Using \autoref{lem:[XkXl]} and \autoref{lem:coadjoint}, compute
    \begin{eqnarray*}
        2\nabla_{X_k}X_l 
        &=&   \ad_{X_k}X_l-\ad^\star_{X_k}X_l-\ad^\star_{X_l}X_k\\
        &=&  -\frac{1}{2}c_{kl}(X_{k+l}-X_{k-l})- \frac{1}{2}c_{kl}(\frac{\lambda_l}{\lambda_{k+l}}X_{k+l}-\frac{\lambda_l}{\lambda_{k-l}}X_{k-l})- \frac{1}{2}c_{lk}(\frac{\lambda_k}{\lambda_{l+k}}X_{l+k}-\frac{\lambda_k}{\lambda_{l-k}}X_{l-k}) \\
        &=& X_{k+l}\left(-\frac{1}{2}c_{kl}-\frac{1}{2}c_{kl}\frac{\lambda_l}{\lambda_{k+l}}-\frac{1}{2}c_{lk}\frac{\lambda_k}{\lambda_{l+k}}\right)
        +X_{k-l}\left(\frac{1}{2}c_{kl}+\frac{1}{2}c_{kl}\frac{\lambda_l}{\lambda_{k-l}}+\frac{1}{2}c_{lk}\frac{\lambda_k}{\lambda_{l-k}}\right)\\
          &=& -\frac{1}{2}c_{kl}\left(1+ \frac{\lambda_l-\lambda_k}{\lambda_{k+l}}\right)X_{k+l}+ \frac{1}{2}c_{kl} \left(1+ \frac{\lambda_l-\lambda_k}{\lambda_{k-l}}\right) X_{k-l},
    \end{eqnarray*}
    \begin{eqnarray*}
       2\nabla_{Y_k}Y_l
&=& \ad_{Y_k}Y_l-\ad^\star_{Y_k}Y_l-\ad^\star_{Y_l}Y_k
\\
&=& \frac{1}{2}c_{kl}(X_{k+l}+X_{k-l})
+ \frac{1}{2}c_{kl}\left(\frac{\lambda_l}{\lambda_{k+l}}X_{k+l} + \frac{\lambda_l}{\lambda_{k-l}}X_{k-l}\right)
+ \frac{1}{2}c_{lk}\left(\frac{\lambda_k}{\lambda_{l+k}}X_{l+k} + \frac{\lambda_k}{\lambda_{l-k}}X_{l-k}\right)
\\
&=& \frac{1}{2}c_{kl}(X_{k+l}+X_{k-l})
+ \frac{1}{2}c_{kl}\left(\frac{\lambda_l}{\lambda_{k+l}}X_{k+l} + \frac{\lambda_l}{\lambda_{k-l}}X_{k-l}\right)
- \frac{1}{2}c_{kl}\left(\frac{\lambda_k}{\lambda_{k+l}}X_{k+l} + \frac{\lambda_k}{\lambda_{k-l}}X_{k-l}\right)
\\
&=& \frac{1}{2}c_{kl} X_{k+l}\left(1 + \frac{\lambda_l}{\lambda_{k+l}} - \frac{\lambda_k}{\lambda_{k+l}}\right)
+
\frac{1}{2}c_{kl} X_{k-l}\left(1 + \frac{\lambda_l}{\lambda_{k-l}} - \frac{\lambda_k}{\lambda_{k-l}}\right)
\\
&=& \frac{1}{2}c_{kl}\left(1 + \frac{\lambda_l-\lambda_k}{\lambda_{k+l}}\right)X_{k+l}
+
\frac{1}{2}c_{kl} \left(1 + \frac{\lambda_l-\lambda_k}{\lambda_{k-l}}\right) X_{k-l}. %
          \\
        2\nabla_{X_k}Y_l 
        &=&   \ad_{X_k}Y_l-\ad^\star_{X_k}Y_l-\ad^\star_{Y_l}X_k
        \\
        &=&  -\frac{1}{2}c_{kl}(Y_{k+l}+Y_{k-l})  
        - \frac{1}{2}c_{kl}(\frac{\lambda_l}{\lambda_{k+l}}Y_{k+l} + \frac{\lambda_l}{\lambda_{k-l}}Y_{k-l})
        - \frac{1}{2}c_{lk}(\frac{\lambda_k}{\lambda_{l+k}}Y_{l+k} - \frac{\lambda_k}{\lambda_{l-k}}Y_{l-k}) 
        \\
        &=&  -\frac{1}{2}c_{kl}(Y_{k+l}+Y_{k-l})  
        - \frac{1}{2}c_{kl}(\frac{\lambda_l}{\lambda_{k+l}}Y_{k+l} + \frac{\lambda_l}{\lambda_{k-l}}Y_{k-l})
        + \frac{1}{2}c_{kl}(\frac{\lambda_k}{\lambda_{l+k}}Y_{l+k} + \frac{\lambda_k}{\lambda_{l-k}}Y_{k-l}) 
        \\
        &=& \frac{1}{2}c_{kl} Y_{k+l}\left(-1 -\frac{\lambda_l}{\lambda_{k+l}} +\frac{\lambda_k}{\lambda_{l+k}}\right)
        +
        \frac{1}{2}c_{kl} Y_{k-l}\left(-1 - \frac{\lambda_l}{\lambda_{k-l}} + \frac{\lambda_k}{\lambda_{l-k}}\right)
        \\
        &=& \frac{1}{2}c_{kl}\left(-1 - \frac{\lambda_l-\lambda_k}{\lambda_{k+l}}\right)Y_{k+l}
        +
       \frac{1}{2}c_{kl} \left(-1 - \frac{\lambda_l-\lambda_k}{\lambda_{k-l}}\right) Y_{k-l}.
    \end{eqnarray*}
    and finally
    \begin{eqnarray*}
      2\nabla_{Y_k}X_l
      &=& \ad_{Y_k}X_l-\ad^\star_{Y_k}X_l-\ad^\star_{X_l}Y_k
      \\
      &=& -\frac{1}{2}c_{kl}(Y_{k+l}-Y_{k-l})
      - \frac{1}{2}c_{kl}\left(\frac{\lambda_l}{\lambda_{k+l}}Y_{k+l} - \frac{\lambda_l}{\lambda_{k-l}}Y_{k-l}\right)
       - \frac{1}{2}c_{lk}\left(\frac{\lambda_k}{\lambda_{l+k}}Y_{l+k} + \frac{\lambda_k}{\lambda_{l-k}}Y_{l-k}\right)
      \\
      &=& -\frac{1}{2}c_{kl}(Y_{k+l}-Y_{k-l})
      - \frac{1}{2}c_{kl}\left(\frac{\lambda_l}{\lambda_{k+l}}Y_{k+l} - \frac{\lambda_l}{\lambda_{k-l}}Y_{k-l}\right)
      + \frac{1}{2}c_{kl}\left(\frac{\lambda_k}{\lambda_{k+l}}Y_{k+l} - \frac{\lambda_k}{\lambda_{k-l}}Y_{k-l}\right)
     \\ 
      &=& \frac{1}{2}c_{kl} Y_{k+l}\left(-1 -\frac{\lambda_l}{\lambda_{k+l}} +\frac{\lambda_k} {\lambda_{k+l}}\right)
     +
     \frac{1}{2}c_{kl} Y_{k-l}\left(1 + \frac{\lambda_l}{\lambda_{k-l}} - \frac{\lambda_k}{\lambda_{k-l}}\right)
     \\
    &=& \frac{1}{2}c_{kl}\left(-1 - \frac{\lambda_l-\lambda_k}{\lambda_{k+l}}\right)Y_{k+l}
     +
     \frac{1}{2}c_{kl} \left(1 + \frac{\lambda_l-\lambda_k}{\lambda_{k-l}}\right) Y_{k-l}.
    \end{eqnarray*}
\end{proof}

\subsection{Proofs of propositions in \autoref{sec:off diagonal}}\label{sec:proofs off diagonal}
Throughout this section, we set  $c_{ab}=\frac{N}{\pi}\sin(\frac{\pi}{N}a\times b)$ for integers $a,b$. 

We first prove the four identities in \autoref{lem:vanishing_off_diagonals} one by one.
\begin{proof}[Proof for \autoref{eq:vanishing_off_diagonal1}]
Let $A_{ab}=1+\frac{\lambda_b-\lambda_a}{\lambda_{a+b}}$, and $B_{ab}=1+\frac{\lambda_b-\lambda_a}{\lambda_{a-b}}$ for notational simplicity. 

We show the stated equality for the cases 1. $E_k=X_k, E_l=X_l$ with $k\neq l$; 2. $E_k=Y_k, E_l=Y_l$ with $k\neq l$; and 3. $E_k=X_k, E_l=Y_l$ separately.

We first show the claim for the first case: $E_k=X_k, E_l=X_l$. For each $i$, we have
\begin{eqnarray}
\langle\nabla_{X_i}X_k,\nabla_{X_l} X_i   \rangle
&=&\frac{1}{4}\langle  \frac{1}{2} c_{ik}(-A_{ik}X_{i+k}+B_{ik}X_{i-k}), \frac{1}{2} c_{li}(-A_{li}X_{l+i}+B_{li}X_{l-i}) \rangle\nonumber \\
&=& \frac{1}{16} c_{ik}c_{li}\langle -A_{ik} X_{i+k}+B_{ik} X_{i-k}, -A_{li}X_{l+i}+B_{li}X_{l-i} \rangle.
\end{eqnarray}

Recall that $\langle X_a, X_b\rangle$ is nonzero if and only if $a=\pm b$. With the assumption that $k\neq \pm l$, we see that $\langle X_{i+k}, X_{l+i}\rangle$ is nonzero if and only if $i+k=-i-l$, namely $i=\frac{-k-l}{2}$ and hence $i+k=-i-l=\frac{k-l}{2}$. Similarly, $\langle X_{i+k},X_{l-i}\rangle$ is nonzero only if  $i=\frac{-k+l}{2}$ and hence $i+k=l-i=\frac{k+l}{2}$. Also, $\langle X_{i-k},X_{l+i}\rangle$ is nonzero only if  $i=\frac{k-l}{2}$ and hence $i-k=-l-i=\frac{- k-l}{2}$.
Finally, $\langle X_{i-k},X_{l-i}\rangle$ is nonzero only if  $i=\frac{k+l}{2}$ and hence $i-k=l-i=\frac{-k+l}{2}$.

Note that both $\nabla_{X_i}X_k$ and $\nabla_{X_l} X_i$ are invariant under the change of $X_i\to X_{-i}$, hence so is $\langle\nabla_{X_i}X_k,\nabla_{X_l} X_i   \rangle$. Hence, under the sign choice of $i=\pm \frac{k+l}{2}$, only one of $\langle X_{i+k}, X_{l+i}\rangle$ or $\langle X_{i-k},X_{l-i}\rangle$ does not vanish, yet keeping the value of $\langle\nabla_{X_i}X_k,\nabla_{X_l} X_i   \rangle$ unchanged. The same applies to the sign choice of $i=\pm \frac{k-l}{2}$ for $\langle X_{i+k},X_{l-i}\rangle$ and  $\langle X_{i-k},X_{l+i}\rangle$. In the proof, we choose the plus sign, while the result is independent of this choice.


Hence we have,
\begin{eqnarray*}
   16 \sum_i {\langle\nabla_{X_i}X_k,\nabla_{X_l} X_i   \rangle \over \langle X_i, X_i \rangle}
   &=&c_{\frac{-k-l}{2}, k}c_{l, \frac{-k-l}{2}} A_{\frac{-k-l}{2}, k} A_{l, \frac{-k-l}{2}} \frac{\langle X_{\frac{k-l}{2}}, X_{-\frac{k-l}{2}}\rangle}{\langle X_{\frac{-k-l}{2}}, X_{\frac{-k-l}{2}} \rangle }\\
   &&- c_{\frac{-k+l}{2},k} c_{l,\frac{-k+l}{2}} A_{\frac{-k+l}{2},k} B_{l,\frac{-k+l}{2}} \frac{\langle X_{\frac{k+l}{2}}, X_{\frac{k+l}{2}}\rangle}{\langle X_{\frac{-k+l}{2}}, X_{\frac{-k+l}{2}} \rangle }.
\end{eqnarray*}

We now compute $\langle\nabla_{Y_i}X_k,\nabla_{X_l} Y_i   \rangle$. For each $i$, we have
\begin{eqnarray}
\langle\nabla_{Y_i}X_k,\nabla_{X_l} Y_i   \rangle
&=&\frac{1}{4}\langle  \frac{1}{2} c_{ik}(-A_{ik}Y_{i+k}+B_{ik}Y_{i-k}), \frac{1}{2} c_{li}(-A_{li}Y_{l+i}-B_{li}Y_{l-i}) \rangle\nonumber\\
&=& \frac{1}{16} c_{ik}c_{li}\langle -A_{ik} Y_{i+k}+B_{ik} Y_{i-k}, -A_{li}Y_{l+i}-B_{li}Y_{l-i} \rangle.
\end{eqnarray}
By the above argument, the same indices $i\in\{\pm\frac{k+l}{2},\pm\frac{k-l}{2}\}$ makes $\langle\nabla_{Y_i}X_k,\nabla_{X_l} Y_i   \rangle$ non-vanishing. Note now that the signs of both $\nabla_{Y_i}X_k$ and $\nabla_{X_l} Y_i$ are flipped under the change of $Y_i\to Y_{-i}$, hence  $\langle\nabla_{Y_i}X_k,\nabla_{X_l} Y_i   \rangle$ remains unchanged.

Therefore we have 
\begin{eqnarray*}
   16 \sum_i {\langle\nabla_{Y_i}X_k,\nabla_{X_l} Y_i   \rangle \over \langle Y_i,Y_i \rangle}
   &=&c_{\frac{-k-l}{2}, k}c_{l, \frac{-k-l}{2}} A_{\frac{-k-l}{2}, k} A_{l, \frac{-k-l}{2}} \frac{\langle Y_{\frac{k-l}{2}}, Y_{-\frac{k-l}{2}}\rangle}{\langle Y_{\frac{-k-l}{2}}, Y_{\frac{-k-l}{2}} \rangle }\\
   &&+ c_{\frac{-k+l}{2},k} c_{l,\frac{-k+l}{2}} A_{\frac{-k+l}{2},k} B_{l,\frac{-k+l}{2}} \frac{\langle Y_{\frac{k+l}{2}}, Y_{\frac{k+l}{2}}\rangle}{\langle Y_{\frac{-k+l}{2}}, Y_{\frac{-k+l}{2}} \rangle }\\
   &=&- c_{\frac{-k-l}{2}, k}c_{l, \frac{-k-l}{2}} A_{\frac{-k-l}{2}, k} A_{l, \frac{-k-l}{2}} \frac{\langle Y_{\frac{k-l}{2}}, Y_{\frac{k-l}{2}}\rangle}{\langle Y_{\frac{-k-l}{2}}, Y_{\frac{-k-l}{2}} \rangle }\\
   && c_{\frac{-k+l}{2},k} c_{l,\frac{-k+l}{2}} A_{\frac{-k+l}{2},k} B_{l,\frac{-k+l}{2}} \frac{\langle Y_{\frac{k+l}{2}}, Y_{\frac{k+l}{2}}\rangle}{\langle Y_{\frac{-k+l}{2}}, Y_{\frac{-k+l}{2}} \rangle },
\end{eqnarray*}
applying $Y_{-a}=-Y_a$ for any $a$.

Combining these computational results, we obtain
\begin{equation*}
    16 \sum_i {\langle\nabla_{X_i}X_k,\nabla_{X_l} X_i   \rangle \over \langle X_i, X_i \rangle} +  16 \sum_i {\langle\nabla_{Y_i}X_k,\nabla_{X_l} Y_i   \rangle \over \langle Y_i,Y_i \rangle}=0
\end{equation*}
using $\langle X_a, X_a\rangle=\langle Y_a, Y_a\rangle$ for all $a$.

For the second case $E_k=Y_k, E_l=Y_l$ with $k\neq \pm l$, it follows that
\begin{equation*}
    \sum_i {\langle\nabla_{X_i}Y_k,\nabla_{Y_l} X_i   \rangle \over \langle X_i, X_i \rangle} +   \sum_i {\langle\nabla_{Y_i}Y_k,\nabla_{Y_l} Y_i   \rangle \over \langle Y_i,Y_i \rangle}=0.
\end{equation*}
from the same argument and computational routine. 

It remains to show the claim for the last case $E_k=X_k, E_l=Y_l$ for any $k,l$. We have in fact
\begin{align*}
\langle\nabla_{X_i}X_k,\nabla_{Y_l} X_i   \rangle = \langle\nabla_{X_i}X_k,\nabla_{Y_l} Y_i   \rangle =0
\end{align*}
for any $i,k,l$. To see this, compute
\begin{eqnarray*}
\langle\nabla_{X_i}X_k,\nabla_{Y_l} X_i   \rangle
&=&\frac{1}{4}\langle  \frac{1}{2} c_{ik}(-A_{ik}X_{i+k}+B_{ik}X_{i-k}), \frac{1}{2} c_{li}(-A_{li}Y_{l+i}+B_{li}Y_{l-i}) \rangle\\
&=& \frac{1}{16} c_{ik}c_{li}\langle -A_{ik} X_{i+k}+B_{ik} X_{i-k}, -A_{li}Y_{l+i}+B_{li}Y_{l-i} \rangle  =0
\end{eqnarray*}
since $\langle X_a, Y_b\rangle=0$ for any $a,b$.  In the same way
\begin{eqnarray*}
\langle\nabla_{Y_i}X_k,\nabla_{Y_l} Y_i   \rangle
&=&\frac{1}{4}\langle  \frac{1}{2} c_{ik}(-A_{ik}Y_{i+k}+B_{ik}Y_{i-k}), \frac{1}{2} c_{li}(A_{li}X_{l+i}+B_{li}X_{l-i}) \rangle\\
&=& \frac{1}{16} c_{ik}c_{li}\langle A_{ik} Y_{i+k}+B_{ik} Y_{i-k}, -A_{li}X_{l+i}+B_{li}X_{l-i} \rangle  =0.
\end{eqnarray*}
This concludes the proof.
\end{proof}

\begin{proof}[Proof for \autoref{eq:vanishing_off_diagonal2}]
    The proof is almost identical to the proof for \autoref{eq:vanishing_off_diagonal1}.
    Here we show the claim only for the case $E_k=X_k, E_l=X_l$ with $k\neq l$. The claims for the case $E_k=Y_k, E_l=Y_l$ with $k\neq l$ and for $E_k=X_k, E_l=Y_l$ follows from the same argument.  
    
    For each $i$, we have
\begin{eqnarray*}
\langle [X_i, X_k], [X_l, X_i]   \rangle
&=&\langle  \frac{1}{2} c_{ik}(-X_{i+k}+X_{i-k}), \frac{1}{2} c_{li}(-X_{l+i}+X_{l-i}) \rangle\\
&=& \frac{1}{4} c_{ik}c_{li}\langle -X_{i+k}+X_{i-k}, -X_{l+i}+X_{l-i} \rangle.
\end{eqnarray*}
Just setting  $A_{ab}=B_{ab}=1$ for all $a,b$ in the proof of \autoref{eq:vanishing_off_diagonal1}, we compute
\begin{eqnarray*}
   4 \sum_i {\langle[X_i, X_k], [X_l, X_i]  \rangle \over \langle X_i, X_i \rangle}
   &=& c_{\frac{-k-l}{2}, k}c_{l, \frac{-k-l}{2}}  \frac{\langle X_{\frac{k-l}{2}}, X_{\frac{k-l}{2}}\rangle}{\langle X_{\frac{-k-l}{2}}, X_{\frac{-k-l}{2}} \rangle }\\
   &&- c_{\frac{-k+l}{2},k} c_{l,\frac{-k+l}{2}}  \frac{\langle X_{\frac{k+l}{2}}, X_{\frac{k+l}{2}}\rangle}{\langle X_{\frac{-k+l}{2}}, X_{\frac{-k+l}{2}} \rangle }.
\end{eqnarray*}
We then  compute $\langle [Y_i,X_k], [X_l, Y_i] \rangle$. For each $i$, we have
\begin{eqnarray*}
\langle[Y_i,X_k], [X_l, Y_i]  \rangle
&=&\langle  \frac{1}{2} c_{ik}(-Y_{i+k}+Y_{i-k}), \frac{1}{2} c_{li}(-Y_{l+i}-Y_{l-i}) \rangle\\
&=& \frac{1}{4} c_{ik}c_{li}\langle - Y_{i+k}+ Y_{i-k}, -Y_{l+i}-Y_{l-i} \rangle,
\end{eqnarray*}
and hence
\begin{eqnarray*}
   4 \sum_i {\langle[Y_i, X_k], [X_l, Y_i]  \rangle \over \langle Y_i, Y_i \rangle}
   &=&-c_{\frac{-k-l}{2}, k}c_{l, \frac{-k-l}{2}}  \frac{\langle Y_{\frac{k-l}{2}}, Y_{\frac{k-l}{2}}\rangle}{\langle Y_{\frac{-k-l}{2}}, Y_{\frac{-k-l}{2}} \rangle }\\
   &&+c_{\frac{-k+l}{2},k} c_{l,\frac{-k+l}{2}}  \frac{\langle Y_{\frac{k+l}{2}}, Y_{\frac{k+l}{2}}\rangle}{\langle Y_{\frac{-k+l}{2}}, Y_{\frac{-k+l}{2}} \rangle }.
\end{eqnarray*}
Summing up the results, we obtain the  claim.
\end{proof}

\begin{proof}[Proof for \autoref{eq:vanishing_off_diagonal3}]
    We again follow the computational routine for the proof of \autoref{eq:vanishing_off_diagonal1}, and show the claim only for the case $E_k=X_k, E_l=X_l$ with $k\neq l$, since the claims for the case $E_k=Y_k, E_l=Y_l$ with $k\neq l$ and for $E_k=X_k, E_l=Y_l$ follows from the same argument.  
    
    Let $A_{ab}={\lambda_b\over \lambda_{a+b}}, B_{ab}={\lambda_b\over \lambda_{a-b}}$. 
    For each $i$, we have
\begin{eqnarray*}
\langle [X_i,X_k],\ad^\star_{X_l}X_i  \rangle
&=&\langle  \frac{1}{2} c_{ik}(-X_{i+k}+X_{i-k}), \frac{1}{2} c_{li}(A_{li}X_{l+i}-B_{li}X_{l-i}) \rangle\\
&=& \frac{1}{4} c_{ik}c_{li}\langle  -X_{i+k} +X_{i-k}, A_{li}X_{l+i}-B_{li}X_{l-i} \rangle.
\end{eqnarray*}
Note that $\langle [X_i,X_k],\ad^\star_{X_l}X_i  \rangle$ is invariant under the change $X_i\to X_{-i}$. This allows us to use the same computational machinery as before. We have
\begin{eqnarray*}
   4 \sum_i {\langle [X_i,X_k],\ad^\star_{X_l}X_i  \rangle \over \langle X_i, X_i \rangle}
   &=&-c_{\frac{-k-l}{2}, k}c_{l,\frac{-k-l}{2}} A_{l,\frac{-k-l}{2}}  \frac{\langle X_{\frac{k-l}{2}}, X_{-\frac{k-l}{2}}\rangle}{\langle X_{\frac{-k-l}{2}}, X_{\frac{-k-l}{2}} \rangle }\\
   &&+ c_{\frac{-k+l}{2},k} c_{l,\frac{-k+l}{2}}  B_{l,\frac{-k+l}{2}} \frac{\langle X_{\frac{k+l}{2}}, X_{\frac{k+l}{2}}\rangle}{\langle X_{\frac{-k+l}{2}}, X_{\frac{-k+l}{2}} \rangle }.
\end{eqnarray*}

On the other hand,
\begin{eqnarray*}
\langle [Y_i,X_k],\ad^\star_{X_l}Y_i  \rangle
&=&\langle  \frac{1}{2} c_{ik}(-Y_{i+k}+Y_{i-k}), \frac{1}{2} c_{li}(A_{li}Y_{l+i}-B_{li}Y_{l-i}) \rangle\\
&=& \frac{1}{4} c_{ik}c_{li}\langle  - Y_{i+k} +Y_{i-k}, A_{li}Y_{l+i}+B_{li}Y_{l-i} \rangle
\end{eqnarray*}
and hence
\begin{eqnarray*}
   4 \sum_i {\langle [Y_i,X_k],\ad^\star_{X_l}Y_i  \rangle \over \langle Y_i, Y_i \rangle}
   &=&-c_{\frac{-k-l}{2}, k}c_{l,\frac{-k-l}{2}} A_{l,\frac{-k-l}{2}}  \frac{\langle Y_{\frac{k-l}{2}}, Y_{-\frac{k-l}{2}}\rangle}{\langle Y_{\frac{-k-l}{2}}, Y_{\frac{-k-l}{2}} \rangle }\\
   &&- c_{\frac{-k+l}{2},k} c_{l,\frac{-k+l}{2}}  B_{l,\frac{-k+l}{2}} \frac{\langle Y_{\frac{k+l}{2}}, Y_{\frac{k+l}{2}}\rangle}{\langle Y_{\frac{-k+l}{2}}, Y_{\frac{-k+l}{2}} \rangle }\\
   &=&c_{\frac{-k-l}{2}, k}c_{l,\frac{-k-l}{2}} A_{l,\frac{-k-l}{2} }  \frac{\langle Y_{\frac{k-l}{2}}, Y_{\frac{k-l}{2}}\rangle}{\langle Y_{\frac{-k-l}{2}}, Y_{\frac{-k-l}{2}} \rangle }\\
   &&-c_{\frac{-k+l}{2},k} c_{l,\frac{-k+l}{2}}  B_{l,\frac{-k+l}{2}} \frac{\langle Y_{\frac{k+l}{2}}, Y_{\frac{k+l}{2}}\rangle}{\langle Y_{\frac{-k+l}{2}}, Y_{\frac{-k+l}{2}} \rangle }.
\end{eqnarray*}
Summing up these two results, we obtain the first claimed equality. 
\end{proof}

\begin{proof}[Proof for \autoref{eq:vanishing_off_diagonal4}]
We show the claim only for the case $E_k=X_k, E_l=X_l$ with $k\neq l$, since the claims for the case $E_k=Y_k, E_l=Y_l$ with $k\neq l$ and for $E_k=X_k, E_l=Y_l$ follows from the same computational routine.  

 Let $A_{ab}={\lambda_b\over \lambda_{a+b}}, B_{ab}={\lambda_b\over \lambda_{a-b}}$. For each $i$, we have
\begin{eqnarray*}
\langle [X_i,X_k],\ad^\star_{X_i}X_l  \rangle
&=&\langle  \frac{1}{2} c_{ik}(-X_{i+k}+X_{i-k}), \frac{1}{2} c_{il}(A_{il}X_{i+l}-B_{il}X_{i-l}) \rangle\\
&=& \frac{1}{4} c_{ik}c_{il}\langle - X_{i+k} +X_{i-k}, A_{il}X_{i+l}-B_{il}X_{i-l} \rangle.
\end{eqnarray*}  
\begin{eqnarray*}
   4 \sum_i {\langle [X_i,X_k],\ad^\star_{X_i}X_l  \rangle \over \langle X_i, X_i \rangle}
   &=&-c_{\frac{-k-l}{2}, k}c_{\frac{-k-l}{2},l} A_{\frac{-k-l}{2}, l}  \frac{\langle X_{\frac{k-l}{2}}, X_{-\frac{k-l}{2}}\rangle}{\langle X_{\frac{-k-l}{2}}, X_{\frac{-k-l}{2}} \rangle }\\
   &&+ c_{\frac{-k+l}{2},k} c_{\frac{-k+l}{2},l}  B_{\frac{-k+l}{2},l} \frac{\langle X_{\frac{k+l}{2}}, X_{-\frac{k+l}{2}}\rangle}{\langle X_{\frac{-k+l}{2}}, X_{\frac{-k+l}{2}} \rangle }\\
   &=&-c_{\frac{-k-l}{2}, k}c_{\frac{-k-l}{2},l} A_{\frac{-k-l}{2}, l}  \frac{\langle X_{\frac{k-l}{2}}, X_{\frac{k-l}{2}}\rangle}{\langle X_{\frac{-k-l}{2}}, X_{\frac{-k-l}{2}} \rangle }\\
   &&+c_{\frac{-k+l}{2},k} c_{\frac{-k+l}{2},l}  B_{\frac{-k+l}{2},l} \frac{\langle X_{\frac{k+l}{2}}, X_{\frac{k+l}{2}}\rangle}{\langle X_{\frac{-k+l}{2}}, X_{\frac{-k+l}{2}} \rangle }.
\end{eqnarray*}
On the other hand,
\begin{eqnarray*}
\langle [Y_i,X_k],\ad^\star_{Y_i}X_l  \rangle
&=&\langle  \frac{1}{2} c_{ik}(-Y_{i+k}+Y_{i-k}), \frac{1}{2} c_{il}(A_{il}Y_{i+l}-B_{il}Y_{i-l}) \rangle\\
&=& \frac{1}{4} c_{ik}c_{il}\langle - Y_{i+k} +Y_{i-k}, A_{il}Y_{i+l}-B_{il}Y_{i-l} \rangle
\end{eqnarray*}  
and hence
\begin{eqnarray*}
   4 \sum_i {\langle [Y_i,X_k],\ad^\star_{Y_i}X_l  \rangle \over \langle Y_i, Y_i \rangle}
   &=&-c_{\frac{-k-l}{2}, k}c_{\frac{-k-l}{2},l} A_{\frac{-k-l}{2}, l}  \frac{\langle Y_{\frac{k-l}{2}}, Y_{-\frac{k-l}{2}}\rangle}{\langle Y_{\frac{-k-l}{2}}, Y_{\frac{-k-l}{2}} \rangle }\\
   &&+ c_{\frac{-k+l}{2},k} c_{\frac{-k+l}{2},l}  B_{\frac{-k+l}{2},l} \frac{\langle Y_{\frac{k+l}{2}}, Y_{-\frac{k+l}{2}}\rangle}{\langle Y_{\frac{-k+l}{2}}, Y_{\frac{-k+l}{2}} \rangle }\\
   &=&c_{\frac{-k-l}{2}, k}c_{\frac{-k-l}{2},l} A_{\frac{-k-l}{2}, l}  \frac{\langle Y_{\frac{k-l}{2}}, Y_{\frac{k-l}{2}}\rangle}{\langle Y_{\frac{-k-l}{2}}, Y_{\frac{-k-l}{2}} \rangle }\\
   &&- c_{\frac{-k+l}{2},k} c_{\frac{-k+l}{2},l}  B_{\frac{-k+l}{2},l} \frac{\langle Y_{\frac{k+l}{2}}, Y_{\frac{k+l}{2}}\rangle}{\langle Y_{\frac{-k+l}{2}}, Y_{\frac{-k+l}{2}} \rangle }.
\end{eqnarray*}
Summing up these two results, we obtain the first claimed equality. 
\end{proof}

\subsection{Proofs of propositions in \autoref{sec:diagonal_entries}}\label{sec:proofs  diagonal entries}
We show the identities given in \autoref{lem:diagonal_entries}. Throughout the following proofs, we set $A_{ab}=1+\frac{\lambda_b-\lambda_a}{\lambda_{a+b}}$,  $B_{ab}=1+\frac{\lambda_b-\lambda_a}{\lambda_{a-b}}$ and use
\begin{equation*}
\frac{\langle X_{i\pm k}, X_{i\pm k}\rangle}{\langle X_i, X_i\rangle}
=
\frac{\langle Y_{i\pm k}, Y_{i\pm k}\rangle}{\langle Y_i, Y_i\rangle} 
=
\frac{-2\pi^2 \lambda_{i\pm k}}{-2\pi^2\lambda_i} = \frac{\lambda_{i\pm k}}{\lambda_i}
\end{equation*}
for each $i,k$.

\begin{proof}[Proof for \autoref{eq:diagonal_entry1}]
We first show the identity for $E_k=X_k$. As in the proof of \autoref{eq:vanishing_off_diagonal1}, we have
for each $i$ that,
\begin{eqnarray*}
\langle\nabla_{X_i}X_k,\nabla_{X_k} X_i \rangle
&=&\frac{1}{4}\langle \frac{1}{2} c_{ik}(-A_{ik}X_{i+k}+B_{ik}X_{i-k}), \frac{1}{2} c_{ki}(-A_{ki}X_{k+i}+B_{ki}X_{k-i}) \rangle
\\
&=& -\frac{1}{16} c_{ik}^2 \langle -A_{ik} X_{i+k}+B_{ik} X_{i-k}, -A_{ki}X_{k+i}+B_{ki}X_{i-k} \rangle.
\end{eqnarray*}
Here, the cross terms $\langle X_{i+k},X_{i-k}\rangle$  only survive if 
$2i=0$ or $2k= 0$, but these cases are excluded from the computation as $X_{0}, Y_{ 0}$ are not bases of $\su(N)$ (in the same way as their continuous counterparts $\cos(0),\sin(0)$ are not bases of $C^\infty_0(\TT^2)$). In such cases, they will anyway vanish as $i\times k=0$ yields $c_{ik}=0$. 

Hence, we have non-vanishing terms like $\langle X_{i\pm k},X_{i\pm k} \rangle $.
\begin{eqnarray*}
16 \langle\nabla_{X_i}X_k,\nabla_{X_k} X_i \rangle
= -c_{ik}^2\Big(A_{ik}A_{ki}\langle X_{i+k}, X_{i+k}\rangle + B_{ik}B_{ki}\langle X_{i-k},X_{i-k}\rangle \Big).
\end{eqnarray*}

Using $Y_{-m}=-Y_{m}$, we also get
\begin{eqnarray*}
\langle\nabla_{Y_i}X_k,\nabla_{X_k} Y_i \rangle
&=&\frac{1}{4}\langle \frac{1}{2} c_{ik}(-A_{ik}Y_{i+k}+B_{ik}Y_{i-k}), \frac{1}{2} c_{ki}(-A_{ki}Y_{k+i}-B_{ki}Y_{k-i}) \rangle
\\
&=& -\frac{1}{16} c_{ik}^2 \langle -A_{ik} Y_{i+k}+B_{ik} Y_{i-k}, -A_{ki}Y_{k+i}+B_{ki}Y_{i-k} \rangle
\\
&=&-\frac{1}{16} c_{ik}^2 \left( A_{ik}A_{ki} \langle Y_{i+k}, Y_{i+k} \rangle + B_{ik}B_{ki} \langle Y_{i-k},Y_{i-k} \rangle\right) .
\end{eqnarray*}

Summing up the above computations and applying  $\langle Y_{i},Y_{i} \rangle =\langle X_{i},X_{i}\rangle$ reads,
   \begin{eqnarray}
    && \sum_i {\langle\nabla_{X_i}X_k,\nabla_{X_k} X_i \rangle \over \langle X_i, X_i \rangle} +      16 \sum_i {\langle\nabla_{Y_i}X_k,\nabla_{X_k} Y_i \rangle \over \langle Y_i, Y_i \rangle}
    \\
   &=&
  - \frac{1}{8} \sum_i c_{ik}^2 \Big(A_{ik}A_{ki} \frac{\langle X_{i+k}, X_{i+k} \rangle}{\langle X_i,X_i \rangle} 
   + B_{ik}B_{ki} \frac{\langle X_{i-k},X_{i-k}\rangle}{\langle X_i,X_i\rangle}\Big)\label{eq:intermiediate diagonal 1}
   \\
   &=&
  - \frac{1}{8} \sum_i c_{ik}^2 \Big(\big(1-\frac{(\lambda_i-\lambda_k)^2}{\lambda_{k+i}^2} \big)\frac{\lambda_{k+i}}{\lambda_{i}} 
   + \big(1-\frac{(\lambda_i-\lambda_k)^2}{\lambda_{k-i}^2}\big) \frac{\lambda_{k-i}}{\lambda_{i}}\Big)
   \\
   &=&
  - \frac{1}{8} \sum_i c_{ik}^2 \Big( \frac{\lambda_{k+i}}{\lambda_i}-\frac{(\lambda_i-\lambda_k)^2\lambda_{k+i}^{-1}}{\lambda_{i}}  
   + \frac{\lambda_{k-i}}{\lambda_{i}} -\frac{(\lambda_i-\lambda_k)^2\lambda_{k-i}^{-1}}{\lambda_{i}}
   \Big)
   \\
   &=&
  - \frac{1}{8} \sum_i c_{ik}^2 \Big( \frac{\lambda_{k+i} - (\lambda_i-\lambda_k)^2\lambda_{k+i}^{-1}  
   + \lambda_{k-i} -(\lambda_i-\lambda_k)^2\lambda_{k-i}^{-1} }{\lambda_{i}}
   \Big)
    \\
   &=&
  - \frac{1}{8} \sum_i c_{ik}^2 \Big( \frac{\lambda_{k+i} + \lambda_{k-i}
  - (\lambda_i-\lambda_k)^2( \lambda_{k+i}^{-1} +\lambda_{k-i}^{-1}) }{\lambda_{i}}
   \Big)
   \\
   &=&
   \sum_i c_{ik}^2 \Big( \frac{ -(\lambda_{k+i} + \lambda_{k-i})
  + (\lambda_i-\lambda_k)^2( \lambda_{k+i}^{-1} +\lambda_{k-i}^{-1}) }{8\lambda_{i}}
   \Big)
  \end{eqnarray}

We now show the claimed identity for $E_k=Y_k$. We have
\begin{eqnarray*}
\langle\nabla_{X_i}Y_k,\nabla_{Y_k} X_i \rangle
&=&\frac{1}{4}\langle \frac{1}{2} c_{ik}(-A_{ik}Y_{i+k}-B_{ik}Y_{i-k}), \frac{1}{2} c_{ki}(-A_{ki}Y_{k+i}+B_{ki}Y_{k-i}) \rangle
\\
&=& -\frac{1}{16} c_{ik}^2\langle -A_{ik} Y_{i+k}-B_{ik} Y_{i-k}, -A_{ki}Y_{i+k}-B_{ki}Y_{i-k} \rangle
\\
&=& -\frac{1}{16} c_{ik}^2\Big( A_{ik}A_{ki}\langle Y_{i+k},Y_{i+k} \rangle + B_{ik}B_{ki}\langle Y_{i-k},Y_{i-k}\rangle \Big)
\end{eqnarray*}
and similarly 
\begin{eqnarray*}
\langle\nabla_{Y_i}Y_k,\nabla_{Y_k} Y_i \rangle
&=&\frac{1}{4}\langle \frac{1}{2} c_{ik}(A_{ki}X_{i+k}+B_{ki}X_{i-k}), \frac{1}{2} c_{ki}(A_{ik}X_{k+i}+B_{ik}X_{k-i}) \rangle
\\
&=& -\frac{1}{16} c_{ik}^2\langle A_{ki}X_{i+k}+ B_{ki}X_{i-k}, A_{ik}X_{k+i} + B_{ik}X_{i-k} \rangle
\\
&=& -\frac{1}{16} c_{ik}^2 \Big( A_{ik}A_{ki} \langle X_{i+k},X_{i+k} \rangle + B_{ik}B_{ki} \langle X_{i-k},X_{i-k}\rangle \Big).
\end{eqnarray*}
As a result we have
\begin{eqnarray*}
&&16\sum_i {\langle\nabla_{X_i}Y_k,\nabla_{Y_k} X_i \rangle \over \langle X_i, X_i \rangle}+{\langle\nabla_{Y_i}Y_k,\nabla_{Y_k} Y_i \rangle \over \langle Y_i, Y_i \rangle} 
\\
&&= - 2 \sum_{i} c_{ik}^2 \left( A_{ik}A_{ki} \frac{\langle X_{i+k}, X_{i+k}\rangle}{\langle X_i, X_i\rangle} + B_{ik}B_{ki} \frac{\langle X_{i-k}, X_{i-k}\rangle}{\langle X_i,X_i\rangle}\right),
\end{eqnarray*}
which is identical to \eqref{eq:intermiediate diagonal 1}. Hence we obtain the same result as the one for $E_k=X_k$.
\end{proof}

\begin{proof}[Proof for \autoref{eq:diagonal_entry2}]
We first show the identity for $E_k=X_k$. As in the proof of \autoref{eq:vanishing_off_diagonal2}, we have
for each $i$ that,
\begin{eqnarray*}
\langle [X_i, X_k], [X_k, X_i] \rangle
&=&\langle \frac{1}{2} c_{ik}(-X_{i+k}+X_{i-k}), \frac{1}{2} c_{ki}(-X_{k+i}+X_{k-i}) \rangle
\\
&=& -\frac{1}{4} c_{ik}^2\langle -X_{i+k}+X_{i-k}, -X_{i+k}+X_{i-k} \rangle
\\
&=& -\frac{1}{4} c_{ik}^2 ( \langle X_{i+k},X_{i+k} \rangle + \langle X_{i-k},X_{i-k}\rangle ),
\end{eqnarray*}
and similarly, we have
\begin{eqnarray*}
\langle[Y_i,X_k], [X_k, Y_i] \rangle
&=&\langle \frac{1}{2} c_{ik}(-Y_{i+k}+Y_{i-k}), \frac{1}{2} c_{ki}(-Y_{k+i}-Y_{k-i}) \rangle
\\
&=& -\frac{1}{4} c_{ik}^2 \langle - Y_{i+k}+ Y_{i-k}, -Y_{i+k} - (-Y_{i-k}) \rangle 
\\
&=& -\frac{1}{4} c_{ik}^2\langle -Y_{i+k}+Y_{i-k}, -Y_{i+k}+Y_{i-k} \rangle 
\\
&=& -\frac{1}{4} c_{ik}^2 ( \langle Y_{i+k}, Y_{i+k}\rangle + \langle Y_{i-k}, Y_{i-k}\rangle ).
\end{eqnarray*}
Summing these equations and noting $\langle X_m,X_m\rangle=\langle Y_m,Y_m\rangle$ yields the first exact sum identity:
\begin{eqnarray*}
&&\sum_i {\langle [X_i, X_k], [X_k, X_i]\rangle\over \langle X_i, X_i\rangle}+ \sum_i {\langle [Y_i, X_k], [X_k, Y_i]\rangle\over \langle Y_i, Y_i\rangle}
\\
&&= -\frac{1}{2} \sum_i c_{ik}^2 \left( \frac{\langle X_{i+k}, X_{i+k}\rangle}{\langle X_i, X_i\rangle} + \frac{\langle X_{i-k}, X_{i-k}\rangle}{\langle X_i,X_i\rangle}\right)
\\
&&=-\frac{1}{2} \sum_i c_{ik}^2 \left( \frac{\lambda_{i+k} +\lambda_{i-k}}{\lambda_i}\right).
\end{eqnarray*}

For the equality for $E_k=Y_k$, we use exactly parallel expansions. For each $i$ we have
\begin{eqnarray*}
\langle [X_i, Y_k], [Y_k, X_i]\rangle
&=&\langle \frac{1}{2} c_{ik}(-Y_{i+k}-Y_{i-k}), \frac{1}{2} c_{ki}(-Y_{k+i}+Y_{k-i}) \rangle
\\
&=& -\frac{1}{4} c_{ik}^2\langle -Y_{i+k}-Y_{i-k}, -Y_{i+k}+ (-Y_{i-k}) \rangle
\\
&=& -\frac{1}{4} c_{ik}^2\langle -Y_{i+k}-Y_{i-k}, -Y_{i+k}-Y_{i-k} \rangle
\\
&=& -\frac{1}{4} c_{ik}^2 ( \langle Y_{i+k}, Y_{i+k} \rangle + \langle Y_{i-k}, Y_{i-k} \rangle )
\end{eqnarray*}
and
\begin{eqnarray*}
\langle[Y_i,Y_k], [Y_k, Y_i] \rangle
&=&\langle \frac{1}{2} c_{ik}(X_{i+k}+X_{i-k}), \frac{1}{2} c_{ki}(X_{k+i}+X_{k-i}) \rangle
\\
&=& -\frac{1}{4} c_{ik}^2 \langle X_{i+k}+ X_{i-k}, X_{i+k}+X_{i-k} \rangle 
\\
&=& -\frac{1}{4} c_{ik}^2 ( \langle X_{i+k}, X_{i+k}\rangle + \langle X_{i-k}, X_{i-k}\rangle ).
\end{eqnarray*}
As a result we get
\begin{eqnarray*}
\sum_i {\langle [X_i, Y_k], [Y_k, X_i]\rangle\over \langle X_i, X_i\rangle}+ \sum_i {\langle [Y_i, Y_k], [Y_k, Y_i]\rangle\over \langle Y_i, Y_i\rangle} = -\frac{1}{2} \sum_i c_{ik}^2 \left( \frac{\langle X_{i+k}, X_{i+k}\rangle}{\langle X_i, X_i\rangle} + \frac{\langle X_{i-k}, X_{i-k}\rangle}{\langle X_i,X_i\rangle}\right).
\end{eqnarray*}
\end{proof}

\begin{proof}[Proof for \autoref{eq:diagonal_entry3}]
We first show the identity for $E_k=X_k$. Applying $c_{ik}=-c_{ki}$ and  $X_{i-k}=X_{k-i}$ yields
\begin{eqnarray*}
\langle [X_i, X_k], \ad^\star_{X_k}X_i \rangle
&=& \langle -\frac{1}{2} c_{ik}(X_{i+k} - X_{i-k}) , -\frac{1}{2} c_{ik}\lambda_i \Big( \frac{X_{i+k}}{\lambda_{i+k}} - \frac{X_{k-i}}{\lambda_{k-i}}\Big)\rangle
\\
&=& \frac{1}{4} c_{ik}^2 \lambda_i \langle X_{i+k} - X_{i-k}, \frac{X_{i+k}}{\lambda_{i+k}} - \frac{X_{i-k}}{\lambda_{i-k}}\rangle
\\
&=& \frac{1}{4} c_{ik}^2 \lambda_i \left( \frac{\langle X_{i+k},X_{i+k} \rangle}{\lambda_{i+k}} + \frac{\langle X_{i-k}, X_{i-k}\rangle}{\lambda_{i-k}} \right).
\end{eqnarray*}
Using $Y_{i-k}=-Y_{k-i}$ we also get
\begin{eqnarray*}
\langle [Y_i, X_k], \ad^\star_{X_k}Y_i \rangle
&=& \langle -\frac{1}{2} c_{ik}(Y_{i+k} - Y_{i-k}) , -\frac{1}{2} c_{ik}\lambda_i \Big( \frac{Y_{k+i}}{\lambda_{k+i}} - \frac{Y_{i-k}}{\lambda_{i-k}}\Big)\rangle
\\
&=& \frac{1}{4} c_{ik}^2 \lambda_i \left( \frac{\langle Y_{i+k}, Y_{i+k} \rangle}{\lambda_{i+k}} + \frac{\langle Y_{i-k}, Y_{i-k}\rangle}{\lambda_{i-k}} \right).
\end{eqnarray*}
As a result we have
\begin{eqnarray*}
&&\sum_i {\langle [X_i, X_k], \ad^\star_{X_k}X_i \rangle\over \langle X_i, X_i\rangle}+ \sum_i {\langle [Y_i, X_k], \ad^\star_{X_k}Y_i\rangle\over \langle Y_i, Y_i\rangle}
\\
&=&
\frac{1}{2} \sum_i c_{ik}^2 \lambda_i \left( \frac{1}{\lambda_{i+k}} \frac{\langle X_{i+k}, X_{i+k}\rangle}{\langle X_i, X_i \rangle} + \frac{1}{\lambda_{i-k}}\frac{\langle X_{i-k},X_{i-k}\rangle}{\langle X_i, X_i\rangle} \right)
\\
&=&
\frac{1}{2} \sum_i c_{ik}^2 \lambda_i \left( \frac{1}{\lambda_{i+k}} \frac{\lambda_{i+k}}{\lambda_i} + \frac{1}{\lambda_{i-k}}\frac{\lambda_{i-k}}{\lambda_i} \right)
=
 \sum_i c_{ik}^2
\end{eqnarray*}

For the equality for $E_k=Y_k$, we notice that
\begin{eqnarray*}
\langle [X_i, Y_k], \ad^\star_{Y_k}X_i\rangle
&=& \langle -\frac{1}{2} c_{ik}(Y_{i+k} + Y_{i-k}) , -\frac{1}{2} c_{ik}\lambda_i \Big(\frac{Y_{i+k}}{\lambda_{i+k}} + \frac{Y_{i-k}}{\lambda_{i-k}} \Big) \rangle
\\
&=& \frac{1}{4} c_{ik}^2 \lambda_i \left( \frac{\langle Y_{i+k}, Y_{i+k} \rangle}{\lambda_{i+k}} + \frac{\langle Y_{i-k}, Y_{i-k} \rangle}{\lambda_{i-k}}\right).
\end{eqnarray*}
Likewise we have, 
\begin{eqnarray*}
\langle [Y_i, Y_k], \ad^\star_{Y_k}Y_i\rangle
&=& \langle \frac{1}{2} c_{ik}(X_{i+k} + X_{i-k}) , \frac{1}{2} c_{ik}\lambda_i\Big( \frac{X_{i+k}}{\lambda_{i+k}} + \frac{X_{i-k}}{\lambda_{i-k}}\Big)\rangle
\\
&=& \frac{1}{4} c_{ik}^2 \lambda_i \left( \frac{\langle X_{i+k}, X_{i+k} \rangle}{\lambda_{i+k}} + \frac{\langle X_{i-k}, X_{i-k} \rangle}{\lambda_{i-k}}\right).
\end{eqnarray*}
which imply that
\begin{eqnarray*}
&&\sum_i\Big( {\langle [X_i, Y_k], \ad^\star_{Y_k}X_i \rangle\over \langle X_i, X_i\rangle}+ {\langle [Y_i, Y_k], \ad^\star_{Y_k}Y_i\rangle\over \langle Y_i, Y_i\rangle}\Big)
\\
&=&
\frac{1}{2} \sum_i c_{ik}^2 \lambda_i \left( \frac{1}{\lambda_{i+k}} \frac{\langle X_{i+k}, X_{i+k}\rangle}{\langle X_i, X_i \rangle} + \frac{1}{\lambda_{i-k}}\frac{\langle X_{i-k},X_{i-k}\rangle}{\langle X_i, X_i\rangle} \right)\\
&=&
\frac{1}{2} \sum_i c_{ik}^2 \lambda_i \left( \frac{1}{\lambda_{i+k}} \frac{\lambda_{i+k}}{\lambda_i} + \frac{1}{\lambda_{i-k}}\frac{\lambda_{i-k}}{\lambda_i} \right)
=
 \sum_i c_{ik}^2.
\end{eqnarray*}
\end{proof}

\begin{proof}[Proof for \autoref{eq:diagonal_entry4}]
We first show the identity for $E_k=X_k$. 
 For each $i$, we have
\begin{eqnarray*}
\langle [X_i, X_k], \ad^\star_{X_i}X_k \rangle
&=& \langle -\frac{1}{2} c_{ik}(X_{i+k} - X_{i-k}) , \frac{1}{2} c_{ik}\lambda_k \Big( \frac{X_{i+k}}{\lambda_{i+k}} - \frac{X_{i-k}}{\lambda_{i-k}}\Big)\rangle
\\
&=& -\frac{1}{4} c_{ik}^2 \lambda_k \langle X_{i+k} - X_{i-k}, \frac{X_{i+k}}{\lambda_{i+k}} - \frac{X_{i-k}}{\lambda_{i-k}}\rangle
\\
&=& -\frac{1}{4} c_{ik}^2 \lambda_k \left( \frac{\langle X_{i+k},X_{i+k} \rangle}{\lambda_{i+k}} + \frac{\langle X_{i-k}, X_{i-k}\rangle}{\lambda_{i-k}} \right).
\end{eqnarray*}
and
\begin{eqnarray*}
\langle [Y_i, X_k], \ad^\star_{Y_i}X_k \rangle
&=& \langle -\frac{1}{2} c_{ik}(Y_{i+k} - Y_{i-k}) , \frac{1}{2} c_{ik}\lambda_k \Big( \frac{Y_{i+k}}{\lambda_{i+k}} - \frac{Y_{i-k}}{\lambda_{i-k}}\Big)\rangle
\\
&=& -\frac{1}{4} c_{ik}^2 \lambda_k \langle Y_{i+k} - Y_{i-k}, \frac{Y_{i+k}}{\lambda_{i+k}} - \frac{Y_{i-k}}{\lambda_{i-k}}\rangle
\\
&=& -\frac{1}{4} c_{ik}^2 \lambda_k \left( \frac{\langle Y_{i+k}, Y_{i+k} \rangle}{\lambda_{i+k}} + \frac{\langle Y_{i-k}, Y_{i-k}\rangle}{\lambda_{i-k}} \right).
\end{eqnarray*}
Using  $\langle X_i,X_i\rangle=\langle Y_i,Y_i\rangle$  we get
\begin{eqnarray*}
&&\sum_i {\langle [X_i, X_k], \ad^\star_{X_i}X_k \rangle\over \langle X_i, X_i\rangle}+ \sum_i {\langle [Y_i, X_k], \ad^\star_{Y_i}X_k\rangle\over \langle Y_i, Y_i\rangle}
\\
&=&
-\frac{1}{2} \sum_i c_{ik}^2 \lambda_k \left( \frac{1}{\lambda_{i+k}} \frac{\langle X_{i+k}, X_{i+k}\rangle}{\langle X_i, X_i\rangle} + \frac{1}{\lambda_{i-k}} \frac{\langle X_{i-k}, X_{i-k}\rangle}{\langle X_i, X_i\rangle} \right)
\\
&=&
-\frac{1}{2} \sum_i c_{ik}^2 \lambda_k \left( \frac{1}{\lambda_{i+k}} 
\frac{\lambda_{i+k}}{\lambda_i} + \frac{1}{\lambda_{i-k}} 
\frac{\lambda_{i-k}}{\lambda_i} \right)
=
- \sum_i c_{ik}^2   \frac{\lambda_k}{\lambda_{i}}.
\end{eqnarray*}

We then show the equality for $E_k=Y_k$. Notice that, for each $i$ 
\begin{eqnarray*}
\langle [X_i, Y_k], \ad^\star_{X_i}Y_k\rangle
&=& \langle -\frac{1}{2} c_{ik}(Y_{i+k} + Y_{i-k}) , \frac{1}{2} c_{ik}\lambda_k \Big(\frac{Y_{i+k}}{\lambda_{i+k}} + \frac{Y_{i-k}}{\lambda_{i-k}} \Big) \rangle
\\
&=& -\frac{1}{4} c_{ik}^2 \lambda_k \langle Y_{i+k} + Y_{i-k}, \frac{Y_{i+k}}{\lambda_{i+k}} + \frac{Y_{i-k}}{\lambda_{i-k}}\rangle
\\
&=& -\frac{1}{4} c_{ik}^2 \lambda_k \left( \frac{\langle Y_{i+k}, Y_{i+k} \rangle}{\lambda_{i+k}} + \frac{\langle Y_{i-k}, Y_{i-k} \rangle}{\lambda_{i-k}}\right).
\end{eqnarray*}
and
\begin{eqnarray*}
\langle [Y_i, Y_k], \ad^\star_{Y_i}Y_k\rangle
&=& \langle \frac{1}{2} c_{ik}(X_{i+k} + X_{i-k}) , -\frac{1}{2} c_{ik}\lambda_k \Big( \frac{X_{i+k}}{\lambda_{i+k}} + \frac{X_{i-k}}{\lambda_{i-k}}\Big)\rangle
\\
&=& -\frac{1}{4} c_{ik}^2 \lambda_k \langle X_{i+k} + X_{i-k}, \frac{X_{i+k}}{\lambda_{i+k}} + \frac{X_{i-k}}{\lambda_{i-k}}\rangle
\\
&=& -\frac{1}{4} c_{ik}^2 \lambda_k \left( \frac{\langle X_{i+k}, X_{i+k} \rangle}{\lambda_{i+k}} + \frac{\langle X_{i-k}, X_{i-k} \rangle}{\lambda_{i-k}}\right).
\end{eqnarray*}
As before, 
\begin{eqnarray*}
&&\sum_i {\langle [X_i, Y_k], \ad^\star_{X_i}Y_k\rangle\over \langle X_i, X_i\rangle}+ \sum_i {\langle [Y_i, Y_k], \ad^\star_{Y_i}Y_k\rangle\over \langle Y_i, Y_i\rangle}
\\
&=&
-\frac{1}{2} \sum_i c_{ik}^2 \lambda_k \left( \frac{1}{\lambda_{i+k}} \frac{\langle X_{i+k}, X_{i+k}\rangle}{\langle X_i, X_i\rangle} + \frac{1}{\lambda_{i-k}} \frac{\langle X_{i-k}, X_{i-k}\rangle}{\langle X_i, X_i\rangle} \right)\\
&=&
-\frac{1}{2} \sum_i c_{ik}^2 \lambda_k \left( \frac{1}{\lambda_{i+k}} 
\frac{\lambda_{i+k}}{\lambda_i} + \frac{1}{\lambda_{i-k}} 
\frac{\lambda_{i-k}}{\lambda_i} \right)
=
- \sum_i c_{ik}^2   \frac{\lambda_k}{\lambda_{i}}.
\end{eqnarray*}
\end{proof}

\subsection{Proofs of propositions in \autoref{sec:sectional curvature}}\label{sec:proofs sectional curvature}
\begin{proof}[Proof of \autoref{lem:sectional_curvature_qualitative}]
   First note that the spectral Laplacian $\Delta^{\rm spec}_N$ on $\SU(N)$ shares the same eigenvalues $\lambda_i\coloneqq -|i|^2$ with the continuous Laplacian $\Delta$ on $\HDiff(\TT^2)$, and the formulas for the sectional curvatures \eqref{eq:K(X_k,X_l)} and \eqref{eq:continuous sectional curvature} differ only in the factor $c_{kl}^2$ and $|k\times l|^2$. Since these factors and the denominator are positive, it suffices to show the claimed inequality for their shared numerator.

   Using $2(\lambda_k +\lambda_l)=- |k+l|^2-|k-l|^2 = -2|k|^2 - 2 |l|^2=\lambda_{k+l}+\lambda_{k-l}$,
     the first and the third terms in the numerator are summed up to $2\lambda_k + 2\lambda_l$. 

Using the Cauchy-Schwartz inequality, we also have
   \begin{align*}
      (\lambda_l-\lambda_k)^2= (|l|^2-|k|^2)^2=((k+l)\cdot (k-l))^2 \leq |k-l|^2|k+l|^2=\lambda_{k-l}\lambda_{k+l}.
   \end{align*}
Here the equality holds if and only if $k$ and $l$ are parallel to each other.  Using these identities, the numerator of \eqref{eq:K(X_k,X_l)} reduces to
   \begin{align*}
       &c_{kl}^2\left(-(\lambda_l-\lambda_k)^2(\lambda_{k-l}^ {-1}+\lambda_{k+l}^{-1})-4(\lambda_k+\lambda_l)+3(\lambda_{k-l}+\lambda_{k+l})\right)\\
       &\leq c_{kl}^2 (-(\lambda_{k+l}+\lambda_{k-l})+\lambda_{k+l}+\lambda_{k-l}))=0
   \end{align*}
   where the equality holds only if $k$ and $l$ are parallel. 
   
\end{proof}

\subsection{Proofs of propositions in \autoref{sec:Coriolis-force}}\label{sec:proofs Coriolis-force}

\begin{proof}[Proof of \autoref{prop:Riemannian tensors for hat su(N)}]
The formula  \eqref{eq:cov-der-centr-ext} for the covariant derivative implies that
     \begin{eqnarray*}
     \widehat{\nabla}_{(X_i,0)}(X_i,0)=\widehat{\ad}_{(X_i,0)}(X_i,0) -2\widehat{\ad}^\star_{(X_i,0)}(X_i,0)=(0,0).
    \end{eqnarray*}
Using the formula \eqref{eq:curv-centra-ext} for the Riemannian curvature tensor and  $\omega_N(X_i,X_k)=\omega_N(X_k,X_i)=0$ (\autoref{lem: identities for omega and T}), we have
   \begin{eqnarray*}
     &&\ll  \widehat{R}\big((X_i,0),(X_k,a)\big)(X_k,a)  ,   (X_i,0)  \gg  
     \\
     &=&\ll\widehat{\nabla}_{(X_i,0)}(X_k,a)    ,   \widehat{\nabla}_{(X_k,a)}(X_i,0) \gg   
     \\
     &&+ \frac{1}{2}\ll \big([X_i,X_k],\omega_N(X_i,X_k) \big)  ~ ,  ~\big([X_k,X_i],\omega_N(X_k,X_i)\big)   
     \\
     &&+ \widehat{\ad}^\star_{(X_k,a)}(X_i,0)    -  \widehat{\ad}^\star_{(X_i,0)}(X_k,a)    \gg\\
     &=&\ll\big( {\nabla}_{X_i}X_k - \frac{1}{2}aT_NX_i ~ , ~ \frac{1}{2}\omega_N(X_i,X_k)\big)~,~   \big({\nabla}_{X_k}X_i- \frac{1}{2}aT_NX_i ~ , ~ \frac{1}{2}\omega_N(X_k,X_i)\big) \gg   
     \\
     &&+ \frac{1}{2}\ll \big([X_i,X_k]~,~ 0 \big)  ~ ,  ~\big([X_k,X_i]   
     + {\ad}^\star_{X_k}X_i    -  {\ad}^\star_{X_i}X_k   -aT_NX_i ~,0~\big)\gg
     \\
     &=&\langle {\nabla}_{X_i}X_k - \frac{1}{2}aT_NX_i ~,~   {\nabla}_{X_k}X_i- \frac{1}{2}aT_NX_i \rangle   
     \\
     &&+ \frac{1}{2}\langle [X_i,X_k] ~ ,  ~[X_k,X_i]   
     + {\ad}^\star_{X_k}X_i    -  {\ad}^\star_{X_i}X_k   -aT_NX_i \rangle
     \\
     &=&\langle R(X_i,X_k)X_k,X_i\rangle + \frac{a^2}{4}\|T_NX_i\|^2 
     \\
     &&- \frac{a}{2}\langle {\nabla}_{X_i}X_k  ~,~   T_NX_i \rangle
     - \frac{a}{2}\langle {\nabla}_{X_k}X_i  ~,~   T_NX_i \rangle
     - \frac{a}{2}\langle [X_i,X_k]  ~,~   T_NX_i \rangle
     \\
     &=&\langle R(X_i,X_k)X_k,X_i\rangle + \frac{a^2}{4}\|T_NX_i\|^2.
     \end{eqnarray*}

For the second part, we proceed similarly:
\begin{eqnarray*}
     &&\ll  \widehat{R}\big((Y_i,0),(X_k,a)\big)(X_k,a)  ,   (Y_i,0)  \gg  
     \\
     &=&\ll\widehat{\nabla}_{(Y  _i,0)}(X_k,a)    ,   \widehat{\nabla}_{(X_k,a)}(Y_i,0) \gg   
     + \frac{1}{2}\ll \big([Y_i,X_k],\omega_N(Y_i,X_k) \big)  ~ ,  ~\big([X_k,Y_i],\omega_N(X_k,Y_i)\big)   
     \\
     &&+ \widehat{\ad}^\star_{(X_k,a)}(Y_i,0)    -  \widehat{\ad}^\star_{(Y_i,0)}(X_k,a)    \gg
     \\
     &=&\ll\big( {\nabla}_{Y_i}X_k - \frac{1}{2}aT_NY_i ~ , ~ \frac{1}{2}\omega_N(Y_i,X_k)\big)~,~   \big({\nabla}_{X_k}Y_i- \frac{1}{2}aT_NY_i ~ , ~ \frac{1}{2}\omega_N(X_k,Y_i)\big) \gg   
     \\
     &&+ \frac{1}{2}\ll \big([Y_i,X_k]~,~  \omega_N(Y_i,X_k) \big)  ~ ,  ~\big([X_k,Y_i]   
     + {\ad}^\star_{X_k}Y_i    -  {\ad}^\star_{Y_i}X_k   -aT_NY_i ~, \omega_N(X_k,Y_i) ~\big)\gg
\end{eqnarray*}
\begin{eqnarray*}
     &=&\langle {\nabla}_{Y_i}X_k - \frac{1}{2}aT_NY_i ~,~   {\nabla}_{X_k}Y_i- \frac{1}{2}aT_NY_i \rangle   
     -\frac{1}{4}\omega_N(Y_i,X_k)^2
     \\
     &&+ \frac{1}{2}\langle [Y_i,X_k] ~ ,  ~[X_k,Y_i]   
     + {\ad}^\star_{X_k}Y_i    -  {\ad}^\star_{Y_i}X_k   -aT_NY_i \rangle 
     -\frac{1}{2}\omega_N(Y_i,X_k)^2
     \\
     &=&\langle {\nabla}_{Y_i}X_k   ~,~   {\nabla}_{X_k}Y_i  \rangle
     -\frac{a}{2}\cancelto{0}{\langle {\nabla}_{Y_i}X_k  ~,~  T_NY_i \rangle}
     -\frac{a}{2}\cancelto{0}{\langle T_NY_i ~,~   {\nabla}_{X_k}Y_i \rangle}
     +\frac{1}{4}a^2\|T_NY_i\|^2  -\frac{1}{4}\omega_N(Y_i,X_k)^2
     \\
     &&+ \frac{1}{2}\langle [Y_i,X_k] ~ ,  ~[X_k,Y_i]   
     + {\ad}^\star_{X_k}Y_i    -  {\ad}^\star_{Y_i}X_k   -aT_NY_i \rangle 
     -\frac{1}{2}\omega_N(Y_i,X_k)^2
     \\
     &=&\langle R(Y_i,X_k)X_k,Y_i  \rangle
     +\frac{1}{4}a^2\|T_NY_i\|^2  -\frac{3}{4}\omega_N(Y_i,X_k)^2.
     \end{eqnarray*}

     To compute the third part, first note that the covariant derivative formula
      \eqref{eq:cov-der-centr-ext} implies that
     \begin{eqnarray*}
     &&\widehat{\nabla}_{(0,b)}(0,b)=(0,0),\\
     &&\widehat{\nabla}_{(0,b)}(X_k,a)  =(-\frac{1}{2}b T_N X_k,0).
    \end{eqnarray*}
     Therefore we have 
     \begin{eqnarray*}
          \ll  \widehat{R}\big((0,b),(X_k,a)\big)(X_k,a)  ,   (0,b)  \gg 
           &=&\ll\widehat{\nabla}_{(0,b)}(X_k,a)    ,   \widehat{\nabla}_{(X_k,a)}(0,b) \gg   
     \\
     &&+ \frac{1}{2}\ll \big([0,X_k],\omega_N(0,X_k) \big)  ~ ,  ~\big([X_k,0],\omega_N(X_k,0)\big)   
     \\
     &&+ \widehat{\ad}^\star_{(X_k,a)}(0,b)    -  \widehat{\ad}^\star_{(0,b)}(X_k,a)    \gg
     \\
     &=&\ll (-\frac{1}{2}b T_N X_k,0),(-\frac{1}{2}b T_N X_k,0) \gg +0\\
     &=&\frac{1}{4}b^2 \|T_N X_k \|^2.
     \end{eqnarray*}
The forth, fifth, and the sixth identities are computed in the same way as the first, second, and the third parts respectively. 

\end{proof}  

\begin{proof}[Proof of \autoref{lem:infinite_lambda_sum}]
Since
$$\sum_{i\in Z_M^+}\frac{i_1^2}{\lambda_i^2} = 2\sum_{i\in Z_M}\frac{i_1^2}{\lambda_i^2},$$ we evaluate $\sum_{i\in Z_M}\frac{i_1^2}{\lambda_i^2}$ as it is simpler.

    First observe the decomposition of grid points $Z_N$ into squared shaped regions. That is, $Z_N=\bigsqcup_{m=1}^M z_m$ where $ z_m \coloneqq \{i \text{ s.t. } |i|_{\infty}=m \text{ i.e, } i_1 = m \text{ or } i_2 =m \}$ and $\# z_m = (2m+1)^2-(2m-1)^2=8m$ 

    Note also that within the squared region for each $m$,  $\min_{i\in z_m} \lambda_i=-2\frac{N^2}{\pi^2}\sin^2(\frac{i_1 \pi}{N})$ is attained by points with $|i_1|=|i_2|=m$, and $\max_{i\in z_m} \lambda_i=-\frac{N^2}{\pi^2}\sin^2(\frac{i_1 \pi}{N})$ is attained by points with $i_1=0$ or $i_2=0$.
    Therefore we have 
    \begin{eqnarray*}
          \sum_{i\in Z_N}\frac{i_1^2}{\lambda_i^2} 
          &=& \sum_{m=1}^M\sum_{i\in z_m} \frac{i_1^2}{\lambda_i^2}
          \geq \sum_{m=1}^M\sum_{i_1=1, i\in z_m} \frac{i_1^2}{\lambda_i^2}
          \geq \sum_{m=1}^M 8m \cdot \frac{m^2}{ \lambda_{m,m}^2}
          =8\sum_{m=1}^M {m^3\over \lambda_{m,m}^2}\\
          &=& 8\sum_{m=1}^M \frac{m^3}{2^2\frac{(2M+1)^4}{\pi^4} \sin^4(\frac{m \pi}{2M+1})}.
    \end{eqnarray*}
    We now use Jordan's inequality $\theta\geq \sin(\theta)\geq \frac{2}{\pi}\theta $ for $\theta\in[0,\pi/2]$, which is applicable for $\theta=(\frac{m\pi}{2M+1})<\frac{2}{\pi}$ as $0\leq m\leq M$.

    Therefore we have 
    \begin{eqnarray*}
        \frac{m^3}{\frac{(2M+1)^4}{\pi^4} \sin^4(\frac{m \pi}{2M+1})}
        \geq \frac{m^3}{m^4}=\frac{1}{m},
    \end{eqnarray*}
    and hence 
    \begin{eqnarray*}
         \sum_{i\in Z_N}\frac{i_1^2}{\lambda_i^2}
         \geq 2 \sum_{m=1}^M \frac{1}{m}=\infty.
    \end{eqnarray*}

    The other claimed (in)equalities are also verified using $\max_{i\in Z_m}\lambda_i=\lambda_{(m,m)}$ or $\min_{i\in Z_m}\lambda_i=\lambda_{(m,0)}$ and one side of Jordan's inequality.
    To show the second equality, we compute
        \begin{eqnarray*}
          \sum_{i\in Z_N}\frac{i_1^2}{\lambda_i^2} 
          &=& \sum_{m=1}^M\sum_{i\in z_m} \frac{i_1^2}{\lambda_i^2}
          \leq \sum_{m=1}^M 8m \cdot \frac{m^2}{ \lambda_{m,0}^2}
          =8\sum_{m=1}^M {m^3\over \lambda_{m,0}^2}\\
          &=& 8\sum_{m=1}^M \frac{m^3}{\frac{(2M+1)^4}{\pi^4} \sin^4(\frac{m \pi}{2M+1})}.
    \end{eqnarray*}
    With the other side of Jordan's inequality, we have
     \begin{eqnarray*}
        \frac{m^3}{\frac{(2M+1)^4}{\pi^4} \sin^4(\frac{m \pi}{2M+1})}
        \leq \frac{m^3}{m^4(\frac{2}{\pi})^4}=\frac{\pi^4}{2^4 m },
    \end{eqnarray*}
    and hence 
    \begin{eqnarray*}
         \sum_{i\in Z_N}\frac{i_1^2}{\lambda_i^2}
         \leq \frac{8\pi^4}{2^4} \sum_{m=1}^M \frac{1}{m}.
    \end{eqnarray*}
    Since  
    \begin{eqnarray*}
        \frac{1}{(2M+1)^2}\sum_{m=1}^M \frac{1}{m}\to 0 \quad \text{ as } M\to \infty,
    \end{eqnarray*}
    we obtain the claim.
\end{proof}

%% file: references.bib
@inproceedings{arnold1966geometrie,
  title={Sur la g{\'e}om{\'e}trie diff{\'e}rentielle des groupes de Lie de dimension infinie et ses applications {\`a} l'hydrodynamique des fluides parfaits},
  author={Arnold, Vladimir},
  booktitle={Annales de l'institut Fourier},
  volume={16},
  number={1},
  pages={319--361},
  year={1966}
}

@book{Arnol-khesin,
  title={Topological methods in hydrodynamics},
  author={Arnold, Vladimir Igorevich and Khesin, Boris A},
  volume={19},
  year={2021},
  publisher={Springer}
}

@article{cruzeiro2008nonergodicity,
  title={Nonergodicity of Euler fluid dynamics on tori versus positivity of the Arnold--Ricci tensor},
  author={Cruzeiro, Ana-Bela and Malliavin, Paul},
  journal={Journal of Functional Analysis},
  volume={254},
  number={7},
  pages={1903--1925},
  year={2008},
  publisher={Elsevier}
}

@article{dowker1992finite,
  title={Finite model of two-dimensional ideal hydrodynamics},
  author={Dowker, John S.  and Wolski, Andrzej},
  journal={Physical Review A},
  volume={46},
  number={10},
  pages={6417},
  year={1992},
  publisher={APS}
}

@article{drivas2022conjugate,
  title={Conjugate and cut points in ideal fluid motion},
  author={Drivas, Theodore D and Misio{\l}ek, Gerard and Shi, Bin and Yoneda, Tsuyoshi},
  journal={Annales math{\'e}matiques du Qu{\'e}bec},
  volume={46},
  number={1},
  pages={207--225},
  year={2022},
  publisher={Springer}
}

@article{elgindi2023remark,
  title={Remark on the Stability of Energy Maximizers for the {2D Euler} equation on {$\mathbb T^2$}},
  author={Elgindi, Tarek M},
  journal={arXiv preprint arXiv:2307.12290},
  year={2023},
}

@article{gallagher2002mathematical,
  title={Mathematical analysis of a structure-preserving approximation of the bidimensional vorticity equation},
  author={Gallagher, Isabelle},
  journal={Numerische Mathematik},
  volume={91},
  number={2},
  pages={223--236},
  year={2002},
  publisher={Springer}
}

@article{hoppe-yau1998some,
  title={Some properties of matrix harmonics on {$\mathbb S^2$}},
  author={Hoppe, Jens and Yau, Shing-Tung},
  journal={Communications in mathematical physics},
  volume={195},
  number={1},
  pages={67--77},
  year={1998},
  publisher={Springer}
}

@article{lee2021nonpositive,
  title={Nonpositive curvature of the quantomorphism group and quasigeostrophic motion},
  author={Lee, Jae Min and Preston, Stephen C.},
  journal={Differential Geometry and its Applications},
  volume={74},
  pages={101698},
  year={2021},
  publisher={Elsevier}
}

@article{Wir-Shep,
  title={Nonlinear stability of Euler flows in two-dimensional periodic domains},
  author={Wirosoetisno, Djoko and Shepherd, Theodore G},
  journal={Geophysical \& Astrophysical Fluid Dynamics},
  volume={90},
  number={3-4},
  pages={229--246},
  year={1999},
  publisher={Taylor \& Francis}
}

@article{le2024conjugate,
  title={Conjugate Points Along Kolmogorov Flows on the Torus},
  author={Le Brigant, Alice and Preston, Stephen C},
  journal={Journal of Mathematical Fluid Mechanics},
  volume={26},
  number={2},
  pages={24},
  year={2024},
  publisher={Springer}
}

@article{lichtenfelz2025ricciZeitlin,
  title={Ricci Curvature for Hydrodynamics on the Sphere},
  author={Lichtenfelz, Leandro and Modin, Klas and Preston, Stephen C},
  journal={Communications in Mathematical Physics},
  volume={407},
  number={2},
  pages={37},
  year={2026},
  publisher={Springer}
}

@article{Lichtenfelz2026personal_communication,
  title= {The Large-{$N$} Limit of {R}icci curvature in the {Z}eitlin Model on the Torus},
  author       = {Lichtenfelz, Leandro and Raad, Isabelle and  Valletta, Justin},
  year         = {2026},
  journal         = {In preparation}
}

@article{lukatskii1979curvature,
  title={Curvature of groups of diffeomorphisms preserving the measure of the 2-sphere},
  author={Lukatskii, Aleksandr M.},
  journal={Functional Analysis and Its Applications},
  volume={13},
  number={3},
  pages={174--177},
  year={1979},
  publisher={Springer}
}

@article{lukatskii1984curvature,
  title={Curvature of the group of measure-preserving diffeomorphisms of the {$n$}-dimensional torus},
  author={Lukatskii, Aleksandr M.},
  journal={Siberian Mathematical Journal},
  volume={25},
  number={6},
  pages={893--903},
  year={1984},
  publisher={Springer}
}

@article{modin2024two,
  title={Two-Dimensional Fluids Via Matrix Hydrodynamics},
  author={Modin, Klas and Viviani, Milo},
  journal={Archive for Rational Mechanics and Analysis},
  volume={250},
  number={1},
  pages={10},
  year={2026},
  publisher={Springer}
}

@article{shkoller2000analysis,
  title={Analysis on groups of diffeomorphisms of manifolds with boundary and the averaged motion of a fluid},
  author={Shkoller, Steve},
  journal={Journal of differential geometry},
  volume={55},
  number={1},
  pages={145--191},
  year={2000},
  publisher={Lehigh University}
}

@article{suri2025stochastic,
  title={Stochastic Euler-Poincar{\'e} reduction for central extension},
  author={Suri, Ali},
  journal={Differential Geometry and its Applications},
  volume={101},
  pages={102290},
  year={2025},
  publisher={Elsevier}
}

@article{suri2024conjugate,
  title={Conjugate points along spherical harmonics},
  author={Suri, Ali},
  journal={Journal of Geometry and Physics},
  volume={206},
  pages={105333},
  year={2024},
  publisher={Elsevier}
}

@article{suri2024curvature,
  title={Curvature and stability of quasi-geostrophic motion},
  author={Suri, Ali},
  journal={Journal of Geometry and Physics},
  volume={198},
  pages={105109},
  year={2024},
  publisher={Elsevier}
}

@article{vizman2001geodesics,
  title={Geodesics on extensions of {Lie} groups and stability: the superconductivity equation},
  author={Vizman, Cornelia},
  journal={Physics Letters A},
  volume={284},
  number={1},
  pages={23--30},
  year={2001},
  publisher={Elsevier}
}

@article{vizman2008cocycles,
  title={Cocycles and stream functions in quasigeostrophic motion},
  author={Vizman, Cornelia},
  journal={Journal of Nonlinear Mathematical Physics},
  volume={15},
  number={2},
  pages={140--146},
  year={2008},
  publisher={Springer}
}

@article{yoshida1997riemannian,
  title={Riemannian curvature on the group of area-preserving diffeomorphisms (motions of fluid) of 2-sphere},
  author={Yoshida, Kyo},
  journal={Physica D: Nonlinear Phenomena},
  volume={100},
  number={3-4},
  pages={377--389},
  year={1997},
  publisher={Elsevier}
}

@article{zeitlin1991torus,
title = {Finite-mode analogs of {2D} ideal hydrodynamics: Coadjoint orbits and local canonical structure},
journal = {Physica D: Nonlinear Phenomena},
volume = {49},
number = {3},
pages = {353-362},
year = {1991},
issn = {0167-2789},
doi = {https://doi.org/10.1016/0167-2789(91)90152-Y},
url = {https://www.sciencedirect.com/science/article/pii/016727899190152Y},
author = { Zeitlin, Vladimir},
}

@article{zeitlinMHD,
  title={On self-consistent finite-mode approximations in (quasi-) two-dimensional hydrodynamics and magnetohydrodynamics},
  author={Zeitlin, Vladimir},
  journal={Physics Letters A},
  volume={339},
  number={3-5},
  pages={316--324},
  year={2005},
  publisher={Elsevier}
}

@book{leveque2007finite,
  title={Finite difference methods for ordinary and partial differential equations: steady-state and time-dependent problems},
  author={LeVeque, Randall J},
  year={2007},
  publisher={SIAM}
}

@article{Dullin2016instability,
author = {Dullin, Holger R. and Marangell, Robert and Worthington, Joachim},
title = {Instability of Equilibria for the Two-Dimensional Euler Equations on the Torus},
journal = {SIAM Journal on Applied Mathematics},
volume = {76},
number = {4},
pages = {1446-1470},
year = {2016},
doi = {10.1137/15M1043054},
URL = { 
       https://doi.org/10.1137/15M1043054
},
eprint = { 
    
        https://doi.org/10.1137/15M1043054
}
}

@article{modinSuri2026geodesic,
  title={Geodesic interpretation of the global quasi-geostrophic equations},
  author={Modin, Klas and Suri, Ali},
  journal={Calculus of Variations and Partial Differential Equations},
  volume={65},
  number={1},
  pages={17},
  year={2026},
  publisher={Springer}
}

@article{fjortoft1953changes,
  title={On the changes in the spectral distribution of kinetic energy for twodimensional, nondivergent flow},
  author={Fj{\o}rtoft, Ragnar},
  journal={Tellus},
  volume={5},
  number={3},
  pages={225--230},
  year={1953},
  publisher={Wiley Online Library}
}

@article{roy2026vakonomic,
  title={Vakonomic Fluids},
  author={Roy-Chowdhury, Ritoban and Nabizadeh, Mohammad Sina and Gross, Oliver and Gruber, Anthony and Chern, Albert},
  journal={arXiv preprint arXiv:2607.18312},
  year={2026}
}

@article{haller2002totally,
  title={Totally geodesic subgroups of diffeomorphisms},
  author={Haller, Stefan and Teichmann, Josef and Vizman, Cornelia},
  journal={Journal of Geometry and Physics},
  volume={42},
  number={4},
  pages={342--354},
  year={2002},
  publisher={Elsevier}
}
